\documentclass[12pt]{article}
\usepackage{epsfig}
\usepackage{amsfonts}
\usepackage{amssymb}
\usepackage{amsmath}
\usepackage{amsthm}
\usepackage{latexsym}
\usepackage{graphicx}
\usepackage{bm}
\usepackage{tabularx}
\usepackage{tikz}
\usepackage{booktabs}
\usepackage{enumerate}
\usepackage{caption}
\usepackage[titletoc]{appendix}
\usepackage[numbers]{natbib}
\usepackage{float}
\usepackage[tight,footnotesize]{subfigure}
\usepackage[margin=1in]{geometry}
\usepackage{mwe}

\catcode`\@=11

\newcommand{\updot}[1]{\raisebox{0.9pt}{$\stackrel{\bullet}{#1}$}} 

\newcommand{\E}[1]{{\mathrm{E}\left[#1\right]}}

\newcommand{\Var}[1]{{\mathrm{Var}\left(#1\right)}}
\newcommand{\Cov}[1]{{\mathrm{Cov}\left[#1\right]}}
\newcommand{\T}{{\mathrm{T}}}
\newcommand{\diag}[1]{{\mathrm{diag}\left(#1\right)}}
\newcommand{\onevec}[1]{{{\mathbf{v}}_{\hspace{-.01in}#1}\hspace{-.01in}}}
\newcommand{\edit}[1]{{#1}}
\newcommand{\minrev}[1]{{#1}}

\usepackage[pdftitle={Hawkes},
  pdfauthor={DawPender},
  pdffitwindow=true,
  breaklinks=true,
  colorlinks=true,
  urlcolor=blue,
  linkcolor=red,
  citecolor=red,
  anchorcolor=red]{hyperref}

\numberwithin{equation}{section}
\numberwithin{table}{section}
\numberwithin{figure}{section}

\theoremstyle{plain}
\newtheorem{theorem}{Theorem}[section]
\newtheorem{lemma}[theorem]{Lemma}
\newtheorem{corollary}[theorem]{Corollary}
\newtheorem{proposition}[theorem]{Proposition}

\theoremstyle{definition}
\newtheorem{definition}{Definition}[section]

\newtheorem{problem}[definition]{Problem}

\theoremstyle{remark}
\newtheorem*{remark}{Remark}

\begin{document}

\title{Queues Driven by Hawkes Processes}
\author{
Andrew Daw \\ School of Operations Research and Information Engineering \\ Cornell University
\\ 257 Rhodes Hall, Ithaca, NY 14853 \\  amd399@cornell.edu \\
 \and
  Jamol Pender \\ School of Operations Research and Information Engineering \\ Cornell University
\\ 228 Rhodes Hall, Ithaca, NY 14853 \\  jjp274@cornell.edu
 }

\maketitle
\begin{abstract}
Many stochastic systems have arrival processes that exhibit clustering behavior. In these systems, arriving entities influence additional arrivals to occur through self-excitation of the arrival process. In this paper, we analyze an infinite server queueing system in which the arrivals are driven by the self-exciting Hawkes process and where service follows a phase-type distribution or is deterministic.  In the phase-type setting, we derive differential equations for the moments and a partial differential equation for the moment generating function; we also derive exact expressions for the transient and steady-state mean, variance, and covariances.  Furthermore, we also derive exact expressions for the auto-covariance of the queue and provide an expression for the cumulant moment generating function in terms of a single ordinary differential equation. In the deterministic service setting, we provide exact expressions for the first and second moments and the queue auto-covariance. As motivation for our Hawkes queueing model, we demonstrate its usefulness through two novel applications.  These applications are trending internet traffic and arrivals to nightclubs. In the web traffic setting, we investigate the impact of a click. In the nightclub or \emph{Club Queue} setting, we design an optimal control problem for the \minrev{optimal} rate to admit club-goers.
\end{abstract}



\section{Introduction} \label{sec_intro}

The arrival process is a fundamental component of stochastic queueing models. In most models, these arrival processes are driven by a Poisson process, which is well suited for environments in which arrivals have no influence on one another. If the arrival process is a simple (single jump) random counting process with independent increments, \citet{prekopa1957poisson} shows that this is equivalent to a non-homogeneous  Poisson process. However, this can be unrealistic for many situations. For example, in the trading of financial assets, transactions tend to occur together as traders are often responding to the same information as their peers or to their actions \cite{azizpour2016exploring}. Additionally, earthquakes and other forms of geological tension frequently occur in quick succession, as aftershocks can continue to \edit{affect} an area soon after the initial tremors \cite{ogata1988statistical}.  Even patterns of violent crime have been known to occur in clusters, as victims may decide to retaliate \cite{mohler2011self}. In each of these examples, the occurrence of an event makes the occurrence of the next more likely to happen in quick succession, which means that the sequence of arrivals \edit{tends to form clusters}. This type of phenomena would be better modeled by variables that are not memoryless, so that the occurrences can have an influence on those that follow soon after \edit{and increments are not independent}.

One stochastic arrival process that captures clustering of arrivals was introduced in 1971 by \citet{hawkes1971spectra}, and is referred to as the Hawkes process. This stochastic \minrev{process} counts the number of arrivals \edit{and}, unlike the Poisson process, it self-excites. This means that when one arrival occurs, it increases the likelihood that another arrival will occur soon afterwards. The Hawkes process does so through treating both the counting process and the rate of arrivals as \minrev{coupled} stochastic processes. Because the arrival rate increases, it is treated in a general sense as the arrival ``intensity,'' which can be thought of as a representation of the excitement at that time. The higher the intensity, the more likely it is that an arrival will occur. In this setting, the number of arrivals and the arrival intensity \edit{represent the system together as a pair}.

Historically, the Hawkes process has been studied \edit{predominantly} in financial settings. However, it has only recently received a significant amount of attention in broader and more general contexts. For a general overview, a review of the Hawkes process was written by \citet{Laub:2015aa}. In \edit{our work}, we are particularly interested in socially informed queueing systems, and we use these systems as a motivation for both studying the Hawkes process and applying it to queueing models. For example, in situations in which a person does not know the value of competing offers or services, \minrev{she} may decide to pursue the option that has the most other people already waiting for it. When one can't be sure of what is earned by waiting, the willingness of others to wait can often be the best indicator.

As a quick example for the sake of building intuition, consider walking past a street performer. If there is only a handful of other people watching, one may not feel a desire to stop and see the performance. However, if there is a large crowd already \edit{watching} it is more enticing to join the group and see what is happening. This is the basic \edit{motivation} of self-exciting and clustering arrival processes. Although this example is simple, the concept itself has powerful implications for service systems. Several naturally occurring examples of these systems were detailed in a recent Chicago Booth Review article, \cite{cbrarticle}. These examples include cellular companies paying employees to join the lines outside stores during product launches and pastry enthusiasts waiting hours in queue to buy baked goods from the famed Dominique Ansel \minrev{Bakery in New York}. \minrev{(The article even includes a story of a German man joining a long queue in 1947 without any knowledge of what awaited him, only to find it was for visas to the United States!)} \edit{Another example of self exciting arrivals in service settings are flight cancellations, discussed in a recent Business Insider article, \cite{airlines}. Because of widespread events like inclement weather and \minrev{information technology} infrastructure failures,  flights are often cancelled in mass. However, as that article notes, even one plane experiencing mechanical issues can cause a cluster of downstream cancellations throughout its flight legs.}

In this paper, we apply our results to \edit{two} main applications:  the viral nature of modern web traffic and the appeal associated with the lengths of queues for nightclubs. In socially informed internet traffic, webpages experience arrivals of users in clusters due to the contagion-like spread of information. If one user shares a webpage, others become more likely to view and share it as well. We demonstrate this through an example from \minrev{Twitter}  data \edit{and explore the impact of a click}. \edit{T}he night club example can be seen as an effect of having to pay a cover fee up front to enter the venue. Because club-goers must pay before ever seeing inside, the number of others already in queue to enter the club gives a sense of the attraction they are awaiting. \edit{In this setting we consider the managerial control problem of how quickly to admit customers to maximize earnings}. Again, in these examples the occurrence of an event or arrival of a customer increases the likelihood that another will happen soon after.

We model these sort of settings through queueing systems in which the arrivals occur according to a Hawkes process and in which service times follow phase-type distributions. This general type of service allows for accurate and robust modeling while preserving key characteristics for queues, such as the Markov property. Mathematically, this work is most similar to recent work by \citet{gao2016functional} and \citet{koops2017infinite}.  \edit{Moreover, transient moments for infinite server queues with Markovian arrivals are also among the findings in \citet{koops2017infinite}, an independent and concurrent work. However the moments in \citet{koops2017infinite} are only derived for exponential service distributions, whereas we give expressions for any phase-type service distribution.  Additionally, we analyze the Hawkes/D/$\infty$ queue and give an explicit analysis for its first two moments.}  Conceptually, our motivation is most similar to  \citet{debo2012signaling}. While the model in \cite{debo2012signaling} is similar to this one in concept, it is quite different in its probabilistic structure. Rather than using a Hawkes process for the arrivals, the authors model the scenario through a Poisson process with a probability of arrivals joining or balking that increases with the length of the queue. This describes the setting well, but there are a few limitations and room for additional considerations. For example, recency plays no role in the influence of the next arrival. For queues of identical length, that model considers the most recent arrival occurring a minute ago to be equivalent to it occurring an hour ago. Additionally, because events arrive according \edit{to} a time-homogeneous Poisson process and then either join or balk, the rate at which arrivals join the queue is bounded by the overall arrival rate, a constant. This prevents any kind of ``viral'' behavior for the events, so a large influx of arrivals over a short time is unlikely to occur. By contrast, these behaviors are inherent to our model. We will explore these ideas and others after the following descriptions of this paper's composition.

 \subsection{Main Contributions of Paper}

In this paper, we provide exact \edit{expressions for the mean, variance, and covariance of the} Hawkes process driven queue for all time, in both transient and steady state. These moments are derived for general \minrev{phase-type service}; we also provide examples for hyper-exponential and Erlang service. These results are \edit{derived by exploiting} linear ordinary differential equations.   We also derive expressions for all moments \minrev{of the} queue. We verify these functions via comparisons to simulations. \edit{We also derive a partial differential equation for the moment generating function and the cumulant moment generating function for the Hawkes/PH/$\infty$ queue.  We are able to show that the solution of the potentially high dimensional PDE for the MGF can be reduced to solving one differential equation, which does not have a closed form expression except in some special cases.  Moreover, we analyze the Hawkes/D/$\infty$ queue where the service times are deterministic.  We derive exact expressions for the mean, variance, and auto-covariance of the queue length process.}  Throughout this work we show the relevance of the Hawkes process by direct comparison to the Poisson process and through novel applications. In our applications, we investigate the long run effects of the self-excitement structure\edit{,} design an optimal control problem\edit{,} and describe how to solve it numerically.

\subsection{Organization of Paper}

 \edit{The remainder of this paper is organized into three main sections. In Section~\ref{sec_Hawk}, we give an overview of results and properties in the Hawkes process literature that are relevant to this work and we then investigate the infinite server Hawkes process driven queue with deterministic service. In Section~\ref{sec_hawkphinf}, we perform the main analysis of this work, which is the \minrev{investigation} of infinite server queues with Hawkes process arrivals and phase-type distributed service. In doing so, we first provide model definitions and technical lemmas, then derive expressions for the moments of the queue, followed by the auto-covariance and moment and cumulant generating functions. In Section~\ref{sec_applications}, we apply this work to two novel settings, trending web traffic and night clubs.} To facilitate comprehension of subject-specific notation, we provide the following table of terminology. Listed by order of appearance, these terms are also stated \minrev{and defined} at their first use. Thus, this reference is simply intended as an aid for the reader. \edit{In particular, we \minrev{draw} attention to this paper's use of $\onevec{}$ to represent the vector of all ones. While such a vector may be more commonly denoted as $\textbf{e}$, we avoid that notation as these vectors are frequently used near matrix exponentials since $\textbf{v}$ is more distinct from $e$ than $\textbf{e}$ is.}\\

\begin{center}
\bgroup
\setlength{\extrarowheight}{2.5pt}
\begin{tabularx}{\textwidth}{ c | X }
Symbol & Definition    \\
  \hline
  $N_t$ & Hawkes counting process, the self-exciting point process \\
$\lambda_t$ & Hawkes process intensity, represents the excitement of the process at time $t$ \\
$\alpha$ & Hawkes process jump parameter, represents the jump in intensity upon an arrival \\
$\beta$ & Hawkes process decay parameter, governs the exponential decrease of $\lambda_t$ \\
$\lambda^*$ & Hawkes process baseline intensity \\
$\lambda_0$ & Initial value of $\lambda_t$ \\
$\lambda_\infty$ & Equal to $\frac{\beta \lambda^*}{\beta - \alpha}$, represents the limit of the mean intensity as $t \to \infty$ \\
$Q_t$ & Queueing system, where $Q_{t,i}$ is the number in phase $i$ of service at time $t$ \\
$S$ & Phase-type distribution transient state sub-generator matrix, represents the exponentially distributed rate of transitions of an entity from one phase of service to another with state 0 designated as the absorbing state for the end of the entity's service. Off diagonal elements are $\mu_{ij}$ and diagonal elements are $-\mu_i$ \\
$\mu_{ij}$ & Transition rate from phase $i$ to phase $j$ where $i \ne j$ \\
$\mu_i$  & Overall transition rate out of phase $i$ \\
$\theta$ & Queueing system initial distribution of arrivals over the $n$ phases of service \\
$\textbf{v}$ & The $n$-dimensional vector of all ones \\
$\textbf{v}_i$ & The $n$-dimensional vector of all zeros other than \edit{the value 1} at the $i^\text{th}$ element \\
 $\textbf{V}_i$ & The $n \times n$ matrix with one at $(i,i)$ and zero otherwise \\
\end{tabularx}
\egroup
\end{center}

\section{Hawkes Arrival Process} \label{sec_Hawk}

\edit{
The Hawkes process, introduced in \cite{hawkes1971spectra},} is a self-exciting point process whose arrival intensity is dependent on the point process sample path.  This is defined through the following dependence on the intensity process $\lambda_t$:
\begin{align}
\mathbb{P}( N_{t+h} - N_t = 1 | \mathcal{F}_t ) &= \lambda_t \cdot h + o(h) \\
\mathbb{P}( N_{t+h} - N_t > 1 | \mathcal{F}_t ) &=  o(h) \\
\mathbb{P}( N_{t+h} - N_t = 0 | \mathcal{F}_t ) &= 1- \lambda_t \cdot h + o(h)
\end{align}
where $\mathcal{F}_t$ is a filtration on the underlying probability space $(\Omega, \mathcal{F}, \mathbb{P})$ generated by $(N_t)_{t\geq0}$.  The arrival intensity is governed by the following stochastic dynamics:
\begin{equation}\label{hawkdyn}
d\lambda_t = \beta  ( \lambda^* - \lambda_t) dt + \alpha  dN_t .
\end{equation}

\edit{
Here} $\lambda^*$ represents \edit{an underlying stationary arrival rate called} the baseline intensity\edit{,} $\alpha > 0$ is the height of the jump in the intensity upon an arrival\edit{, and} $\beta > 0$ describes the decay of the intensity as time passes after an arrival. That is, when the number of arrivals $N_t$ increases  by one, the arrival intensity will jump \minrev{up} by amount $\alpha$, and this increases the probability of another jump occurring.  This is why the Hawkes process is called self-exciting\edit{:} its prior activity increases the likelihood of its future activity.  However, as soon as an arrival occurs the intensity begins to decay exponentially \edit{at} rate $\beta$ to the baseline intensity $\lambda^*$. Because of the jumps and the decay, the arrivals tend to cluster. If one applies Ito's lemma to the kernel function $e^{-\beta t} \lambda_t$, then one can show that
\begin{equation}\label{exphawk}
\lambda_t = \lambda^* + e^{-\beta t}  ( \lambda_0 - \lambda^*)  + \alpha \int^{t}_{0} e^{-\beta (t -s)} dN_s ,
\end{equation}
\noindent
\edit{as in \cite{da2014hawkes}, which also discusses the impact of the initial value of the intensity $\lambda_0$}. This process is known to be stable for $\alpha < \beta$, see \cite{Laub:2015aa}. \edit{Additionally, it is Markovian when conditioned on the present value of both the counting process and the intensity, which is also given in \cite{Laub:2015aa}.} For the rest of this study we will restrict our setting to this exponential kernel assumption. When we use the term ``Hawkes process'' we assume that it has such a kernel. Before proceeding with \edit{a review of relevant Hawkes process results from the literature, we motivate the use of this process by showing both its similarities and its differences} with the Poisson process.

 \subsection{Comparison to the Poisson Process}\label{hawkpoissoncomp}
In Equation~\ref{exphawk}, note that if $\alpha = 0$ and $\lambda_0 = \lambda^*$ then $\lambda_t = \lambda^*$ for all $t$. In this case, the Hawkes process is equivalent to a stationary Poisson process with rate $\lambda^*$.  However, if $\alpha = 0$ but $\lambda_0 \ne \lambda^*$ it is equivalent to a non-stationary Poisson process. So, conceptually, a Poisson process is a Hawkes process without excitement.  Furthermore, a Hawkes process with $\lambda_0 = \lambda^*$ is in essence a stationary Poisson process until the first arrival occurs. However, once an arrival occurs the intensity process jumps by an amount $\alpha$ from the initial value and then begins to decay towards the baseline rate according to the exponential decay rate $\beta$. This is demonstrated in the example in Figure~\ref{lambdaex} below. \edit{This simulation, in addition to all the others throughout this work, is constructed by use of the algorithm described in \citet{ogata1981lewis}}.

  \vspace{-.25cm}
 \begin{figure}[H]
\begin{center}	
\includegraphics[scale=.25]{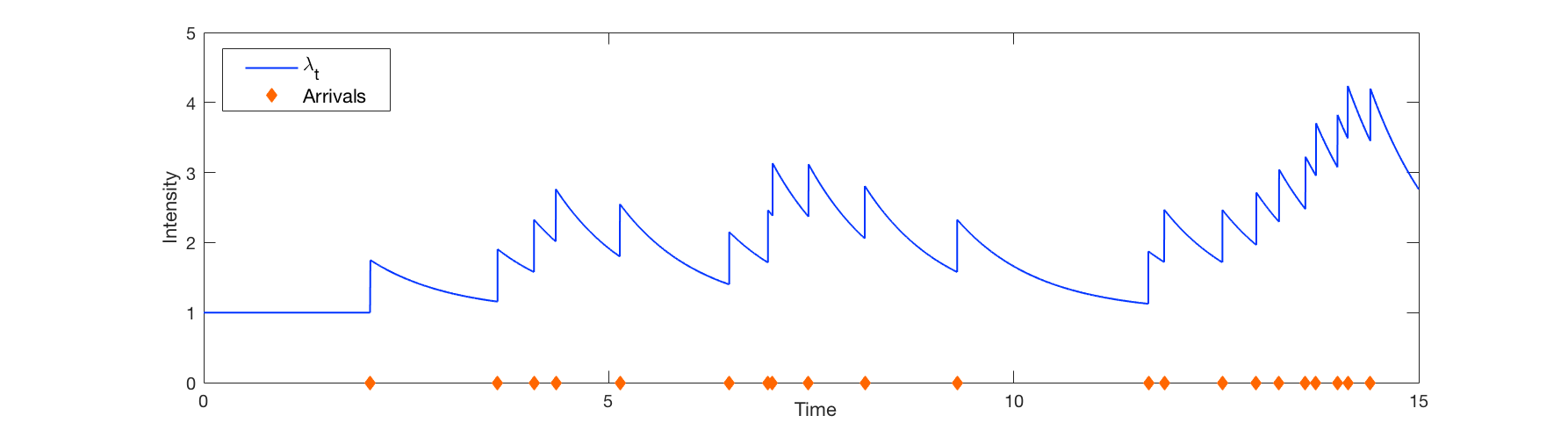}
\vspace{-.2cm}
\caption{Simulated ${\lambda_t}$, where ${\alpha = \frac 3 4}$, ${\beta = 1}$, and $\lambda^* = 1$.} \label{lambdaex}
\end{center}
\end{figure}
\vspace{-.5cm}

This example also shows another key difference between the Hawkes and Poisson processes. Because the self-excitation increases the likelihood of an arrival occurring soon after another, the Hawkes process tends to cluster arrivals together across time. This means that the variance of a Hawkes process will be larger than that of \edit{a} Poisson process, which is known to be equal to its mean. Below we demonstrate this through simulated limit distributions of the Hawkes process compared with the known Poisson probability mass function (PMF), each with the same mean.

 \vspace{-.1in}
 \begin{figure}[H]
\begin{center}	
\hspace{-.35in}
    \begin{minipage}{0.475\textwidth}
    \begin{center}
\includegraphics[scale=.24]{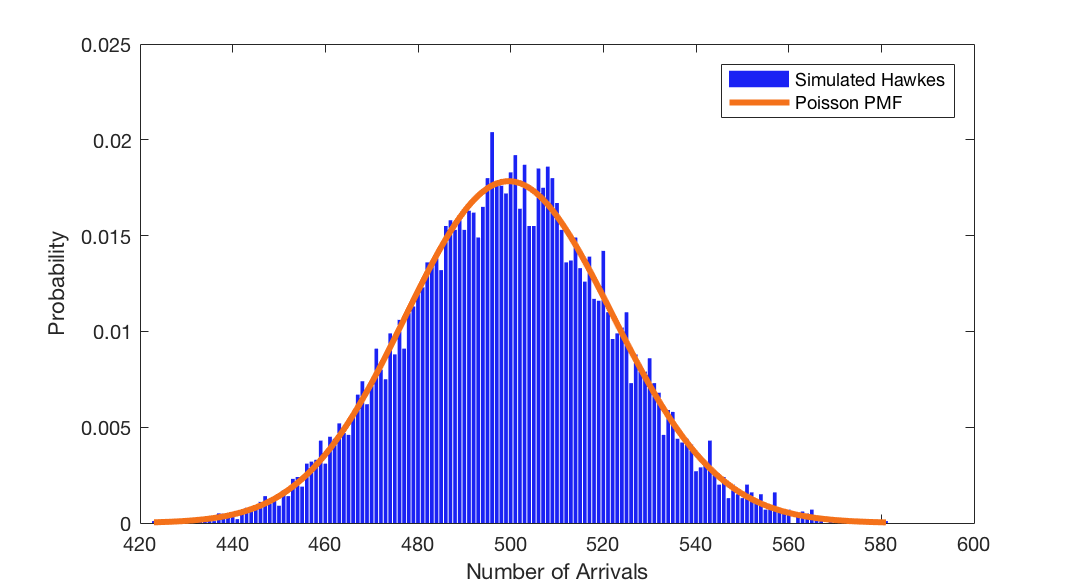}
	\end{center}
    \end{minipage}\hspace{.2in}
    \begin{minipage}{0.475\textwidth}
    \begin{center}
    \includegraphics[scale=.24]{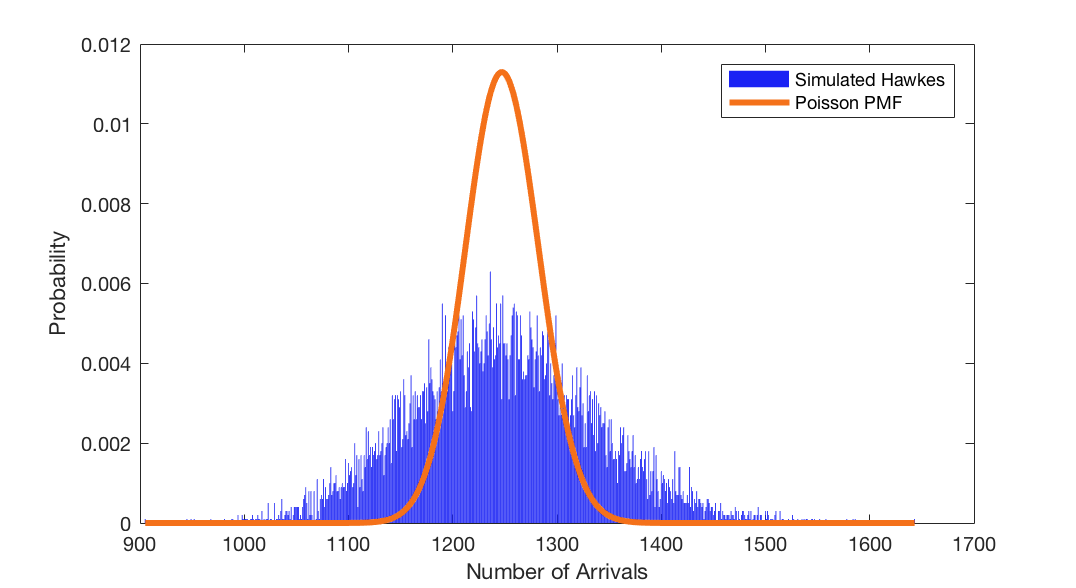}
    \end{center}
    \end{minipage}
\caption{Limit Distributions for ${\lambda^* = \beta = 1}$ and ${\alpha = 0}$ (left) and ${0.6}$ (right).}
\label{hawkesvpoisson}
\end{center}
\vspace{-.25in}
\end{figure}

The simulated results are based on 10,000 replications, each with an end time of 500. As described previously, the two processes are equivalent for $\alpha = 0$. However, as $\alpha$ increases, the similarity between the Hawkes process and the Poisson process starts to disappear. Through these examples, we observe that the Hawkes process behaves quite differently from the Poisson process since it has \edit{heavier} tails and therefore, is more variable.  Thus, this provides theoretical motivation for our following investigation.

 \subsection{\edit{Review of Relevant Hawkes Process Literature}}\label{hawkmeandyn}
 \edit{
We now review a brief selection of Hawkes process results that support our following analysis of Hawkes process driven queueing systems. These results can be found in greater detail in \citet{dassios2011dynamic, da2014hawkes, da2015clustering}, as discussed specifically after each result statement. This review is primarily focused on the transient and stationary moments of the Hawkes process, and is included both for the sake of completeness and understanding of the problem, but also so that it may be incorporated later in this work. In the first statement, Proposition~\ref{hawkgeneral}, differential equations for the moments of the Hawkes process are provided.
}
\begin{proposition}\label{hawkgeneral}
Given a Hawkes process $X_t = (\lambda_t, N_t)$ with dynamics given by Equation \ref{hawkdyn}, then we have the following differential equations for the moments of $N_t$ and $\lambda_t$,
\begin{align}
\frac{\mathrm{d}}{\mathrm{d}t}\E{N_t^m} &= \sum^{m-1}_{j=0} { m \choose j }  \E{ \lambda_t  N_t^j  } \\
\frac{\mathrm{d}}{\mathrm{d}t}\E{\lambda_t^m} &= m  \beta  \lambda^*   \E{\lambda_t^{m-1}}
-
m \beta  \E{\lambda_t^{m}}  + \sum^{m-1}_{j=0} { m \choose j }  \alpha^{m-j}  \E{ \lambda_t^{j+1} } \\
\frac{\mathrm{d}}{\mathrm{d}t}\E{\lambda_t^m  N_t^l } &= m  \beta  \lambda^*  \E{\lambda_t^{m-1}  N_t^l} - m  \beta  \E{\lambda_t^{m}  N_t^l}   + \sum_{(j,k) \in S} { m \choose j }  { l \choose k }  \alpha^{m-j}  \E{ \lambda_t^{j+1}  N_t^k}
\end{align}
where $S = (\{0, \dots, m\}\times \{0, \dots, l\})\setminus \{(m,l)\}$.
\end{proposition}

\begin{proof}
\edit{
This follows directly from the approach involving the infinitesimal generator described in Sections 2.1 and 2.2 of \cite{da2014hawkes}, followed by simplification using the binomial theorem. For the first and second moments of $N_t$ and $\lambda_t$ and the first product moment, these equations are stated exactly \edit{in that work}.
}
\end{proof}

\edit{
As has been observed in the literature, the differential equations for the moments form a system of linear ordinary differential equations that have explicit solutions. We now provide the exact dynamics of the first two moments of the Hawkes process since this is of particular relevance to our later analysis. \edit{We also} define notation that will be used throughout the remainder of this work.
}

\begin{proposition}\label{hawksolveprop}
Given a Hawkes process $X_t = (\lambda_t, N_t)$ with dynamics given by Equation \ref{hawkdyn} \edit{with $\alpha < \beta$}, then the mean, variance, and covariance of $N_t$ and $\lambda_t$ are provided by the following equations for all $t \geq 0$,
\begin{align}
\E{\lambda_t} &=
\lambda_\infty + \left(\lambda_0 - \lambda_\infty\right) e^{-(\beta - \alpha)t}
\\
\E{N_t} &=
\lambda_\infty t + \frac{\lambda_0 - \lambda_\infty}{\beta - \alpha} \left(1 - e^{-(\beta - \alpha)t}\right)
\\
\Var{\lambda_t} &=
\frac{\alpha^2\lambda_\infty}{2(\beta - \alpha)} + \frac{\alpha^2(\lambda_0 - \lambda_\infty)}{\beta - \alpha}e^{-(\beta - \alpha)t} - \frac{\alpha^2(2\lambda_0 - \lambda_\infty)}{2(\beta - \alpha)}e^{-2(\beta - \alpha)t}
\\
\Var{N_t} &
=
\frac{\beta^2 \lambda_\infty}{(\beta - \alpha)^2}
t
+
\frac{\alpha^2(2\lambda_0 - \lambda_\infty)}{2(\beta - \alpha)^3}
\left(
1 - e^{-2(\beta - \alpha)t}
\right)
-
\frac{2\alpha\beta(\lambda_0 - \lambda_\infty)}{(\beta - \alpha)^2} te^{-(\beta - \alpha)t}
\nonumber
\\
&
\quad
+
\left(
\frac{\beta + \alpha}{(\beta - \alpha)^2}(\lambda_0 - \lambda_\infty)
-
\frac{2\alpha\beta}{(\beta - \alpha)^3}\lambda_\infty
\right)
(1 - e^{-(\beta - \alpha)t})
\\
\Cov{\lambda_t,N_t} &=
\left(
\frac{\alpha\lambda_\infty}{\beta - \alpha} + \frac{\alpha^2\lambda_\infty}{2(\beta - \alpha)^2}
\right)
\left(
1 - e^{-(\beta - \alpha)t}
\right)
+
\frac{\alpha^2(2\lambda_0 - \lambda_\infty)}{2(\beta - \alpha)^2}
\left(
e^{-2(\beta - \alpha)t} - e^{-(\beta - \alpha)t}
\right)
\nonumber
\\
&
\quad
+
\frac{\alpha\beta (\lambda_0 - \lambda_\infty) }{\beta - \alpha}
t e^{-(\beta - \alpha)t}
\end{align}
\edit{
where
$$
\lambda_\infty = \frac{\beta \lambda^*}{\beta - \alpha}.
$$
}
\end{proposition}

\begin{proof}
\edit{The proof of this result can be found in Section 3.4 of \cite{dassios2011dynamic} (as a particular case where $\rho = 0$) and in Section 3.2 of \cite{da2015clustering}, or by solving the corresponding ODE system stated above \minrev{Proposition}~\ref{hawkgeneral}.
}
\end{proof}

By further observation of Proposition \ref{hawksolveprop} \edit{ or simply by further review of the references in this section}, the steady-state behavior of various Hawkes process statistics \edit{is also available}. \edit{These expressions are stated in} the following corollary.

\begin{corollary}\label{hawklimitcor}
Given a Hawkes process $X_t = (\lambda_t, N_t)$ with dynamics given by Equation \ref{hawkdyn} \edit{with $\alpha < \beta$}, then \edit{ the steady state values of the mean and variance of the intensity and of the covariance between the intensity and the counting process are as follows:}
\begin{eqnarray}
\lim_{t \to \infty} \E{\lambda_t} &=& \frac{\beta \lambda^*}{\beta - \alpha} = \lambda_\infty, \\
\lim_{t \to \infty} \Var{\lambda_t} &=&  \frac{\alpha^2  \lambda_\infty  }{2(\beta-\alpha)}, \\
\lim_{t \to \infty} \Cov{\lambda_t , N_t } &=& \frac{\alpha\lambda_\infty}{\beta - \alpha} + \frac{\alpha^2\lambda_\infty}{2(\beta - \alpha)^2}.
\end{eqnarray}
\end{corollary}

\edit{In Proposition~\ref{hawksolveprop} and Corollary~\ref{hawklimitcor}, we assume that $\alpha < \beta$, which is a known stability condition in the literature, as detailed in \cite{Laub:2015aa}. However, we can also consider the case where $\alpha \geq \beta$ and investigate the behavior of the system through its transient mean values. This is performed in the following corollary.}

\begin{corollary}
Given a Hawkes process $X_t = (\lambda_t, N_t)$ with dynamics given by Equation \ref{hawkdyn} with $\alpha \geq \beta$, the transient mean intensity and transient mean of the counting process for $t \geq 0$ are
\begin{align}
\E{\lambda_t}
&=
\frac{\beta \lambda^*}{\alpha - \beta}\left(e^{(\alpha - \beta)t} - 1\right)
+
\lambda_0 e^{(\alpha - \beta)t}
\\
\E{N_t}
&=
\left(\frac{\beta\lambda^*}{(\alpha - \beta)^2} + \frac{\lambda_0 }{\alpha - \beta}\right)\left(e^{(\alpha - \beta)t} - 1\right)
-
\frac{\beta \lambda^*}{\alpha - \beta}t
\end{align}
when $\alpha > \beta$, and
\begin{align}
\E{\lambda_t}
&=
\beta \lambda^* t + \lambda_0
\\
\E{N_t}
&=
\frac{\beta\lambda^*}{2}t^2 + \lambda_0 t
\end{align}
when $\alpha = \beta$.
\end{corollary}

As is stated in the stability condition, we see that the limits of these functions as $t$ goes to infinity diverge for $\alpha \geq \beta$. The effect of the relationship of $\alpha$ and $\beta$ on the system can be observed in the following graph.

\edit{
\begin{figure}[H]
	\begin{center}
\includegraphics[width=.6\textwidth]{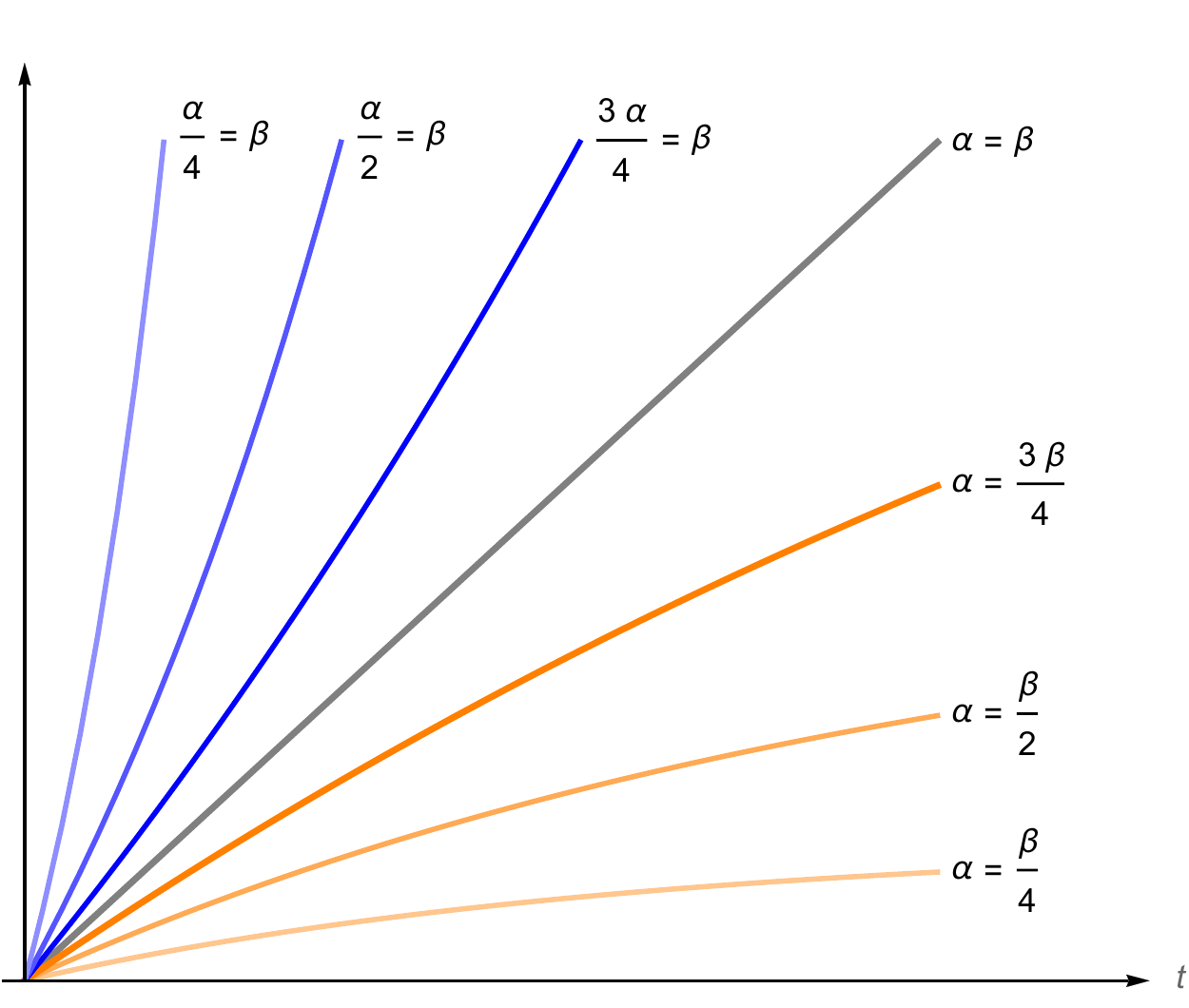}
\caption{Transient Mean Intensity for ${\alpha < \beta}$, ${\alpha = \beta}$, and ${\alpha > \beta}$.} \label{lambdaplots}
\end{center}
\end{figure}
}
\edit{For the majority of this work we will consider settings in which the arrival process is stable and so we will assume $\alpha < \beta$. However, there are settings in which the transient behavior of the unstable arrival process is of interest, and so in our analysis of the queueing system we will also explore the mean behavior of queues under such arrival conditions.
}

\subsection{$Hawkes/D/\infty$ Queue} \label{hawkesDinf}

\edit{Before moving on to the phase-type distributed service systems, we will first investigate the deterministic service setting. Since we have a good understanding about the Hawkes process itself, we can leverage our knowledge to analyze the $Hawkes/D/\infty$ queue where D is deterministic and is equal to the exact amount of time each customer spends in service.  We exploit the fact that the $Hawkes/D/\infty$ queue can be written as the difference between the Hawkes process evaluated at time $t$ and the Hawkes process evaluated at time $t-D$ i.e. }
  \begin{equation}
  Q_t = N_t - N_{t-D}.
    \end{equation}
   This representation of the $Hawkes/D/\infty$ queue leads us to a theorem that provides explicit expressions for the mean, variance, and auto-covariance of the $Hawkes/D/\infty$ queueing process.  However, before we state the result, we need a lemma that describes the transient auto-covariance of the Hawkes process.  This lemma will be extremely useful for our future calculations of other quantities of interest for the $Hawkes/D/\infty$ queue.

    \begin{lemma}\label{cdeflemma}
    Let $N_t$ be a Hawkes process with dynamics given by Equation~\ref{hawkdyn} with $\alpha < \beta$ and suppose $N_t$ is initialized at zero.  If we define $ \mathcal{C}(t, \tau)$ as the following function
     \begin{equation}
 \mathcal{C}(t, \tau) \equiv \mathrm{Cov}[ N_t, N_{t-\tau} ],
    \end{equation}
    then
     \begin{align}
     &
  \mathcal{C}(t, \tau) =
  \frac{
\alpha
\left(1 - e^{-(\beta - \alpha)  \tau }\right)
   }
   {2 (\beta -\alpha )^3}
\left(
(2 \beta -\alpha ) \lambda _{\infty}
-
2
e^{-(\beta - \alpha) (t - \tau )}
\left(
\alpha\lambda_0
+
\beta(\lambda_\infty - \lambda_0)(\beta - \alpha) (t - \tau)
+
(\beta -\alpha ) \lambda _{\infty }
\right)
\right)
\nonumber
\\
&
\quad
+
\left(
\lambda_\infty
+
\frac{2\alpha\lambda_\infty}{\beta - \alpha}
+
\frac{\alpha^2\lambda_\infty}{(\beta - \alpha)^2}
\right)
(t-\tau)
+
\frac{\alpha^2(2\lambda_0 - \lambda_\infty)}{2(\beta - \alpha)^3}
\left(
1 - e^{-(\beta - \alpha)(2t-\tau)}
\right)
-
\frac{2\alpha\beta(\lambda_0 - \lambda_\infty)}{(\beta - \alpha)^2}
\nonumber
\\
&
\quad
\cdot
(t-\tau)e^{-(\beta - \alpha)(t-\tau)}
+
\left(
\frac{\beta + \alpha}{(\beta - \alpha)^2}(\lambda_0 - \lambda_\infty)
-
\frac{2\alpha\beta}{(\beta - \alpha)^3}\lambda_\infty
\right)
(1 - e^{-(\beta - \alpha)(t-\tau)})
    \end{align}
for all $t \geq \tau \geq 0$; otherwise $\mathcal{C}(t,\tau) = 0$.
\end{lemma}
    \begin{proof}
    To see this, we manipulate the definition of the auto-covariance to find an expression in terms of other known functions. Starting from the definition of covariance, we have
    \begin{align*}
    \Cov{N_t, N_{t-\tau}}
    &=
    \E{N_t N_{t-\tau}} - \E{N_t}\E{ N_{t-\tau}}
    \end{align*}
    and by Proposition~\ref{hawksolveprop} we have expressions for $\E{N_t}$ and $\E{N_{t-\tau}}$. Thus, we focus on $\E{N_t N_{t-\tau}}$. However, for brevity's sake we do not yet substitute these known expressions into the equation. By the tower property, we have that
    \begin{align*}
     \mathcal{C}(t, \tau)  &=
    \E{\E{N_t N_{t-\tau}\mid \mathcal{F}_{t-\tau}}} - \E{N_t}\E{ N_{t-\tau}}
    \end{align*}
    where $\mathcal{F}_{t-\tau}$ is the filtration of the Hawkes process up to time $t - \tau$. Through this conditioning, $N_{t - \tau}$ is known in the inner expectation, and so we can replace $\E{\E{N_t N_{t-\tau}\mid \mathcal{F}_{t-\tau}}}$ with $\E{\E{N_t \mid \mathcal{F}_{t-\tau}}N_{t-\tau}}$. Then, again by  Proposition~\ref{hawksolveprop} we have that $\E{N_t \mid \mathcal{F}_{t-\tau}} = \lambda_\infty \tau + \frac{\lambda_{t - \tau} - \lambda_\infty}{\beta - \alpha}\left(1 - e^{-(\beta - \alpha)\tau}\right) + N_{t - \tau}  $. Making use of this, we now have that
    \begin{align*}
     \mathcal{C}(t, \tau) &=
    \lambda_\infty \tau\E{N_{t-\tau}}
    +
    \E{\lambda_{t - \tau}N_{t-\tau}}\frac{1 - e^{-(\beta - \alpha)\tau}}{\beta - \alpha}
    -
    \frac{\lambda_\infty}{\beta - \alpha}\E{N_{t - \tau}}
            \\
    &
    \quad
    \cdot
    \left(1 - e^{-(\beta - \alpha)\tau}\right)
    +
    \E{N_{t - \tau}^2}
    -
    \E{N_t}\E{ N_{t-\tau}},
    \end{align*}
    and by the definitions of covariance and variance this is equivalent to
    \begin{align*}
     \mathcal{C}(t, \tau) &=
    \lambda_\infty \tau \E{N_{t-\tau}}
    +
    \frac{\Cov{\lambda_{t-\tau},N_{t-\tau}} + \E{\lambda_{t-\tau}}\E{N_{t-\tau}}}{\beta - \alpha}\left(1-e^{-(\beta - \alpha)\tau}\right)
    \\
    &
    \quad
    -
    \frac{\lambda_\infty}{\beta - \alpha}\E{N_{t-\tau}}\left(1-e^{-(\beta - \alpha)\tau}\right)
    -
    \E{N_t}\E{N_{t-\tau}}
    +
    \Var{N_{t-\tau}}
    +
    \E{N_{t-\tau}}^2.
    \end{align*}
    Here we can recognize that each term in this expression has a known form from Proposition~\ref{hawksolveprop}. Hence, by substituting these expressions and simplifying, we achieve the stated result.
    \end{proof} 

     \vspace{-.1in}
 \begin{figure}[h]
\begin{center}	
\includegraphics[width=.6\textwidth]{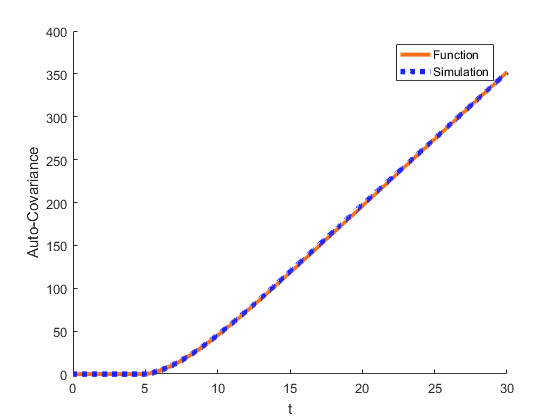}
\caption{Auto-covariance of the Hawkes Process with ${D} = 5$, ${\lambda}^* = 1$, ${\alpha} = \frac 3 4$, and ${\beta} = \frac 5 4$.}
\label{hawkesautocofig}
\end{center}
\end{figure}

    With the expression for the transient auto-covariance of the Hawkes process in hand, we can now give explicit forms of the mean, variance, and auto-covariance of the $Hawkes/D/\infty$ queue.
    \begin{theorem}\label{hawkesDinfthem}
    The transient mean of the $Hawkes/D/\infty$ when $\alpha < \beta$ is given by the following expression
      \begin{equation}
 \mathbb{E}[ Q_t ] = \begin{cases}
\lambda_\infty t + \frac{\lambda_0 - \lambda_\infty}{\beta - \alpha} \left(1 - e^{-(\beta - \alpha)t}\right)
& \text{if $t \leq D $},\\
\lambda_\infty D + \frac{\lambda_0 - \lambda_\infty}{\beta - \alpha} \left( e^{-(\beta - \alpha)(t-D)} -  e^{-(\beta - \alpha)t}\right)
& \text{if $t > D $}.
\end{cases}
    \end{equation}
    Thus, in steady state the mean queue length is
          \begin{equation}
 \mathbb{E}[ Q_{\infty} ] = \lambda_\infty D .
    \end{equation}
    Moreover, the transient variance of the $Hawkes/D/\infty$ queue is given by the following expression
        \begin{equation}
 \mathrm{Var}[ Q_t ] = \begin{cases}
 \mathcal{C}(t,0)
& \text{if $t \leq D $},\\
   \mathcal{C}(t,0) +   \mathcal{C}(t-D,0) - 2 \, \mathcal{C}(t,D)
& \text{if $t > D $}.
\end{cases}
    \end{equation}
    Lastly, the transient auto-covariance of the $Hawkes/D/\infty$ queue is given by the following expression when $\tau \geq D$,
\begin{equation}
     \mathrm{Cov}[ Q_t, Q_{t-\tau} ] =
     \begin{cases}
     0 & \text{if $t \leq \tau$,}\\
     \mathcal{C}(t,\tau) - \mathcal{C}(t-D,\tau - D) & \text{if $\tau <t \leq \tau + D$}\\
     \mathcal{C}(t,\tau) +   \mathcal{C}(t-D,\tau) - \mathcal{C}(t,\tau+D) - \mathcal{C}(t-D, \tau - D) & \text{if $\tau + D < t$}
     \end{cases}
    \end{equation}
  \noindent  and when $ \tau < D$, then
        \begin{equation}
     \mathrm{Cov}[ Q_t, Q_{t-\tau} ] =
    \begin{cases}
     0 & \text{if $t \leq \tau$,}\\
     \mathcal{C}(t,\tau) & \text{if $\tau < t \leq D$,}\\
     \mathcal{C}(t,\tau) - \mathcal{C}(t - \tau, D -\tau ) & \text{if $D <t \leq \tau + D$}\\
     \mathcal{C}(t,\tau) +   \mathcal{C}(t-D,\tau) - \mathcal{C}(t,\tau+D) - \mathcal{C}(t-\tau, D- \tau ) & \text{if $\tau + D < t$.}
     \end{cases}
     \end{equation}
    \end{theorem}

    \begin{proof} Throughout this proof we make use of the form of the auto-covariance of $N_t$ given in Lemma~\ref{cdeflemma}.
    The transient mean is straightforward since it follows from the linearity property of expectation and just taking the difference of the two means.  Moreover, for the variance we have
    \begin{eqnarray*}
     \mathrm{Var}[ Q_t ] &=& \mathrm{Var}[ N_t - N_{t-D} ] \\
     &=&  \mathrm{Var}[ N_{t} ] +  \mathrm{Var}[ N_{t-D} ] - 2  \mathrm{Cov}[ N_{t} , N_{t-D} ]  \\
     &=&  \mathrm{Var}[ N_{t} ] +  \mathrm{Var}[ N_{t-D} ] - 2  \,\mathcal{C}(t,D) \\
     &=&   \mathcal{C}(t,0) +   \mathcal{C}(t-D,0) - 2\,  \mathcal{C}(t,D) .
    \end{eqnarray*}
    Finally for the auto-covariance, if $\tau \geq D$ we have that
      \begin{eqnarray*}
     \mathrm{Cov}[ Q_t, Q_{t-\tau} ] &=&
     \begin{cases}
     0 & \text{if $t \leq \tau$,}\\
     \Cov{N_t - N_{t-D}, N_{t-\tau}} & \text{if $\tau <t \leq \tau + D$}\\
     \Cov{N_t - N_{t-D}, N_{t-\tau}-N_{t-\tau-D}} & \text{if $\tau + D < t$}
     \end{cases}
    \end{eqnarray*}
    by the definition of the $Hawkes/D/\infty$ queue and from the linearity of covariance. Now, for $\tau < D$, we have that
    \begin{eqnarray*}
     \mathrm{Cov}[ Q_t, Q_{t-\tau} ] &=&
    \begin{cases}
     0 & \text{if $t \leq \tau$,}\\
     \Cov{N_t, N_{t-\tau}} & \text{if $\tau < t \leq D$,}\\
     \Cov{N_t - N_{t-D}, N_{t-\tau}} & \text{if $D <t \leq \tau + D$}\\
     \Cov{N_t - N_{t-D}, N_{t-\tau}-N_{t-\tau-D}} & \text{if $\tau + D < t$.}
     \end{cases}
     \end{eqnarray*}
     Again by the definition of the deterministic, Hawkes-driven, infinite server queue and the linearity of covariance, we achieve the stated result.
    \end{proof}

 \vspace{-.1in}
 \begin{figure}[h]
\begin{center}	
\includegraphics[width=.6\textwidth]{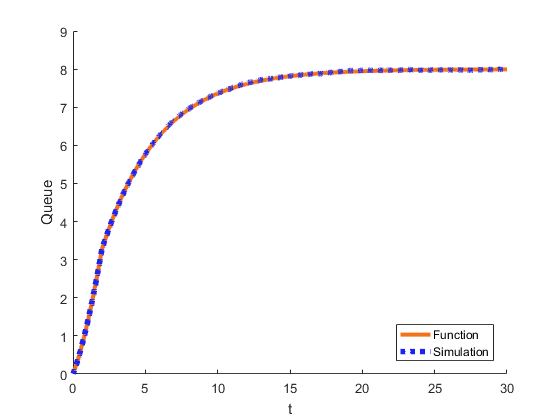}
\caption{Mean of the ${Hawkes/D/\infty}$ Queue with ${D} = 5$, ${\lambda}^* = 1$, ${\alpha} = \frac 3 4$, and ${\beta} = \frac 5 4$.}
\label{hawkesautocofig}
\end{center}
\end{figure}

\section{$Hawkes/PH/\infty$ Queue} \label{sec_hawkphinf}

\edit{
In this section, we will explore queueing systems in which arrivals occur according to a Hawkes process. This section is organized in the following manner. In Subsection~\ref{sslemmas}, we provide key model definitions such as the phase-type distribution and \minrev{we detail} technical lemmas that support our analysis. Next, in Subsection~\ref{ssmeandyn}, we derive differential equations for all moments of the queueing system and solve for exact expressions for the first and second moments. In Subsection~\ref{sslimits}, we consider the stationary limits of queues with stable arrival processes and investigate the transient behavior of those with unstable arrivals. Afterwards, we consider the auto-covariance of the queue in Subsection~\ref{ssautoco}\minrev{.} Finally, in Subsection~\ref{ssgenfunc} we derive partial differential equations for the moment generating function and the cumulant moment generating function for this system.
}

\subsection{Model Definitions and Technical Lemmas}\label{sslemmas}
\edit{To begin, we define the phase-type distribution. This form of service, formally defined below, can be thought of as a sequence of sub-services that have independent and exponentially distributed durations. We use this primarily for two \edit{factors}. The first is that this is more general than just exponential service, and it can be shown that phase-type distributions can approximate any non-negative continuous distribution, see \cite{cox1955use}. Secondly, because the phase-type distribution is comprised of independent exponential service times, a queueing system with such service distributions is Markovian. Thus, these two properties together give us a system that is both flexible in application and practical in terms of analysis.
}
A phase-type distribution with $n$ phases represents the time taken from an initial state to an absorbing state of a continuous time Markov chain (CTMC) with the following infinitesimal generator matrix,
\begin{equation*}
	\Gamma = \left[
		\begin{array}{cc}
			0 & \mathbf{0} \\
			\mathbf{s} & \mathbf{S}
		\end{array} \right]. \label{eqn_q}
\end{equation*}
Here $\mathbf{0}$ is a $1\times n$ zero vector, $\mathbf{s}$ is an $n \times 1$ vector, and $\mathbf{S}$ is an $n \times n$ matrix. Note $\mathbf{s} = -\mathbf{S}\mathbf{v}$ where $\mathbf{v}$ is an $n \times 1$ vector of ones. The matrix $\mathbf{S}$ and the initial distribution $\theta$, which is a $1 \times n$ vector, identify the phase-type distributions.
The number of phases in $\mathbf{S}$ is $n$. The matrix $\mathbf{S}$ and vector $\mathbf{s}$ can be expressed as:
\begin{align}
	\mathbf{S} &=
		\begin{bmatrix}
			-\mu_{1} & \cdots & \mu_{1,n}\\
			\vdots & \ddots & \vdots \\
			\mu_{n,1} & \cdots & -\mu_{n}
		\end{bmatrix} , \quad \mathbf{s} = (\mu_{1,0}, \ldots, \mu_{n, 0})^\mathrm{T}, \label{eqn_ss}
\end{align}
where the $\mu_{ij}$'s agree with the definition of the infinitesimal generator matrix $\mathbf{\Gamma}$.  For notational consistency, we use a term \emph{phase} to indicate the state of CTMC of the phase-type distributions throughout this paper. \edit{Additionally, we now note that in all following use of the matrix $S$ we will not use a bold notation as in those settings additional emphasis that it is a matrix is not necessary.}

With the phase-type distributions as described \edit{above}, we build a Markovian queueing model referred to as the $Hawkes/PH/\infty$ queue. We assume that the system starts with no customers and that there are infinitely many servers. Further, we suppose that there are $n$ phases of service and the transition rate between two distinct phases $i$ and $j$ is $\mu_{ij}$. Let $\theta \in [0,1]^n$ be a distribution over the phases such that the probability that an arriving entity joins the $i^\text{th}$ phase is $\theta_i$, with $\sum_{i=1}^n \theta_i = 1$. An entity departs the system at rate $\mu_{i0}$, where $i$ is the entity's phase of service before leaving. For brevity of notation, define $\mu_i \equiv \mu_{i0} + \mu_{i1} + \dots + \mu_{i,i-1} + \mu_{i,i+1} + \mu_{i,n}$. Let $Q_t \in \mathbb{N}^n$ represent the number of entities in the queueing system, with $Q_{t,i}$ representing the number in phase $i$ of service i.e.
\begin{equation}
Q_t = \sum^{n}_{i=1} Q_{t,i} \onevec{i}
\end{equation}
 \noindent where $\onevec{i}$ is the unit column vector in the $i^\text{th}$ coordinate.  We let $(\lambda_t, N_t)$ represent a Hawkes process as described in Equation~\ref{hawkdyn}. We will now  find the infinitesimal generator for real valued functions of the state space, $f : \mathbb{R}^+ \times \mathbb{N} \times \mathbb{N}^n \to \mathbb{R}$. For simplicity of notation, when describing the difference \edit{in values} of $f$ for changed arguments  we will only list the variables that change, rather than listing all $n$ queueing phase variables. This generator is shown below.
\begin{align}
&\mathcal{L}f(x)
=
\underbrace{\beta ( \lambda^* - \lambda_t)\frac{\partial f(x)}{\partial \lambda_t} }_{\edit{\text{Excitation Decay}}}
+
\underbrace{\sum_{i=1}^n \lambda_t \theta_i \left( f(\lambda_t + \alpha, N_t + 1, Q_{t,i} + 1) - f(x) \right)}_{\mathrm{Arrivals}}
\\&+ \underbrace{\sum_{i=1}^n\sum_{\substack{j = 1\\j \ne i}}^n \mu_{ij}Q_{t,i} \left(f(\lambda_t, N_t, Q_{t,i} - 1, Q_{t,j} + 1) - f(x)\right)}_{\mathrm{Transfers}} + \underbrace{\sum_{i=1}^n\mu_{i0}Q_{t,i}\left(f(\lambda_t,N_t,Q_{t,i}-1) - f(x)\right)}_{\text{Departures}} \nonumber
\end{align}
Here, $x$ is an element of the state space $\left(\mathbb{R}^+ \times \mathbb{N} \times \mathbb{N}^n\right)$. We can use this to obtain Dynkin's formula for the full $Hawkes/PH/\infty$ queueing system. We have that
\begin{equation}\label{dynQ}
\text{E}_t\left[ f(X_s)  \right] = f(X_t) + E_t\left[ \int^{s}_{t} \mathcal{L}f(X_u) du \right]  ,
\end{equation}
where $X_t = (\lambda_t, N_t, Q_t)$. This gives rise to \edit{the following lemma}.

\begin{lemma}\label{fubinifundQ}
Let $f$ be a function such that Equation~\ref{dynQ} holds. Then,
$$
\frac{\mathrm{d}}{\mathrm{d}t} \E{f(X_t)}
=
\E{\mathcal{L}f(X_t)}
$$
for all $t \geq 0$.
\end{lemma}
\begin{proof} \edit{T}his is achieved through use of Fubini's theorem and the fundamental theorem of calculus. Using Equation~\ref{dynQ} we have that
\begin{align*}
\frac{\mathrm{d}}{\mathrm{d}t}\E{f(X_t)} &= \frac{\mathrm{d}}{\mathrm{d}t}\left(f(X_0) + \E{\int_0^t \mathcal{L}f(X_u) \text {d}u}\right)
\\&= \frac{\mathrm{d}}{\mathrm{d}t}\E{\int_0^t \mathcal{L}f(X_u) \text {d}u} = \frac{\mathrm{d}}{\mathrm{d}t}\int_0^t \E{\mathcal{L}f(X_u) }\text {d}u = \E{\mathcal{L}f(X_t) }
\end{align*}
and this completes this proof.
\end{proof}

\begin{remark}
\edit{It is important for the reader to recognize that this is equivalent to Dynkin's theorem.  In most textbooks, Dynkin's theorem is proved for sufficiently differentiable and more importantly bounded functions.  However, this assumption of boundedness can often be relaxed.  In fact this relaxation of the boundedness is very common for extending results like Ito's lemma and the Feynman Kac formula for unbounded, but polynomial bounded functions.  This is often extended by stopping the process when it hits a certain level by using stopping times.  Then one applies the previous results for bounded functions and takes limits as the bound tends to infinity.  For the interested reader, see Lemma 2 of \citet{oelschlager1984martingale} for a proof. }
\\ \end{remark}

Now, before using these differential equations to find explicit functions as we did previously, we will first introduce a series of technical lemmas to aid our analysis. \minrev{These lemmas are presented without proof as they follow from standard approaches for matrix exponentials and integration.} First, we \minrev{give} a form for the indefinite integral of the exponential of a non-singular matrix.

\begin{lemma}\label{expmint}
Let $L \in \mathbb{R}^{n\times n}$ be invertible. Then, if the integral of $e^{L t}$ exists it can be expressed
$$
\int e^{L t} \, \mathrm{d}t = L^{-1}e^{Lt} + c
$$
where $c$ is some constant of integration.
\end{lemma}
\begin{proof}
\minrev{The proof follows from standard approaches.}
\end{proof} 

\minrev{The second lemma now} provides explicit forms for the definite integral from 0 to $t$ of the product of an exponential of an invertible matrix, a vector, a scalar power of the variable of integration, and a scalar exponential function of the variable of integration.

\begin{lemma}\label{singleintlemma}
Let $L \in \mathbb{R}^{n\times n}$ be invertible, let $\nu \in \mathbb{R}^n$, let $\eta \in \mathbb{N}$, and let $\gamma \in \mathbb{R}$. Then, if $L + \gamma I$ is invertible,
$$
\int_0^t e^{L  s } \nu  s ^{\eta} e^{\gamma  s } \,\mathrm{d} s
=
\sum_{k=0}^{\eta} \frac{\eta !}{(\eta - k)!}(-1)^k\left(L + \gamma I \right)^{-(k+1)}\left(e^{Lt}\nu t^{\eta - k} e^{\gamma t} \right) -  \eta ! (-1)^{\eta}\left(L + \gamma I \right)^{-(\eta+1)}\nu
$$
for $t > 0$.
\end{lemma}
\begin{proof}
\minrev{The proof follows from the preceding lemma, induction, and integration by parts.}
\end{proof} 

\minrev{The} next lemma is a quick demonstration of commutativity of the inverse of a matrix exponential and an inverse of the same matrix shifted in the direction of the identity.

\begin{lemma}\label{commutelemma}
Let $A \in \mathbb{R}^{n \times n}$ be invertible and let $b, c \in \mathbb{R}$ be such that $cA + bI$ is also invertible. Then,
$$
e^{-A}\left(cA + bI\right)^{-1} = \left(cA + bI\right)^{-1}e^{-A} .
$$
\end{lemma}
\begin{proof}
\minrev{The proof follows from the definition of the matrix exponential.}
\end{proof}

\edit{These lemmas now come together to give the general solution to differential equations of a certain form.}

\begin{lemma}\label{mainlemma}
Let $g(t) \in \mathbb{R}^n$ be a function described by the dynamics
$$
\updot{g}(t) = - L g(t) + \sum_{i \in \mathcal{S}}  \nu_i t^{\eta_i} e^{\gamma_i t}
$$
with an initial condition of $g(0) = g_0$, where $L \in \mathbb{R}^{n \times n}$ is invertible and $\mathcal{S}$ is a finite index set such that $\nu_i \in \mathbb{R}^n$, $\eta_i \in \mathbb{N}$, and $\gamma_i \in \mathbb{R}$ for each $i \in \mathcal{S}$. Then, if $L + \gamma_i I$ is invertible for all $i \in \mathcal{S}$ the explicit function for $g(t)$ is given by
\begin{align*}
g(t) &=
\sum_{i \in \mathcal{S}}\sum_{k=0}^{\eta_i} \frac{\eta_i !(-1)^k}{(\eta_i - k)!}\left(L  + \gamma_i I \right)^{-(k+1)}
\left(\nu_i t^{\eta_i - k} e^{\gamma_i t} \right)
-
\eta_i ! (-1)^{\eta_i}\left(L + \gamma_i I \right)^{-(\eta_i+1)}e^{-L t}\nu_i
+
e^{-L t} g_0
\end{align*}
for all $t \geq 0$.
\end{lemma}
\begin{proof}
\minrev{The proof follows from standard differential equation techniques and the three preceding lemmas.}
\end{proof}


Now, before introducing one final lemma we first define a useful matrix.
For $\gamma, c \in \mathbb{R}$, $\nu \in \mathbb{R}^n$, and $L \in \mathbb{R}^{n \times n}$, let $M_{\gamma, \nu, L}(t) \in \mathbb{R}^{n \times n}$ be such that
\begin{equation}\label{mdef}
M_{\gamma, \nu, L}(t)
=
\int_0^t
e^{(\gamma I - L^\T)s}
\nu\nu^\T  e^{-L s}
\,\mathrm{d}s
\end{equation}
for all $t \geq 0$. Element-wise, we can express this matrix after integration as
\begin{equation*}\label{meldef}
\left(M_{\gamma, \nu, L}(t)\right)_{i,j}
  =
\begin{cases}
\sum_{k=1}^n\sum_{l=1}^n \nu_k\nu_l  \sum_{r=0}^\infty \sum_{w = 0}^\infty \frac{(L^r)_{k,i}(L^w)_{l,j}}{\gamma^{r+w+1}} {r + w \choose r } \left(e^{\gamma t} \sum_{z=0}^{r+w} \frac{(-\gamma t)^z}{z!} - 1\right)
& \text{if $\gamma \ne 0$},\\
\sum_{k=1}^n\sum_{l=1}^n \nu_k\nu_l  \sum_{r=0}^\infty \sum_{w = 0}^\infty \frac{(L^r)_{k,i}(L^w)_{l,j} t^{r+w+1}}{r!w!(r+w+1)}
& \text{if $\gamma = 0$}.
\end{cases}
\end{equation*}
This  function provides shorthand when integrating a particular function that otherwise does not produce a nice linear algebraic form. The difficulty of expressing this integral in matrix form stems from the fact that $L$ and $\nu \nu^\T$ need not commute. With defining $M_{\gamma, \nu, L}(t)$ we circumvent this issue by integrating on the element-level, but if $L$ and $\nu \nu^\T$ were to commute we could avoid this function entirely, as we will later see. For now, this definition leads us to our next lemma.

\begin{lemma}\label{doublemintgum}
Let $\eta, \gamma, c \in \mathbb{R}$, $\nu \in \mathbb{R}^n$, $L \in \mathbb{R}^{n \times n}$ be such that $L$, $\gamma I + L$, and $(\eta+1)\gamma I - L$ are each invertible. Then,
\begin{align*}
&
\int_0^t
\bigg(
\left((\eta + 1)\gamma I - L^\T\right)^{-1}
\left(e^{(\eta\gamma I - L^\T)s} - e^{-\gamma Is}\right)
\nu\nu^\T c e^{-L s}
+
e^{-L^\T s}
\nu\nu^\T c
\left(e^{(\eta\gamma I - L)s} - e^{-\gamma Is}\right)
\\
&
\quad
\cdot
\left((\eta + 1)\gamma I - L\right)^{-1}
\bigg)
\mathrm{d}s
\\
=
&
c\left((\eta + 1)\gamma I - L^\T\right)^{-1}
\bigg(
(\eta + 2)\gamma  M_{\eta \gamma, \nu, L}(t)
+
 e^{(\eta\gamma I - L^\T)t}\nu\nu^\T  e^{-L t}
 -
 \nu\nu^\T
+
\nu\nu^\T  \left(e^{-(\gamma I + L) t} -  I\right)(\gamma I + L)^{-1}
\\
&
\quad
\cdot \left((\eta + 1)\gamma I - L\right)
+
\left((\eta + 1)\gamma I - L^\T\right)(\gamma I + L^\T)^{-1}
\left(e^{-(\gamma I + L^\T) t} - I \right)\nu\nu^\T
\bigg)
\left((\eta + 1)\gamma I - L\right)^{-1}
\end{align*}
for all $t \geq 0$.
\end{lemma}
\begin{proof}
\minrev{The proof follows from the given definition of $M_{\gamma, \nu , L}(t)$, the product rule, and the preceding lemma.}
\end{proof}

\edit{With these lemmas and definitions now in hand we can proceed to our analysis of the $Hawkes/PH/\infty$ queueing system. \minrev{These results}, stated in the following theorem, \minrev{make} use of the form of the infinitesimal generator in Lemma~\ref{fubinifundQ}, \minrev{with simplification} through linearity of expectation and the binomial theorem.}

\subsection{Mean Dynamics of the $Hawkes/PH/\infty$ Queue} \label{ssmeandyn}

\edit{To begin investigation of the $Hawkes/PH/\infty$ queueing system, we first derive differential equations for the moments of the number in each phase of service and the intensity.}

\begin{theorem}\label{hawkdethm}
Consider a queueing system with arrivals occurring in accordance to a Hawkes process $\left(\lambda_t, N_t\right)$ with dynamics given in Equation~\ref{hawkdyn} and phase-type distributed service. Then we have differential equations for the moments of $Q_{t,i}$ given by
\begin{align}
%
\frac{\mathrm{d}}{\mathrm{d}t}\E{Q_{t,i}^m} &= \theta_i \sum_{g=0}^{m-1} {m \choose g}  \E{\lambda_t Q_{t,i}^g}
+ \sum_{g=0}^{m-1}\sum_{\substack{j=1\\ j\ne i}}^n {m \choose g} \mu_{ji} \E{Q_{t,j} Q_{t,i}^g}
\\&+ \sum_{g=1}^{m}{m \choose g - 1}\mu_i(-1)^{m-g+1}\E{Q_{t,i}^g}                                                        ,  \nonumber
\intertext{for the products of $Q_{t,i}$ and $Q_{t,j}$ where $i \ne j$ given by}
 \frac{\mathrm{d}}{\mathrm{d}t}\E{Q_{t,i}^mQ_{t,j}^l} &=
\theta_i \sum_{g=0}^{m-1} {m \choose g} \E{\lambda_t  Q_{t,j}^l  Q_{t,i}^g}
+
\theta_j \sum_{h=0}^{l-1} {l \choose h} \E{\lambda_t  Q_{t,i}^m Q_{t,j}^h }
\\&
+
\sum_{\substack{k=1 \\ i \ne k \ne j}}^n \sum_{g=0}^{m-1} {m \choose g} \mu_{ki} \E{Q_{t,k} Q_{t,i}^{g} Q_{t,j}^l  }
+
\sum_{\substack{k=1 \\ j \ne k \ne i}}^n \sum_{h=0}^{l-1} {l \choose h} \mu_{kj} \E{Q_{t,k} Q_{t,i}^m  Q_{t,j}^{h} }  \nonumber
\\&
+
\mu_i\sum_{g=0}^{m-1} {m \choose g}  (-1)^{m-g} \E{Q_{t,j}^l Q_{t,i}^{g+1}}
+
\mu_{ij} \sum_{g=0}^{m} \sum_{h=0}^{l-1} {m \choose g} {l \choose h}  (-1)^{m-g}\E{Q_{t,i}^{g+1} Q_{t,j}^h }          \nonumber
\\&
+
\mu_j\sum_{h=0}^{l-1} {l \choose h} (-1)^{l-h}\E{Q_{t,i}^m Q_{t,j}^{h+1}}
+
\mu_{ji} \sum_{h=0}^{l} \sum_{g=0}^{m-1} {l \choose h} {m \choose g}   (-1)^{l-h}\E{Q_{t,j}^{h+1} Q_{t,i}^g }        , \nonumber
\intertext{and for the products of $\lambda_t$ and $Q_{t,i}$ given by}
 \frac{\mathrm{d}}{\mathrm{d}t}\E{\lambda_t^mQ_{t,i}^l}
 &=
 \beta\lambda^* m \E{\lambda_t^{m-1} Q_{t,i}^l }
-  \beta m \E{\lambda_t^{m} Q_{t,i}^l }
+ \theta_i \sum_{g=0}^m \sum_{h=0}^{l-1} {m \choose g} {l \choose h}
\\&\cdot \alpha^{m-g} \E{ \lambda_t^{g+1} Q_{t,i}^h}
+  \sum_{g=0}^{m-1} {m \choose g} \alpha^{m-g}\E{\lambda_t^{g+1} Q_{t,i}^l } + \mu_i \sum_{h=0}^{l-1}{l \choose h}          \nonumber
\\&\cdot  (-1)^{l-h} \E{ \lambda_t^m Q_{t,i}^{h+1}}
+ \sum_{\substack{j=1 \\ j \ne i}}^n \sum_{h=0}^{l-1} {l \choose h} \mu_{ji}  \E{\lambda_t^m Q_{t,j}   Q_{t,i}^h}         ,  \nonumber
%
%
%
\end{align}
where $t \geq 0$.
\end{theorem}
\vskip 1em
\begin{proof}
We can first observe that each of these moments can be generalized to $\E{\lambda_t^m Q_{t,i}^l Q_{t,j}^k}$. From Lemma~\ref{fubinifundQ} we see that
\begin{align*}
&\frac{\mathrm{d}}{\mathrm{d}t} \E{\lambda_t^m Q_{t,i}^l Q_{t,j}^k}
=
\text{E}\Bigg[
\beta(\lambda^* - \lambda_t)m \lambda_t^{m-1} Q_{t,i}^l Q_{t,j}^k
+
\lambda_t \theta_i \left( (\lambda_t + \alpha)^m (Q_{t,i} +1)^l Q_{t,j}^k - \lambda_t^m Q_{t,i}^l Q_{t,j}^k \right)
\\
&
\quad
+
\lambda_t \theta_j \left( (\lambda_t + \alpha)^m Q_{t,i}^l (Q_{t,j} +1)^k  - \lambda_t^m Q_{t,i}^l Q_{t,j}^k \right)
+
\sum_{\substack{x=1 \\ j \ne x \ne i}}^n \lambda_t \theta_x Q_{t,i}^l
Q_{t,j}^k \left((\lambda_t + \alpha)^m - \lambda_t^m\right)
\\
&
\quad
+
\sum_{\substack{x=1 \\ i \ne x \ne j}}^n \mu_{xi} Q_{t,x} \lambda_t^m Q_{t,j}^k \left((Q_{t,i}+1)^l - Q_{t,i}^l\right)
+
\sum_{\substack{x=1 \\ j \ne k \ne i}}^n \mu_{xj} Q_{t,x} \lambda_t^m Q_{t,i}^l \left((Q_{t,j}+1)^k - Q_{t,j}^k\right)
\\
&
\quad
+
\sum_{\substack{x=0 \\ i \ne x \ne j}}^n \mu_{ix} Q_{t,i} \lambda_t^m Q_{t,j}^k \left((Q_{t,i}-1)^l - Q_{t,i}^l\right)
+
\sum_{\substack{x=0 \\ j \ne x \ne i}}^n \mu_{jx} Q_{t,j} \lambda_t^m Q_{t,i}^l \left((Q_{t,j}-1)^k - Q_{t,j}^k\right)
\\
&
\quad
+
\mu_{ij} Q_{t,i} \lambda_t^m \left((Q_{t,i}-1)^l(Q_{t,j}+1)^k - Q_{t,i}^lQ_{t,j}^k \right)
 +
 \mu_{ji} Q_{t,j} \lambda_t^m \left((Q_{t,j}-1)^k(Q_{t,i}+1)^l - Q_{t,i}^lQ_{t,j}^k \right)
\Bigg]
\intertext{\edit{where} we have combined the transfers from one phase to another and departures from that phase into the same summation by starting the index at 0. Using the binomial theorem and linearity of expectation, we have the following:}
&
\minrev{\frac{\mathrm{d}}{\mathrm{d}t} \E{\lambda_t^m Q_{t,i}^l Q_{t,j}^k}}
=
\beta \lambda^* m \E{\lambda_t^{m-1} Q_{t,i}^l Q_{t,j}^k}
-
\beta m \E{\lambda_t^{m} Q_{t,i}^l Q_{t,j}^k}
+
\sum_{\substack{x=1 \\ j \ne x \ne i}}^n \sum_{y=0}^{m-1} {m \choose y} \theta_x \alpha^{m-y}
\\
&
\cdot
\E{ \lambda_t^{y+1} Q_{t,i}^l Q_{t,j}^k }
+
\theta_i \left( \sum_{x=0}^m  \sum_{y=0}^l  {m \choose x}   {l \choose y} \alpha^{m-x} \E{ \lambda_t^{x+1} Q_{t,i}^y Q_{t,j}^k }   - \E{ \lambda_t^{m+1} Q_{t,i}^l Q_{t,j}^k } \right)
\\
&
+
\theta_j \Bigg( \sum_{x=0}^m  \sum_{y=0}^k {m \choose x}   {k \choose y}
  \alpha^{m-x} \E{ \lambda_t^{x+1} Q_{t,i}^l Q_{t,j}^y }   - \E{ \lambda_t^{m+1} Q_{t,i}^l Q_{t,j}^k } \Bigg)
+
\sum_{\substack{x=1 \\ i \ne x \ne j}}^n \sum_{y = 0}^{l-1} {l \choose y} \mu_{xi}
\\
&
\cdot
 \E{\lambda_t^m Q_{t,x} Q_{t,i}^y Q_{t,j}^k }
+
\sum_{\substack{x=1 \\ i \ne x \ne j}}^n \sum_{y = 0}^{k-1} {k \choose y} \mu_{xj}
\E{\lambda_t^m Q_{t,x} Q_{t,i}^l Q_{t,j}^y }
+
\sum_{\substack{x=0 \\ i \ne x \ne j}}^n \sum_{y=0}^{l-1} {l \choose y} (-1)^{l-y} \mu_{ix} \E{ \lambda_t^m Q_{t,i}^{y+1} Q_{t,j}^k }
\\
&
+
\sum_{\substack{x=0 \\ i \ne x \ne j}}^n \sum_{y=0}^{k-1} {k \choose y} (-1)^{k-y} \mu_{jx} \E{ \lambda_t^m Q_{t,i}^{l} Q_{t,j}^{y+1} }
+
\mu_{ij} \Bigg( \sum_{x=0}^l  \sum_{y=0}^k  {l \choose x}   {k \choose y} (-1)^{l-x} \E{ \lambda_t^{m} Q_{t,i}^{x+1} Q_{t,j}^y }
\\
&
  - \E{ \lambda_t^{m} Q_{t,i}^{l+1} Q_{t,j}^k } \Bigg)
+
\mu_{ji} \left( \sum_{x=0}^l  \sum_{y=0}^k  {l \choose x}   {k \choose y} (-1)^{k-y} \E{ \lambda_t^{m} Q_{t,i}^{x} Q_{t,j}^{y+1} }   - \E{ \lambda_t^{m} Q_{t,i}^{l} Q_{t,j}^{k+1} } \right) .
\intertext{Now we simplify by recognizing that $\sum_{x \ne j}\mu_{ix} = \mu_i - \mu_{ij}$ and $\sum_{i \ne x \ne j} \theta_x = 1 - \theta_i - \theta_j$. This leaves us with}
&
\minrev{\frac{\mathrm{d}}{\mathrm{d}t} \E{\lambda_t^m Q_{t,i}^l Q_{t,j}^k}}
=
\beta \lambda^* m \E{\lambda_t^{m-1} Q_{t,i}^l Q_{t,j}^k}
-
\beta m \E{\lambda_t^{m} Q_{t,i}^l Q_{t,j}^k}
+
\sum_{y=0}^{m-1} {m \choose y} \alpha^{m-y} \E{ \lambda_t^{y+1} Q_{t,i}^l Q_{t,j}^k }
\\
&
+
\theta_i \sum_{x=0}^m  \sum_{y=0}^{l-1}  {m \choose x}   {l \choose y} \alpha^{m-x} \E{ \lambda_t^{x+1} Q_{t,i}^y Q_{t,j}^k }
+
\theta_j \sum_{x=0}^m  \sum_{y=0}^{k-1} {m \choose x}   {k \choose y}
  \alpha^{m-x} \E{ \lambda_t^{x+1} Q_{t,i}^l Q_{t,j}^y }
\\
&
+
\sum_{\substack{x=1 \\ i \ne x \ne j}}^n \sum_{y = 0}^{l-1} {l \choose y} \mu_{xi} \E{\lambda_t^m Q_{t,x} Q_{t,i}^y Q_{t,j}^k }
+
\sum_{\substack{x=1 \\ i \ne x \ne j}}^n \sum_{y = 0}^{k-1} {k \choose y} \mu_{xj}
\E{\lambda_t^m Q_{t,x} Q_{t,i}^l Q_{t,j}^y }
\\
&
+
\mu_i \sum_{y=0}^{l-1} {l \choose y} (-1)^{l-y}  \E{ \lambda_t^m Q_{t,i}^{y+1} Q_{t,j}^k }
+
\mu_{ij} \sum_{x=0}^l  \sum_{y=0}^{k-1}  {l \choose x}   {k \choose y} (-1)^{l-x} \E{ \lambda_t^{m} Q_{t,i}^{x+1} Q_{t,j}^y }
\\
&
+
\mu_j   \sum_{y=0}^{k-1} {k \choose y} (-1)^{k-y}  \E{ \lambda_t^m Q_{t,i}^{l} Q_{t,j}^{y+1} }
+
\mu_{ji}  \sum_{y=0}^k \sum_{x=0}^{l-1}  {l \choose x}   {k \choose y} (-1)^{k-y} \E{ \lambda_t^{m} Q_{t,i}^{x} Q_{t,j}^{y+1} }
\end{align*}
which is equivalent to each stated result when $m = k = 0$, $k = 0$, and $m = 0$, respectively.
\end{proof} 

\edit{We can now observe that we can form closed systems of linear ordinary differential equations from these equations. To do so, we \minrev{restrict our focus to the equations for} moments of combined power at most $m \in \mathbb{Z}^+$. Of course, the collection of equations that is of most practical interest is \minrev{found by setting} $m=2$, as this yields a system for the means and variances.}  This now gives rise to Corollary~\ref{hawkqueuede}, which \edit{states} the differential equations for the mean, variance, and covariances of queues driven by Hawkes processes.

\begin{corollary}\label{hawkqueuede}
Consider a queueing system with arrivals occurring in accordance to a Hawkes process $\left(\lambda_t, N_t\right)$ with dynamics given in Equation~\ref{hawkdyn} and phase-type distributed service. Then, we have the following differential equations for the mean, variance, and covariances of the number of entities in each phase and in the system as a whole:
\begin{align}
%
%
 \frac{\mathrm{d}}{\mathrm{d}t}\E{Q_{t,i}} &= \theta_i \E{\lambda_t} + \sum_{\substack{j=1 \\ j \ne i}}^n \mu_{ji} \E{Q_{t,j}} - \mu_{i} \E{Q_{t,i}}
 \label{meanqueuede}
%
%
%
%
%
%
%
%
 \\ \frac{\mathrm{d}}{\mathrm{d}t}  \Var{Q_{t,i}} &= \theta_i \E{\lambda_t} + 2 \theta_i \Cov{\lambda_t, Q_{t,i}}
 + 2\sum_{\substack{j=1 \\ j \ne i}}^n \mu_{ji} \Cov{Q_{t,i},Q_{t,j}} + \mu_{i}\E{Q_{t,i}}
 \label{varqueuede}
\\&\quad\,\,\,\,  +  \sum_{\substack{j=1 \\ j \ne i}}^n \mu_{ji} \E{Q_{t,j}} - 2\mu_{i}\Var{Q_{t,i}} \nonumber
%
%
\\ \frac{\mathrm{d}}{\mathrm{d}t}\Cov{\lambda_t, Q_{t,i}} &=
(\alpha - \beta - \mu_i) \Cov{\lambda_t, Q_{t,i} }
+  \alpha \theta_i \E{ \lambda_t }
+ \sum_{\substack{j=1 \\ j \ne i}}^n \mu_{ji}  \Cov{\lambda_t, Q_{t,j}}
\label{covintensityde}
\\&\quad\,\,\,\, + \theta_i \Var{\lambda_t} \nonumber
%
%
\\ \frac{\mathrm{d}}{\mathrm{d}t}\Cov{Q_{t,i}, Q_{t,j}} &=
\theta_i \Cov{\lambda_t, Q_{t,j}}
+
\theta_j \Cov{\lambda_t, Q_{t,i}}
-
(\mu_i + \mu_j)\Cov{Q_{t,i},Q_{t,j}}
\label{covqueuede}
\\&\quad\,\,\,\,
\nonumber
+
\sum_{\substack{k=1 \\ k \ne i}}^n \mu_{ki} \Cov{Q_{t,k}, Q_{t,j}}
+
\sum_{\substack{k=1 \\ k \ne i}}^n \mu_{kj} \Cov{Q_{t,k}, Q_{t,i}}
-
\mu_{ij} \E{Q_{t,i}}
-
\mu_{ji} \E{Q_{t,j}} .
%
%
%
%
%
%
%
\end{align}
\end{corollary}

\edit{We will find that it is quite useful to also be able to state the equations in Corollary~\ref{hawkqueuede} in linear algebraic form. Recall that the vector of the number in each phase of service is $Q_t \in \mathbb{N}^{n}$, the distribution of arrivals into phases is $\theta \in [0,1]^n$, and the sub-generator-matrix for the $n$ phases of service is $S \in \mathbb{R}^{n \times n}$ so that $S_{i,i} = - \mu_i$ for each $i \in \{1, \dots, n\}$ and $S_{i,j} = \mu_{i,j}$ for all $j \ne i$.  We now also incorporate the notation $\diag{x} \in \mathbb{R}^{n \times n}$ for $x \in \mathbb{R}^n$ as $\diag{x} \equiv \sum_{i=1}^n \textbf{V}_i x \textbf{v}_i^\T$, where $\textbf{v}_i \in \mathbb{R}^{n}$ is the unit column vector in the direction of the $i^\text{th}$ coordinate and $\textbf{V}_i =  \textbf{v}_i \textbf{v}_i^\T$, meaning that the $i^\text{th}$ diagonal element is 1 and the rest are 0.
Together, we have that the vector form of Equation~\ref{meanqueuede} is
$$
\frac{\mathrm{d}}{\mathrm{d}t}\E{Q_t}
=
\theta \E{\lambda_t} + S^\T \E{Q_t},
$$
the vector form of Equation~\ref{covintensityde} is
$$
\frac{\mathrm{d}}{\mathrm{d}t}\Cov{\lambda_t, Q_t}
=
\left(S^\T - (\beta - \alpha)I\right)\Cov{\lambda_t, Q_t} + \alpha \theta \E{\lambda_t} + \theta \Var{\lambda_t},
$$
and the matrix form of Equations~\ref{varqueuede} and~\ref{covqueuede} is
\begin{align*}
\frac{\mathrm{d}}{\mathrm{d}t}\Cov{Q_t, Q_t}
&=
S^\T\Cov{Q_t, Q_t} + \Cov{Q_t, Q_t}S + \theta \Cov{\lambda_t, Q_t}^\T + \Cov{\lambda_t, Q_t}\theta^\T
\\&\quad
+ \diag{\theta \E{\lambda_t} + S^\T \E{Q_t}} - S^\T \diag{\E{Q_t}} - \diag{\E{Q_t}}S
\end{align*}
where the diagonal elements of the matrix $\Cov{Q_t, Q_t}$ correspond to the variance of the number in each phase of service and the off-diagonal elements represent the covariance between two phases of service. We can now use the  technical lemmas in Subsection~\ref{sslemmas} to find explicit linear algebraic solutions to the closed system of differential equations in Corollary~\ref{hawkqueuede}.
}

\begin{theorem}\label{hawkqueuegeneral}
Consider a queueing system with arrivals occurring in accordance to a Hawkes process $\left(\lambda_t, N_t\right)$ with dynamics given in Equation~\ref{hawkdyn} \edit{ with $\alpha < \beta$} and phase-type distributed service. Let $S \in \mathbb{R}^{n \times n}$ be the sub-generator matrix for the transient states in the phase-distribution CTMC and let $\theta \in [0,1]^n$ be the initial distribution for arrivals to these states. If \edit{$S + (\beta - \alpha)I$ is} invertible, then \edit{the vector of the mean number in service in each phase of service is}
\begin{align}
&
%
%
%
%
%
%
\E{Q_t}
=
\lambda_\infty\hspace{-.1cm} \left(- S^\mathrm{T}\right)^{-1}\big(I - e^{S^\mathrm{T} t}\big)\theta
-
\left( \lambda_0 - \lambda_\infty \right)\left(S^\mathrm{T} + (\beta - \alpha)I\right)^{-1}\hspace{-.1cm}\big(e^{-(\beta - \alpha) t} I - e^{S^\mathrm{T} t}\big)\theta
\intertext{where $\lambda_\infty = \frac{\beta \lambda^*}{\beta - \alpha}$. Further, \edit{the vector of covariances between the intensity and each phase of service is}
\vspace{-.25cm}
}
&
%
%
%
%
%
\Cov{\lambda_t, Q_t}
=
\frac{\alpha(2\beta - \alpha)\lambda_\infty}{2(\beta - \alpha)}\left((\beta - \alpha)I - S^\mathrm{T}\right)^{-1}\left(I - e^{(S^\mathrm{T} - (\beta - \alpha)I)t}\right)\theta
-
\frac{\alpha\beta(\lambda_0 - \lambda_\infty)}{\beta - \alpha}
\nonumber
\\
&
\quad
\cdot \left(S^\mathrm{T}\right)^{-1} \left(e^{-(\beta - \alpha)t}I - e^{(S^\mathrm{T} - (\beta - \alpha)I)t}\right) \theta
+
\frac{\alpha^2 (2\lambda_0 - \lambda_\infty)}{2(\beta - \alpha)} \left(S^\mathrm{T} + (\beta - \alpha)I\right)^{-1}
\nonumber
\\
&
\quad
\cdot \left(e^{-2(\beta - \alpha)t}I - e^{(S^\mathrm{T} - (\beta - \alpha)I)t}\right) \theta  \, .
\intertext{Finally, \edit{the matrix of covariances between phases of service is given by}
\vspace{-.25cm}
}
&
%
%
%
%
%
\Cov{Q_t, Q_t}
=
\frac{\alpha(2\beta - \alpha)\lambda_\infty}{2(\beta - \alpha)}
\left((\beta - \alpha)I - S^\T\right)^{-1}
\Bigg(
2(\beta - \alpha)e^{S^\T t} M_{0,\theta, S}(t)e^{S t}
+
\theta \theta^\T
-
e^{S^\T t}
\theta \theta^\T
e^{S t}
\nonumber
\\
&
+
e^{S^\T t}
\theta \theta^\T
\left(e^{-(\beta - \alpha) t}I - e^{S t} \right)
((\beta - \alpha)I + S)^{-1}((\beta - \alpha)I - S)
+
\left((\beta - \alpha)I - S^\T\right)\left((\beta - \alpha)I + S^\T\right)^{-1}
\nonumber
\\
&
\cdot
\left(e^{-(\beta - \alpha) t}I - e^{S^\T t}\right) \theta \theta^\T e^{S t}
\Bigg)
\left((\beta - \alpha)I - S\right)^{-1}
+
\frac{\alpha\beta(\lambda_0 - \lambda_\infty)}{\beta - \alpha}
\left( S^\T\right)^{-1}
\Bigg(
(\beta - \alpha)e^{S^\T t} M_{-(\beta - \alpha),\theta, S}(t)e^{S t}
\nonumber
\\
&
+
e^{-(\beta - \alpha) t}\theta \theta^\T
-
e^{S^\T t}
\theta \theta^\T
e^{S t}
-
e^{S^\T t}
\theta \theta^\T  \left(e^{-(\beta - \alpha) t}I - e^{S t} \right)
((\beta - \alpha)I + S)^{-1}S
-
S^\T\left((\beta - \alpha)I + S^\T\right)^{-1}
\nonumber
\\
&
\cdot
\Big(e^{-(\beta - \alpha) t}I - e^{S^\T t}\Big)
\theta \theta^\T e^{S t}
\Bigg)
S^{-1}
-
\frac{\alpha^2(2\lambda_0 - \lambda_\infty)}{2(\beta - \alpha)}\big((\beta - \alpha)I + S^\T\big)^{-1}
\Bigg(
e^{-2(\beta - \alpha) t}\theta \theta^\T
-
e^{S^\T t}
\theta \theta^\T
e^{S t}
\nonumber
\\
&
-
e^{S^\T t}
\theta \theta^\T  \left(e^{-(\beta - \alpha) t}I - e^{S t} \right)
-
\left(e^{-(\beta - \alpha) t}I - e^{S^\T t}\right) \theta \theta^\T e^{S t}
\Bigg)
\left((\beta - \alpha)I + S\right)^{-1}
-
\lambda_\infty
\mathrm{diag}\Big(
\left(S^\mathrm{T}\right)^{-1}
  \nonumber
 \\
 &
 \cdot
  \left(I - e^{S^\mathrm{T} t}\right)\theta
\Big)
-
 ( \lambda_0 - \lambda_\infty )
\diag{
\left(S^\mathrm{T} + (\beta - \alpha)I\right)^{-1}\left(e^{-(\beta - \alpha) t} I - e^{S^\mathrm{T} t}\right)\theta
}
\end{align}
where all $t \geq 0$.
\end{theorem}
\begin{proof}
Throughout this proof we use the fact that a matrix being invertible implies that its transpose is invertible as well. To begin, we can see from Corollary~\ref{hawkqueuede} that
$$
\frac{\mathrm{d}}{\mathrm{d}t} \E{Q_t} = S^\mathrm{T}\E{Q_t} + \theta\E{\lambda_t} = S^\mathrm{T}\E{Q_t} + \theta\left( \lambda_\infty + \left(\lambda_0 - \lambda_\infty\right)e^{-(\beta - \alpha)t}\right)
$$
and so we apply Lemma~\ref{mainlemma}. Let $\nu_1 = \theta \lambda_\infty$ and $\eta_1 = \gamma_1 = 0$, and let $\nu_2 = \theta(\lambda_0 - \lambda_\infty)$, $\eta_2 = 0$, and $\gamma_2 = -(\beta - \alpha)$. We assume that the queue starts empty. Then, we have
\begin{align*}
\E{Q_t} &= -\left(S^\mathrm{T}\right)^{-1}\theta \lambda_\infty + \left(S^\mathrm{T}\right)^{-1}e^{-S^\mathrm{T} t}\theta \lambda_\infty - \left(S^\mathrm{T} + (\beta - \alpha)I\right)^{-1}\theta(\lambda_0 - \lambda_\infty)e^{-(\beta - \alpha)t}
\\&\quad
+ \left(S^\mathrm{T} + (\beta - \alpha)I\right)^{-1}e^{-S^\mathrm{T}t}\theta(\lambda_0 - \lambda_\infty)
\end{align*}
which now simplifies to the stated result. \edit{Note \minrev{that} $S$ is invertible because it is diagonally dominant by definition and we have assumed the invertibility of $S + (\beta - \alpha)I$, which implies non-singularity of the respective transposes.} We find the stated result for $\Cov{\lambda_t, Q_t}$ through repeating the same technique to the corresponding differential equation systems, where \edit{again we make use of the linear algebraic representation. Thus, we are left to} solve for the covariance matrix. Note that from Corollary~\ref{hawkqueuede}, the variance of each phase and the covariance between phases can form one linear algebraic form as the covariance matrix, as shown below.
\begin{align*}
\frac{\mathrm{d}}{\mathrm{d}t}\Cov{Q_t, Q_t}
&=
S^\T \Cov{Q_t, Q_t}
+
\Cov{Q_t, Q_t} S
+
\theta \Cov{\lambda_t, Q_t}^\T
+
\Cov{\lambda_t, Q_t} \theta^\T
\\
&
\quad
+
\diag{\theta \E{\lambda_t} + S^\T \E{Q_t}}
-
S^\T\diag{\E{Q_t}}
-
\diag{\E{Q_t}}S
\end{align*}
Using the product rule and multiplying through by matrix exponentials on the right and left, we can also express this as below:
\begin{align*}
\frac{\mathrm{d}}{\mathrm{d}t}\left(e^{-S^\T t} \Cov{Q_t, Q_t} e^{- S t} \right)
&=
e^{-S^\T t}
\theta \Cov{\lambda_t, Q_t}^\T
e^{- S t}
+
e^{-S^\T t}
\Cov{\lambda_t, Q_t} \theta^\T
e^{- S t}
\\
&
\quad
+
e^{-S^\T t}
\diag{\theta \E{\lambda_t} + S^\T \E{Q_t}}
e^{- S t}
-
e^{-S^\T t}
S^\T\diag{\E{Q_t}}
e^{- S t}
\\
&
\quad
-
e^{-S^\T t}
\diag{\E{Q_t}}S
e^{- S t} .
\end{align*}
For the pair of $\Cov{\lambda_t, Q_t}$ terms, we use Lemma~\ref{doublemintgum} in conjunction with the explicit function for $\Cov{\lambda_t, Q_t}$ to find
\begin{align*}
&
\int_0^t
\left(
e^{-S^\T s}
\theta \Cov{\lambda_s, Q_s}^\T
e^{- S s}
+
e^{-S^\T s}
\Cov{\lambda_s, Q_s} \theta^\T
e^{- S s}
\right) \mathrm{d}s
\\
&
=
\frac{\alpha(2\beta - \alpha)\lambda_\infty}{2(\beta - \alpha)}
\left((\beta - \alpha)I - S^\T\right)^{-1}
\Bigg(
2(\beta - \alpha) M_{0,\theta, S}(t)
+
e^{-S^\T t}\theta \theta^\T e^{- S t}
-
\theta\theta^\T
+
\theta \theta^\T \left(e^{-((\beta - \alpha)I + S) t} - I \right)
\\
&
\cdot
((\beta - \alpha)I + S)^{-1}((\beta - \alpha)I - S)
+
\left((\beta - \alpha)I - S^\T\right)\left((\beta - \alpha)I + S^\T\right)^{-1}\left(e^{-((\beta - \alpha)I + S^\T) t} - I\right) \theta \theta^\T
\Bigg)
\\
&
\cdot
\left((\beta - \alpha)I - S\right)^{-1}
+
\frac{\alpha\beta(\lambda_0 - \lambda_\infty)}{\beta - \alpha}
\left( S^\T\right)^{-1}
\Bigg(
(\beta - \alpha) M_{-(\beta - \alpha),\theta, S}(t)
+
e^{-((\beta - \alpha)I + S^\T) t}\theta \theta^\T e^{- S t}
-
\theta\theta^\T
\\
&
-
\theta \theta^\T  \left(e^{-((\beta - \alpha)I + S) t} - I \right)
((\beta - \alpha)I + S)^{-1}S
-
S^\T\left((\beta - \alpha)I + S^\T\right)^{-1}\left(e^{-((\beta - \alpha)I + S^\T) t} - I\right) \theta \theta^\T
\Bigg)
S^{-1}
\\
&
-
\frac{\alpha^2(2\lambda_0 - \lambda_\infty)}{2(\beta - \alpha)}\big((\beta - \alpha)I + S^\T\big)^{-1}
\Bigg(
e^{-(2(\beta - \alpha)I + S^\T) t}\theta \theta^\T e^{- S t}
-
\theta\theta^\T
-
\theta \theta^\T  \left(e^{-((\beta - \alpha)I + S) t} - I \right)
\\
&
-
\left(e^{-((\beta - \alpha)I + S^\T) t} - I\right) \theta \theta^\T
\Bigg)
\left((\beta - \alpha)I + S\right)^{-1}
\end{align*}
and so we now integrate the remaining terms in the covariance matrix differential equations. Note that the product rule for three terms is $(fgh)' = f'gh + fg'h + fgh'$. We have already used this in concatenating the covariance matrix terms in the differential equation, and we can now make use of it again. Recall that $\frac{\mathrm{d}}{\mathrm{d}t} \E{Q_t} = S^\mathrm{T}\E{Q_t} + \theta\E{\lambda_t}$. Using this realization, the integral of the remaining three terms is
\begin{align*}
&
\int_0^t
\Big(
e^{-S^\T s}
\diag{\theta \E{\lambda_s} + S^\T \E{Q_s}}
e^{- S s}
-
e^{-S^\T s}
S^\T\diag{\E{Q_s}}
e^{- S s}
-
e^{-S^\T s}
\diag{\E{Q_s}}S
e^{- S s}
\Big)
\mathrm{d}s
\\
&
=
e^{-S^\T t}
\diag{\E{Q_t}}
e^{- S t}
\\
&
=
-
e^{-S^\T t}
\diag{
\left(S^\mathrm{T}\right)^{-1}\left(I - e^{S^\mathrm{T} t}\right)\theta
}
 e^{- S t} \lambda_\infty
 -
e^{-S^\T t}
\diag{
\left(S^\mathrm{T} + (\beta - \alpha)I\right)^{-1}\left(e^{-(\beta - \alpha) t} I - e^{S^\mathrm{T} t}\right)\theta
}
 \\
 &
 \qquad
\cdot
e^{- S t} ( \lambda_0 - \lambda_\infty )
\end{align*}
\edit{where} we are justified in moving the differentiation through the diagonalization and distributing it across sums via the definition of diagonalization as a linear combination. Combining this with the integral for the covariance between the queue and intensity and multiplying each side by the corresponding exponentials, we achieve the stated result.
\end{proof} 

 \vspace{-.2in}
 \begin{figure}[h]
\begin{center}	
\includegraphics[width=.6\textwidth]{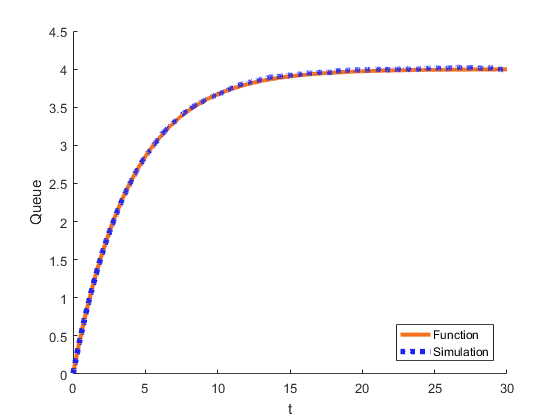}
\caption{Example Mean of the ${Hawkes/PH/\infty}$ Queue with Sub-Generator Matrix ${S_{\mathrm{Cox}}}$ as in Equation~\ref{coxdef}.}
\label{hawkesautocofig}
\end{center}
\end{figure}
 \vspace{-.2in}

As a brief example, consider a Hawkes process driven queueing system with infinite servers and suppose that the service is phase-type distributed with initial distribution $\theta = \onevec{1}$ and the following sub-generator matrix:
\begin{align}\label{coxdef}
S_{\mathrm{Cox}}
=
\begin{bmatrix}
-4 & 3 & 0 & 0 & 0\\
0 & -2 & 1 & 0 & 0\\
0 & 0 & -3 & 2 & 0\\
0 & 0 & 0 & -5 & 4\\
0 & 0 & 0 & 0 & -1
\end{bmatrix}
.
\end{align}
This is referred to as a Coxian distribution. It is characterized by each phase of service having an associated probability of either system departure or advancement to the next phase upon service completion. In this example, $\lambda^* = 1$, $\alpha = \frac 3 4$, and $\beta = 1$. The simulation is based on 100,000 replications.

\begin{remark}
\edit{We now note that the assumed nonsingularity of $S + (\beta - \alpha)I$ is necessary to implement the technical lemmas, but need not hold in order for a closed form solution to exist. If these conditions do not hold, one can instead make use of the structure of invertibility that is implied by a specific phase-type distribution. In Corollaries \ref{hawkqueueerlang} and \ref{hawkqueuehyper}, we demonstrate this for Erlang and hyper-exponential service, respectively. Like we have seen in Theorem~\ref{hawkqueuegeneral}, these expressions can be found through solving systems of differential equations provided by Corollary~\ref{hawkqueuede}.}
\\ \end{remark}

\edit{We start with the case of service times following a} Erlang distribution. In this case, we define $N \in \mathbb{R}^{n \times n}$ as the matrix of all ones on the first lower diagonal and zeros otherwise. Then, $S^\T = n\mu(N - I)$ for this phase-type distribution. Observe that $N$ is a nilpotent matrix of a particular structure: for $k \in \mathbb{N}$, $N^k$ is the matrix of all ones on the $k^\text{th}$ lower diagonal if $k \leq n -1$ and is the zero matrix otherwise. Additionally, in this case $\theta = \textbf{v}_1$ as all arrivals occur in the first phase. With this in hand, we see that
\begin{align*}
\left(M_{\gamma, \textbf{v}_1, n\mu(I-N^\T)}(t)\right)_{i,j} &= \left(M_{\gamma + 2n\mu, \textbf{v}_1,n\mu N^\T}(t)\right)_{i,j}
\\
&
=
\begin{cases}
{i+j-2 \choose i - 1}(n\mu)^{i+j-2}\frac{e^{(\gamma +  2n\mu)t}\sum_{k=0}^{i+j-2}\frac{(-(\gamma + 2n\mu)t)^k}{k!} - 1}{(\gamma + 2n\mu)^{i+j-1}} & \text{if $\gamma + 2n\mu \ne 0$}\\
\frac{(tn\mu)^{i+j-1}}{n\mu(i-1)!(j-1)!(i+j-1)} & \text{if $\gamma + 2n\mu = 0$}
\end{cases}
\end{align*}
and we make use of this in the following corollary.

\begin{corollary}\label{hawkqueueerlang}
Consider a queueing system with arrivals occurring in accordance to a Hawkes process $\left(\lambda_t, N_t\right)$ with dynamics given in Equation~\ref{hawkdyn} \edit{ with $\alpha < \beta$} and Erlang distributed service with $n$ phases and mean $\frac{1}{\mu}$. Then, \edit{when $n\mu \ne \beta - \alpha$, the vector of mean number in each phase of service is given by}
\edit{
\begin{align}
&
%
%
%
%
%
%
\E{Q_t}
=
\frac{\lambda_\infty}{n\mu} \big(I - e^{n\mu(N- I) t}\big){\mathbf{v}}
-
\left( \lambda_0 - \lambda_\infty \right)\left(n\mu N - (n\mu - \beta + \alpha)I\right)^{-1}
\big(e^{-(\beta - \alpha) t} I - e^{n\mu(N- I) t}\big) {\mathbf{v}}_1,
\end{align}
and when $n\mu = \beta - \alpha$, this vector is
\begin{align}
\E{Q_t}
=
\frac{\lambda_\infty}{n\mu}\hspace{-.1cm} \big(I - e^{n\mu(N- I) t}\big){\mathbf{v}}
+
\left( \lambda_0 - \lambda_\infty \right)e^{n\mu(N-I)t}x(t),
\end{align}
}where $\lambda_\infty = \frac{\beta \lambda^*}{\beta - \alpha}$ \edit{ and $x:\mathbb{R}^+ \to \mathbb{R}^n$ is such that $x_i(t) = \frac{(-n\mu)^{i-1}t^i}{i!}$}. Further, \edit{when $n\mu \ne \beta - \alpha$ the vector of covariances between the number in each phase of service and the intensity is
\begin{align}
&
%
%
%
%
%
\Cov{\lambda_t, Q_t}
=
\lambda_\infty\left(\alpha + \frac{\alpha^2}{2(\beta - \alpha)}\right)\left((n\mu + \beta - \alpha)I - n\mu N\right)^{-1}\left(I - e^{(n\mu N - (n\mu + \beta - \alpha)I)t}\right) {\mathbf{v}}_1
\nonumber
\\
&
\quad
+
\frac{\alpha\beta(\lambda_0 - \lambda_\infty)}{n\mu(\beta - \alpha)}
 \left(e^{-(\beta - \alpha)t}I - e^{(n\mu N - (n\mu + \beta - \alpha)I)t}\right)  {\mathbf{v}}
+
\frac{\alpha^2 (2\lambda_0 - \lambda_\infty)}{2(\beta - \alpha)}
\left(n\mu N - (n\mu - \beta + \alpha)I\right)^{-1}
\nonumber
\\
&
\quad
\cdot  \left(e^{-2(\beta - \alpha)t}I - e^{(n\mu N - (n\mu + \beta - \alpha)I)t}\right)  {\mathbf{v}}_1 ,
\end{align}
and when $n\mu = \beta - \alpha$, this is
\begin{align}
&
\Cov{\lambda_t, Q_t}
=
\lambda_\infty\left(\frac{\alpha}{n\mu} + \frac{\alpha^2}{2(n\mu)^2}\right)\left(2I - N\right)^{-1}\left(I - e^{n\mu (N - 2I)t}\right) {\mathbf{v}}_1
+
(\lambda_0 - \lambda_\infty)
\left(\frac{\alpha}{n\mu} + \frac{\alpha^2}{(n\mu)^2}\right)
\nonumber
\\
&
\quad
\cdot
 \left(e^{-n\mu t}I - e^{n\mu (N - 2I)t}\right)  {\mathbf{v}}
-
\frac{\alpha^2 (2\lambda_0 - \lambda_\infty)}{2n\mu}
 e^{n\mu(N - 2I)t}x(t).
\end{align}
}Finally, \edit{ when $ n\mu \ne \beta - \alpha$, the matrix of the covariance between the number in the phases of service is given by
\begin{align}
&
%
%
%
%
%
\Cov{Q_t, Q_t}
=
\frac{\alpha(2\beta - \alpha)\lambda_\infty}{2(\beta - \alpha)}
\left((n\mu + \beta - \alpha)I - n\mu N\right)^{-1}
\Bigg(
2(\beta - \alpha)e^{n\mu(N-I) t} M_{2n\mu,\onevec{1}, n\mu N^\T}(t)e^{n\mu(N^\T -I) t}
\nonumber
\\
&
\quad
+
\onevec{1} \onevec{1}^\T
-
e^{n\mu(N-I) t}
\onevec{1} \onevec{1}^\T
e^{n\mu(N^\T-I) t}
+
e^{n\mu(N-I) t}
\onevec{1} \onevec{1}^\T
\hspace{-.1cm}
\left(e^{-(\beta - \alpha) t}I - e^{n\mu(N^\T-I) t} \right)
\hspace{-.1cm}
(n\mu N^\T - (n\mu - \beta + \alpha)I)^{-1}
\nonumber
\\
&
\quad
\cdot
((n\mu + \beta - \alpha)I - n\mu N^\T)
+
((n\mu + \beta - \alpha)I - n\mu N)(n\mu N - (n\mu - \beta + \alpha)I)^{-1}
\left(e^{-(\beta - \alpha) t}I - e^{n\mu (N - I) t}\right)
\nonumber
\\
&
\quad
\cdot
 \onevec{1} \onevec{1}^\T e^{n\mu (N^\T - I) t}
\Bigg)
\left((n\mu + \beta - \alpha)I - n\mu N^\T \right)^{-1}
+
\frac{\alpha\beta(\lambda_0 - \lambda_\infty)}{(n\mu)^2(\beta - \alpha)}
\left( N - I\right)^{-1}
\Bigg(
(\beta - \alpha)e^{n\mu(N-I) t}
\nonumber
\\
&
\quad
\cdot
M_{2n\mu-\beta + \alpha,\onevec{1}, n\mu N^\T}(t)e^{n\mu(N^\T - I) t}
+
e^{-(\beta - \alpha) t}\onevec{1} \onevec{1}^\T
-
e^{n\mu(N-I) t}
\onevec{1} \onevec{1}^\T
e^{n\mu (N^\T - I) t}
-
n\mu
e^{n\mu(N-I) t}
\onevec{1} \onevec{1}^\T
\nonumber
\\
&
\quad
\cdot
 \left(e^{-(\beta - \alpha) t}I - e^{n\mu (N^\T - I) t} \right)
(n\mu N^\T - (n\mu - \beta + \alpha)I)^{-1}(N^\T-I)
-
n\mu(N-I)(n\mu N - (n\mu - \beta + \alpha)I)^{-1}
\nonumber
\\
&
\quad
\cdot
\Big(e^{-(\beta - \alpha) t}I - e^{n\mu(N-I) t}\Big)
\onevec{1} \onevec{1}^\T  e^{n\mu(N^\T-I) t}
\Bigg)
(N^\T - I)^{-1}
-
\frac{\alpha^2(2\lambda_0 - \lambda_\infty)}{2(\beta - \alpha)}(n\mu N - (n\mu - \beta + \alpha)I)^{-1}
\nonumber
\\
&
\quad
\cdot
\Bigg(
e^{-2(\beta - \alpha) t}\onevec{1} \onevec{1}^\T
-
e^{n\mu(N-I) t}
\onevec{1} \onevec{1}^\T
e^{n\mu (N^\T - I) t}
-
e^{n\mu(N-I) t}
\onevec{1} \onevec{1}^\T  \left(e^{-(\beta - \alpha) t}I - e^{n\mu(N^\T-I) t}\right)
\nonumber
\\
&
\quad
-
 \left(e^{-(\beta - \alpha) t}I - e^{n\mu(N-I) t}\right)\onevec{1} \onevec{1}^\T e^{n\mu(N^\T-I) t}
\Bigg)
(n\mu N^\T - (n\mu - \beta + \alpha)I)^{-1}
+
\frac{\lambda_\infty}{n\mu}
\mathrm{diag}\Big(
  \left(I - e^{n\mu(N-I) t}\right)\onevec{}
\Big)
  \nonumber
 \\
 &
 \quad
-
 ( \lambda_0 - \lambda_\infty )
\diag{
\left(n\mu N - (n\mu - \beta + \alpha)I\right)^{-1}\left(e^{-(\beta - \alpha) t} I - e^{n\mu(N-I) t}\right)\onevec{1}
},
\end{align}
whereas when $n\mu = \beta - \alpha$, this matrix is
\begin{align}
&\Cov{Q_t, Q_t}
=
\diag{
\frac{\lambda_\infty}{n\mu}\hspace{-.1cm} \big(I - e^{n\mu(N- I) t}\big){\mathbf{v}}
+
\left( \lambda_0 - \lambda_\infty \right)e^{n\mu(N-I)t}x(t)
}
+
e^{n\mu (N - I)t}
\Bigg(
\lambda_\infty \left(\frac{\alpha}{n\mu}+\frac{\alpha^2}{2(n\mu)^2}\right)
\nonumber
\\
&
\quad
\cdot
\bigg(
\left(M_{2n\mu,\onevec{1},n\mu N^\T}(t) - x(t)\onevec{1}^\T\right)\left(2I - N^\T\right)^{-1}
+
\left(2I - N\right)^{-1} \left(M_{2n\mu,\onevec{1},n\mu N^\T}(t) - \onevec{1}x^\T(t)\right)
\bigg)
+
(\lambda_0 - \lambda_\infty)
\nonumber
\\
&
\quad
\cdot
\left(\frac{\alpha}{n\mu} + \frac{\alpha^2}{(n\mu)^2}\right)
\bigg(
M_{2n\mu,\onevec{1},n\mu N^\T}(t)\left(I - N^\T\right)^{-1} + \left(I - N\right)^{-1} M_{2n\mu,\onevec{1},n\mu N^\T}(t)
-
x(t)\onevec{}^\T - \onevec{}x^\T (t)
\bigg)
\nonumber
\\
&
\quad
-
\frac{\alpha^2(2\lambda_0-\lambda_\infty)}{2n\mu}
\left(X(t) + X^\T (t)\right)
\Bigg)
e^{n\mu (N^\T - I)t},
\end{align}
}
where  all $t \geq 0$ and $X:\mathbb{R}^+ \to \mathbb{R}^{n\times n}$ is such that $X_{i,j}(t) = \frac{(-n\mu)^{i+j-2}t^{i+j-1}}{(i-1)!j!(i+j)}$.
\end{corollary}

As with the Erlang, we also provide explicit formulas for the hyper-exponential distribution. In this case we have that $S = -D$ where $D$ is a diagonal matrix of the rates of service in each phase. This allows it to commute with the symmetric $\theta \theta^\T$, giving us
$$
M_{\gamma, \theta, -D}(t)
=
\int_0^t
e^{(\gamma I + D)s}
\theta\theta^\T  e^{D s}
\,\mathrm{d}s
=
\int_0^t
e^{(\gamma I + 2D)s}
\,\mathrm{d}s
\theta\theta^\T
=
(\gamma I + 2D)^{-1}
\left(
e^{(\gamma I + 2D)t}
-
I
\right)
\theta\theta^\T
$$
as long as $\gamma I + 2D$ is invertible. \edit{However, we also seek to address the case where $(\beta - \alpha)I + S = (\beta - \alpha)I - D$ is not invertible. In the hyper-exponential service setting, $(\beta - \alpha)I - D$ being singular implies that some $\mu_i = \beta - \alpha$, but it is not clear which or for how many $\mu_i$ this is the case. So, we instead use the element-level equations in Corollary~\ref{hawkqueuede} to solve for the explicit expressions. This method is preferable to the linear algebra approach for hyper-exponential service since in this setting $\mu_{ij} = 0$ for every $i$ and $j$.}

\begin{corollary}\label{hawkqueuehyper}
Consider a queueing system with arrivals occurring in accordance to a Hawkes process $\left(\lambda_t, N_t\right)$ with dynamics given in Equation~\ref{hawkdyn} \edit{with $\alpha < \beta$} and hyper-exponential distributed service with $n$ phases \edit{and} distinct service rate\edit{s  $\mu_1, \dots, \mu_n$}. Then, \edit{the mean number in phase $i \in \{1, \dots , n\}$ of service is
\begin{align}
&
%
%
%
%
%
%
\E{Q_{t,i}}
=
\begin{cases}
\frac{\lambda_\infty}{\mu_i}\left(1 - e^{-\mu_i t}\right)\theta_i
+
\frac{\lambda_0 - \lambda_\infty }{\mu_i - \beta + \alpha}\left(e^{-(\beta - \alpha) t}  - e^{- \mu_i t}\right)\theta_i
&
\text{if $\mu_i \ne \beta - \alpha$,}
\\
\frac{\lambda_\infty}{\mu_i}\left(1 - e^{-\mu_i t}\right)\theta_i
+
\left(\lambda_0 - \lambda_\infty \right)\theta_i t e^{-\mu_i t}
&
\text{if $\mu_i = \beta - \alpha$,}
\end{cases}
\end{align}
where} $\lambda_\infty = \frac{\beta \lambda^*}{\beta - \alpha}$. \minrev{Furthermore} \edit{the covariance \minrev{between} the number in phase $i$ of service and the intensity is
\begin{align}
&
%
%
%
%
%
\Cov{\lambda_t, Q_{t,i}}
=
\begin{cases}
\frac{\alpha \theta_i (2\beta - \alpha)\lambda_\infty}{2(\beta - \alpha)(\mu_i + \beta - \alpha)}\left(1 - e^{-(\mu_i + \beta - \alpha)t}\right)
+
\frac{\alpha\beta\theta_i(\lambda_0 - \lambda_\infty)}{\mu_i(\beta - \alpha)} \big(e^{-(\beta - \alpha)t}
&
\\
\quad
- e^{-(\mu_i + \beta - \alpha)t}\big)
-
\frac{\alpha^2 \theta_i (2\lambda_0 - \lambda_\infty)}{2(\beta - \alpha)(\mu_i - \beta + \alpha)}
 \left(e^{-2(\beta - \alpha)t} - e^{-(\mu_i + \beta - \alpha)t}\right)
 &
 \text{if $\mu_i \ne \beta - \alpha$,}
 \\
 \frac{\alpha \theta_i (2\mu_i + \alpha)\lambda_\infty}{4\mu_i^2}\left(1 - e^{-2\mu_i t}\right)
+
\frac{\alpha\beta\theta_i(\lambda_0 - \lambda_\infty)}{\mu_i^2} \big(e^{-\mu_i t}
- e^{-2\mu_i t}\big)
&
\\
\quad
-
\frac{\alpha^2 \theta_i (2\lambda_0 - \lambda_\infty)}{2\mu_i}
te^{-2\mu_i t}
 &
 \text{if $\mu_i = \beta - \alpha$.}
\end{cases}
\end{align}
Then,} \edit{ the covariance between the number in phase $i$ of service and the number in phase $j$ of service where $i, j \in \{1, \dots ,\}$ and $i \ne j$ is
\begin{align}
\Cov{Q_{t,i},Q_{t,j}}
&
=
\begin{cases}
\frac{\alpha\theta_i\theta_j(2\beta-\alpha)\lambda_\infty}{2(\beta-\alpha)(\mu_j+\beta-\alpha)}
\left(
\frac{1-e^{-(\mu_i+\mu_j)t}}{\mu_i+\mu_j}
-
\frac{e^{-(\mu_j+\beta-\alpha)t}-e^{-(\mu_i+\mu_j)t}}{\mu_i-\beta+\alpha}
\right)
&
\\
\quad
+
\frac{\alpha\beta\theta_i\theta_j(\lambda_0-\lambda_\infty)}{\mu_j(\beta-\alpha)}
\left(
\frac{e^{-(\beta-\alpha)t} - e^{-(\mu_i+\mu_j)t}}{\mu_i+\mu_j-\beta+\alpha}
-
\frac{e^{-(\mu_j+\beta-\alpha)t}-e^{-(\mu_i+\mu_j)t}}{\mu_i-\beta+\alpha}
\right)
&
\\
\quad
-
\frac{\alpha^2\theta_i \theta_j(2\lambda_0-\lambda_\infty)}{2(\beta-\alpha)(\mu_j-\beta+\alpha)}
\left(
\frac{e^{2(\beta-\alpha)t}-e^{-(\mu_i+\mu_j)t}}{\mu_i+\mu_j-2\beta+2\alpha}
-
\frac{e^{-(\mu_j+\beta-\alpha)t}-e^{-(\mu_i+\mu_j)t}}{\mu_i-\beta+\alpha}
\right)
&
\\
\quad
+
\frac{\alpha\theta_i\theta_j(2\beta-\alpha)\lambda_\infty}{2(\beta-\alpha)(\mu_i+\beta-\alpha)}
\left(
\frac{1-e^{-(\mu_i+\mu_j)t}}{\mu_i+\mu_j}
-
\frac{e^{-(\mu_i+\beta-\alpha)t}-e^{-(\mu_i+\mu_j)t}}{\mu_j-\beta+\alpha}
\right)
&
\\
\quad
+
\frac{\alpha\beta\theta_i\theta_j(\lambda_0-\lambda_\infty)}{\mu_i(\beta-\alpha)}
\left(
\frac{e^{-(\beta-\alpha)t} - e^{-(\mu_i+\mu_j)t}}{\mu_i+\mu_j-\beta+\alpha}
-
\frac{e^{-(\mu_i+\beta-\alpha)t}-e^{-(\mu_i+\mu_j)t}}{\mu_j-\beta+\alpha}
\right)
&
\\
\quad
-
\frac{\alpha^2\theta_i \theta_j(2\lambda_0-\lambda_\infty)}{2(\beta-\alpha)(\mu_i-\beta+\alpha)}
\left(
\frac{e^{2(\beta-\alpha)t}-e^{-(\mu_i+\mu_j)t}}{\mu_i+\mu_j-2\beta+2\alpha}
-
\frac{e^{-(\mu_i+\beta-\alpha)t}-e^{-(\mu_i+\mu_j)t}}{\mu_j-\beta+\alpha}
\right)
&
\text{if $\mu_i \ne \beta - \alpha \ne \mu_j$,}
\\
\frac{\alpha\theta_i\theta_j(2\beta-\alpha)\lambda_\infty}{4\mu_j^2}
\left(
\frac{1-e^{-(\mu_i+\mu_j)t}}{\mu_i+\mu_j}
-
\frac{e^{-2\mu_j t}-e^{-(\mu_i+\mu_j)t}}{\mu_i-\mu_j}
\right)
+
\frac{\alpha\beta\theta_i\theta_j(\lambda_0-\lambda_\infty)}{\mu_j^2}
&
\\
\quad
\cdot
\left(
\frac{e^{-\mu_j t} - e^{-(\mu_i+\mu_j)t}}{\mu_i}
-
\frac{e^{-2\mu_j t}-e^{-(\mu_i+\mu_j)t}}{\mu_i-\mu_j}
\right)
-
\frac{\alpha^2\theta_i \theta_j(2\lambda_0-\lambda_\infty)}{2\mu_j}
&
\\
\quad
\cdot
\left(
\frac{te^{-2\mu_i t}}{\mu_j-\mu_i}
+
\frac{e^{-(\mu_i+\mu_j)t} - e^{-2 \mu_i t}}{(\mu_j - \mu_i)^2}
\right)
+
\frac{\alpha\theta_i\theta_j(2\beta-\alpha)\lambda_\infty}{2\mu_j(\mu_i+\mu_j)}
\bigg(
\frac{1-e^{-(\mu_i+\mu_j)t}}{\mu_i+\mu_j}
&
\\
\quad
-
te^{-(\mu_i+\mu_j)t}
\bigg)
+
\frac{\alpha\beta\theta_i\theta_j(\lambda_0-\lambda_\infty)}{\mu_i \mu_j}
\bigg(
\frac{e^{-\mu_j t} - e^{-(\mu_i+\mu_j)t}}{\mu_i}
-
te^{-(\mu_i+\mu_j)t}
\bigg)
&
\\
\quad
-
\frac{\alpha^2\theta_i \theta_j(2\lambda_0-\lambda_\infty)}{2\mu_j(\mu_i-\mu_j)}
\left(
\frac{e^{2\mu_j t}-e^{-(\mu_i+\mu_j)t}}{\mu_i-\mu_j}
-
te^{-(\mu_i+\mu_j)t}
\right)
&
\text{if $\mu_i \ne \beta - \alpha = \mu_j$,}
\end{cases}
\end{align}
Finally,}
\edit{ the variance of the number in phase $i \in \{1, \dots, n \}$ of service is given by
\begin{align}
\Var{Q_{t,i}}
=
\begin{cases}
\frac{\lambda_\infty \theta_i}{\mu_i}\left(1-e^{-\mu_i t}\right)
+
\frac{\alpha\theta_i^2(2\beta-\alpha)\lambda_\infty}{2\mu_i(\beta-\alpha)(\mu_i+\beta-\alpha)}\left(1 - e^{-2\mu_i t}\right)
-
\Big(
\frac{\alpha\theta_i^2(2\beta-\alpha)\lambda_\infty}{(\beta-\alpha)(\mu_i+\beta-\alpha)}
&
\\
\quad
+
\frac{2\alpha\beta\theta_i^2(\lambda_0 - \lambda_\infty)}{\mu_i(\beta-\alpha)}
-
\frac{\alpha^2\theta_i^2(2\lambda_0-\lambda_\infty)}{(\beta-\alpha)(\mu_i-\beta+\alpha)}
\Big)
\frac{e^{-(\mu_i+\beta-\alpha)t}-e^{-2\mu_i t}}{\mu_i-\beta+\alpha}
+
\Big(
(\lambda_0-\lambda_\infty)\theta_i
&
\\
\quad
+ \frac{\mu_i(\lambda_0-\lambda_\infty)\theta_i}{\mu_i-\beta+\alpha} + \frac{2\alpha\beta\theta_i^2(\lambda_0-\lambda_\infty)}{\mu_i(\beta-\alpha)}
\Big)
\frac{e^{-(\beta-\alpha)t}-e^{-2\mu_i t}}{2\mu_i - \beta + \alpha}
-
\frac{\alpha^2\theta_i^2(2\lambda_0-\lambda_\infty)}{2(\beta-\alpha)(\mu_i-\beta+\alpha)^2}
&
\\
\quad
\cdot
\left(e^{-2(\beta-\alpha)t} - e^{-2\mu_i t}\right)
-
\frac{(\lambda_0-\lambda_\infty)\theta_i}{\mu_i-\beta+\alpha}\left(e^{-\mu_i t} - e^{-2\mu_i t}\right)
&
\text{if $\mu_i \ne \beta - \alpha \ne 2\mu_i$,}
\\
\frac{\lambda_\infty \theta_i}{\mu_i}\left(1-e^{-\mu_i t}\right)
+
\frac{\alpha\theta_i^2(2\beta-\alpha)\lambda_\infty}{2\mu_i(\beta-\alpha)(\mu_i+\beta-\alpha)}\left(1 - e^{-2\mu_i t}\right)
-
\Big(
\frac{\alpha\theta_i^2(2\beta-\alpha)\lambda_\infty}{(\beta-\alpha)(\mu_i+\beta-\alpha)}
&
\\
\quad
+
\frac{2\alpha\beta\theta_i^2(\lambda_0 - \lambda_\infty)}{\mu_i(\beta-\alpha)}
-
\frac{\alpha^2\theta_i^2(2\lambda_0-\lambda_\infty)}{(\beta-\alpha)(\mu_i-\beta+\alpha)}
\Big)
\frac{e^{-(\mu_i+\beta-\alpha)t}-e^{-2\mu_i t}}{\mu_i-\beta+\alpha}
+
\Big(
(\lambda_0-\lambda_\infty)\theta_i
&
\\
\quad
+
\frac{\mu_i(\lambda_0-\lambda_\infty)\theta_i}{\mu_i-\beta+\alpha} + \frac{2\alpha\beta\theta_i^2(\lambda_0-\lambda_\infty)}{\mu_i(\beta-\alpha)}
\Big)
te^{-2\mu_i t}
-
\frac{\alpha^2\theta_i^2(2\lambda_0-\lambda_\infty)}{2(\beta-\alpha)(\mu_i-\beta+\alpha)^2}
\big(e^{-2(\beta-\alpha)t}
&
\\
\quad
-
e^{-2\mu_i t}\big)
-
\frac{(\lambda_0-\lambda_\infty)\theta_i}{\mu_i-\beta+\alpha}\left(e^{-\mu_i t} - e^{-2\mu_i t}\right)
&
\text{if $2\mu_i = \beta - \alpha$,}
\\
\frac{\lambda_\infty \theta_i}{\mu_i}\left(1-e^{-\mu_i t}\right)
+
\frac{\alpha\theta_i^2(2\beta-\alpha)\lambda_\infty}{4\mu_i^3}\left(1 - e^{-2\mu_i t}\right)
-
\Big(
\frac{\alpha\theta_i^2(2\beta-\alpha)\lambda_\infty}{2\mu_i^2}
&
\\
\quad
+
\frac{2\alpha\beta\theta_i^2(\lambda_0 - \lambda_\infty)}{\mu_i^2}
\Big)
te^{-2\mu_i t}
+
\Big(
(\lambda_0-\lambda_\infty)\theta_i
+
\frac{2\alpha\beta\theta_i^2(\lambda_0-\lambda_\infty)}{\mu_i^2}
\Big)
&
\\
\quad
\cdot
\frac{e^{-\mu_i t}-e^{-2\mu_i t}}{\mu_i }
-
\frac{\alpha^2\theta_i^2(2\lambda_0-\lambda_\infty)}{2\mu_i }
t^2 e^{-2\mu_i t}
+
(\lambda_0-\lambda_\infty)\theta_i
\bigg(
\frac{t e^{-\mu_i t}}{\mu_i}
&
\\
\quad
+
\frac{e^{-2\mu_i t} - e^{-\mu_i t}}{\mu_i^2}
\bigg)
&
\text{if $\mu_i = \beta - \alpha$,}
\end{cases}
\end{align}
}
where all $t \geq 0$.
\end{corollary}

\edit{We now note that in both Corollary~\ref{hawkqueueerlang} and Corollary~\ref{hawkqueuehyper}, taking $n=1$ reduces the setting to exponential service. We demonstrate the simplification and use of the singe-phase expressions in finding the auto-covariance of the $Hawkes/M/\infty$ queue, shown in Proposition~\ref{hawkesminfautoco}. We also note that these findings} compare quite nicely \edit{to simulations} in numerical demonstrations. In Subsection~\ref{subsec_sim}, we provide several example figures of these equations and their simulated counterparts.\\

\edit{Now that we have investigated the transient behavior of the $Hawkes/PH/\infty$ queue for a variety of settings it is natural to consider the behavior of the system in steady-state. This, along with the behavior of the system with an unstable arrival process, is the focus of the next subsection.}

\subsection{Limiting Behavior of the $Hawkes/PH/\infty$ Queue}\label{sslimits}

\edit{In many situations, the steady-state behavior of a queueing system may be of particular interest. With that in mind, we now investigate the mean and variance of the $Hawkes/PH/\infty$ queue as time goes to infinity.}

\begin{corollary}\label{hawkqueuelimit}
Consider a queueing system with arrivals occurring in accordance to a Hawkes process $\left(\lambda_t, N_t\right)$ with dynamics given in Equation~\ref{hawkdyn} and phase-type distributed service. Let $S \in \mathbb{R}^{n \times n}$ be the sub-generator matrix for the transient states in the phase-distribution CTMC and let $\theta \in [0,1]^n$ be the initial distribution for arrivals to these states. \edit{Then, the steady-state mean number in each phase of service is given by the vector}
\begin{align}
&
%
%
%
%
%
%
\mathcal{Q}_\infty
\equiv
\lim_{t \to \infty}
\E{Q_t}
=
\lambda_\infty\hspace{-.1cm} \left(- S^\mathrm{T}\right)^{-1}\theta
\end{align}
where $\lambda_\infty = \frac{\beta \lambda^*}{\beta - \alpha}$. Further, \edit{the vector of steady-state covariances between the number in each phase of service and the intensity is}
\vspace{-.25cm}
\begin{align}
&
%
%
%
%
%
\mathcal{C}_\infty
\equiv
\lim_{t \to \infty}
\Cov{\lambda_t, Q_t}
=
\lambda_\infty\frac{\alpha(2\beta - \alpha)}{2(\beta - \alpha)}
\left((\beta - \alpha)I - S^\mathrm{T}\right)^{-1}\theta
\, .
\end{align}
Finally, \edit{the matrix of steady-state covariances between each phase of service $\mathrm{lim}_{t\to\infty} \Cov{Q_t, Q_t}$, denoted $\mathcal{V}_\infty$, is given by the solution to the Lyapunov equation
\begin{align}
S^\T \mathcal{V}_\infty + \mathcal{V}_\infty S + \mathcal{M} = 0
\end{align}
where $\mathcal{M} = \theta \mathcal{C}_\infty^\T + \mathcal{C}_\infty \theta^\T - S^\T \diag{\mathcal{Q}_\infty}  - \diag{\mathcal{Q}_\infty} S$. If $S$ is symmetric, then $\mathcal{V}_\infty  =-\frac{1}{2}S^{-1}\mathcal{M}$.}
\end{corollary}

\begin{proof}
The proof follows by either taking the limit of the equations in Theorem~\ref{hawkqueuegeneral} or setting the corresponding differential equations to 0 and finding the equilibrium solution.
\end{proof}

\edit{\begin{remark}\label{hawkqueuelimitremark}
We note that in steady-state the invertibility conditions from Theorem~\ref{hawkqueuegeneral} are no longer necessary. We can further observe that these equations reveal an interesting relationship among these steady-state values for the case of single phase service. For $\mu$ as the rate of exponential service, Corollary~\ref{hawkqueuelimit} yields
\begin{equation}\label{hminfsteady}
\mathcal{V}_\infty = \mathcal{Q}_\infty + \frac{1}{\mu}\mathcal{C}_\infty  = \frac{\lambda_\infty}{\mu}\left(1 + \frac{\alpha(2\beta - \alpha)}{2(\beta - \alpha)(\mu + \beta - \alpha)}\right) .
\end{equation}
Thus, we have that the steady-state variance of the number in system for the $Hawkes/M/\infty$ queue is \minrev{equal to} the mean number in system plus the expected service duration times the steady-state covariance between the number in system and the intensity. Thus this provides an explicit contrast with Poisson-driven queues, as the steady-state distribution of a $M/M/\infty$ system is known to be Poisson distributed with rate equal to the steady-state mean number in system. This implies that the steady-state variance for such a queue is equal to its steady-state mean, unlike the relationship we observe for the $Hawkes/M/\infty$ system in Equation~\ref{hminfsteady}.\\ \end{remark}}

\edit{However, as we have noted, if $\alpha \geq \beta$ the Hawkes process is unstable and so steady-state analysis of the queue will not apply. Thus, in this scenario we instead investigate the transient behavior of the mean of the queue under the unstable arrival process. }

\edit{
\begin{corollary}
Consider a queueing system with arrivals occurring in accordance to a Hawkes process $\left(\lambda_t, N_t\right)$ with dynamics given in Equation~\ref{hawkdyn} with $\alpha \geq \beta$ and phase-type distributed service. Let $S \in \mathbb{R}^{n \times n}$ be the sub-generator matrix for the transient states in the phase-distribution CTMC and let $\theta \in [0,1]^n$ be the initial distribution for arrivals to these states. Then the vector of mean number in service in each phase of service is given by
\begin{align}
\E{Q_t}
&=
\left((\alpha - \beta)I - S^\T\right)^{-1}\left(e^{(\alpha - \beta)t} I - e^{S^\T t}\right)\theta \left(\frac{\beta\lambda^*}{\alpha - \beta} + \lambda_0\right)
\nonumber
\\&\quad
+
(S^\T)^{-1}\left(I - e^{S^\T t}\right)\theta \frac{\beta \lambda^*}{\alpha - \beta}
\intertext{when $\alpha > \beta$ and
}
\E{Q_t}
&=
-(S^\T)^{-1}\left(I - e^{S^\T t}\right)\theta(\lambda_0 -\beta\lambda^*) - (S^\T)^{-1}\theta\beta\lambda^* t
\end{align}
when $\alpha = \beta$.
\end{corollary}
}

\subsection{Auto-covariance of the $Hawkes/PH/\infty$ Queue}\label{ssautoco}

\edit{We now consider the auto-covariance of the number in this queueing system, $Q_t \in \mathbb{R}^n$. Analogous to the auto-covariance for the number of arrivals from the Hawkes process discussed in Subsection~\ref{hawkesDinf}, this matrix quantity is defined as
$$
\Cov{Q_t, Q_{t-\tau}}
=
\E{Q_t Q_{t-\tau}^\T} - \E{Q_t}\E{Q_{t-\tau}}^\T
$$
where $t \geq \tau \geq 0$ and otherwise the covariance is equal to 0. For an infinite server queue with Hawkes process arrivals and phase-type distributed service, the findings in Subsection~\ref{ssmeandyn} \minrev{give} us expressions for $\E{Q_t}$ and $\E{Q_{t-\tau}}$. \minrev{Let $\mathcal{F}_s$ be the filtration of the queueing system, the Hawkes process, and the intensity at time $s \geq 0$. Then, assuming $S + (\beta - \alpha)I$ is invertible, conditional expectation yields}
\begin{align*}
\E{Q_t Q_{t-\tau}^\T}
&=
\E{\E{Q_t \mid F_{t-\tau}} Q_{t-\tau}^\T}
\\
&
=
\mathrm{E}\Big[
\Big(
\lambda_\infty \left(- S^\mathrm{T}\right)^{-1}\big(I - e^{S^\mathrm{T} \tau}\big)\theta
-
\left( \lambda_{t-\tau} - \lambda_\infty \right)\left(S^\mathrm{T} + (\beta - \alpha)I\right)^{-1}\big(e^{-(\beta - \alpha) \tau} I
\\
&
\quad
- e^{S^\mathrm{T} \tau}\big)\theta
+ e^{S^\T \tau}Q_{t - \tau}
\Big)
Q_{t-\tau}^\T
\Big]
\\
&
=
\lambda_\infty \left(- S^\mathrm{T}\right)^{-1}\big(I - e^{S^\mathrm{T} \tau}\big)\theta
\E{Q_{t-\tau}}^\T
-
\left(S^\mathrm{T} + (\beta - \alpha)I\right)^{-1}\big(e^{-(\beta - \alpha) \tau} I - e^{S^\mathrm{T} \tau}\big)\theta
\\
&
\quad
\cdot
\left(
\E{\lambda_{t-\tau} Q_{t-\tau}^\T}
-
\lambda_\infty \E{Q_{t-\tau}}^\T
\right)
+ e^{S^\T \tau}\E{Q_{t-\tau}Q_{t-\tau}^\T}
\end{align*}
by application of the expression for the vector of the mean number in each phase given in  Theorem~\ref{hawkqueuegeneral}, modified to start at time $t - \tau$. Upon recognizing that $\E{\lambda_{t-\tau} Q_{t-\tau}^\T} = \Cov{\lambda_{t-\tau},Q_{t-\tau}} + \E{\lambda_{t-\tau}} \E{Q_{t-\tau}}^\T$ and $\E{Q_{t-\tau}Q_{t-\tau}^\T} = \Cov{Q_{t-\tau}, Q_{t-\tau}^\T} + \E{Q_{t-\tau}}\E{Q_{t-\tau}}^\T$, we have that
\begin{align}\label{autocorphaseeq}
&\Cov{Q_{t}, Q_{t - \tau}}
=
\lambda_\infty\left(-S^\T \right)^{-1}\left(I - e^{S^\T \tau}\right)\theta\E{Q_{t-\tau}}^\T
-
\left(S^\T + (\beta - \alpha)I\right)^{-1} \left(e^{-(\beta - \alpha)\tau} I - e^{S^\T \tau} \right)
\nonumber
\\
&
\qquad
\cdot
\theta \left(\Cov{\lambda_{t - \tau}, Q_{t-\tau}}^\T + \E{\lambda_{t-\tau}}\E{Q_{t-\tau}}^\T - \lambda_\infty \E{Q_{t-\tau}}^\T\right)
+
e^{S^\T \tau}
\Cov{Q_{t - \tau},Q_{t-\tau}}
\nonumber
\\
&
\qquad
+
\left(
e^{S^\T \tau}
\E{Q_{t-\tau}}
-
\E{Q_t}
\right)
\E{Q_{t-\tau}}^\T
\end{align}
and that each term in this expression can be calculated by applying Theorem~\ref{hawkqueuegeneral}. An explicit expression for the transient auto-covariance using this approach is given in the Appendix. In this section we give an explicit expression for the auto-covariance of the $Hawkes/M/\infty$ queue. In this setting with service rate $\mu$, the same approach as above yields
\begin{align}\label{autocorexpeq1}
\Cov{Q_t, Q_{t - \tau}}
&=
\frac{\lambda_\infty}{\mu}\left(1 - e^{-\mu \tau}\right)\E{Q_{t-\tau}}
+
e^{-\mu \tau}
\Var{Q_{t-\tau}}
+
\Cov{\lambda_{t-\tau}, Q_{t-\tau}}\frac{e^{-(\beta - \alpha)\tau} - e^{-\mu \tau}}{\mu - \beta + \alpha}
\nonumber
\\&\quad
+
 \left(\E{\lambda_{t-\tau}} - \lambda_\infty\right)\E{Q_{t-\tau}}
\frac{e^{-(\beta - \alpha)\tau} - e^{-\mu \tau}}{\mu - \beta + \alpha}
+
e^{-\mu \tau}
\E{Q_{t-\tau}}^2
-
\E{Q_t}\E{Q_{t-\tau}}
\end{align}
when $\mu \ne \beta - \alpha$ and
\begin{align}\label{autocorexpeq2}
\Cov{Q_t, Q_{t - \tau}}
&=
\frac{\lambda_\infty}{\mu}\left(1 - e^{-\mu \tau}\right)\E{Q_{t-\tau}}
+
e^{-\mu \tau}
\Var{Q_{t-\tau}}
+
\Cov{\lambda_{t-\tau}, Q_{t-\tau}}\tau e^{-\mu \tau}
\nonumber
\\&\quad
+
 \left(\E{\lambda_{t-\tau}} - \lambda_\infty\right)\E{Q_{t-\tau}}
\tau e^{-\mu \tau}
+
e^{-\mu \tau}
\E{Q_{t-\tau}}^2
-
\E{Q_t}\E{Q_{t-\tau}}
\end{align}
when $\mu = \beta - \alpha$, where each of these makes use of Corollary~\ref{hawkqueuehyper} with $n = 1$, $\theta_1 = 1$, and $\mu_i = \mu$. These expressions are made explicit in the following proposition.
}

\begin{proposition}\label{hawkesminfautoco}
Consider a queueing system with arrivals occurring in accordance to a Hawkes process $\left(\lambda_t, N_t\right)$ with dynamics given in Equation~\ref{hawkdyn} with $\alpha < \beta$ and exponentially distributed service with rate $\mu$. Then, for $t \geq \tau \geq 0$ the auto-covariance of the number in system is
\vspace{-.1in}
\begin{align}
&
\Cov{Q_t, Q_{t-\tau}}
=
\frac{\lambda_\infty}{\mu}\left(1 - e^{-\mu \tau}\right)
\left(
\frac{\lambda_\infty}{\mu}\left(1 - e^{-\mu (t-\tau)}\right)
+
\frac{\lambda_0 - \lambda_\infty }{\mu - \beta + \alpha}\left(e^{-(\beta - \alpha) (t-\tau)}  - e^{- \mu (t-\tau)}\right)
\right)
+
\frac{\lambda_\infty }{\mu}
\nonumber
\\
&
\quad
\cdot
\left(e^{-\mu \tau}-e^{-\mu t}\right)
+
\frac{\alpha(2\beta-\alpha)\lambda_\infty}{2\mu(\beta-\alpha)(\mu+\beta-\alpha)}\left(e^{-\mu \tau} - e^{-\mu (2 t - \tau)}\right)
-
\bigg(
\frac{\alpha(2\beta-\alpha)\lambda_\infty}{(\beta-\alpha)(\mu+\beta-\alpha)}
+
\frac{2\alpha\beta(\lambda_0 - \lambda_\infty)}{\mu(\beta-\alpha)}
\nonumber
\\
&
\quad
-
\frac{\alpha^2(2\lambda_0-\lambda_\infty)}{(\beta-\alpha)(\mu-\beta+\alpha)}
\bigg)
\frac{e^{-(\mu+\beta-\alpha)t + (\beta-\alpha)\tau}-e^{-\mu (2t - \tau)}}{\mu-\beta+\alpha}
+
\bigg(
\lambda_0-\lambda_\infty
+
\frac{\mu(\lambda_0-\lambda_\infty)}{\mu-\beta+\alpha} + \frac{2\alpha\beta(\lambda_0-\lambda_\infty)}{\mu(\beta-\alpha)}
\bigg)
\nonumber
\\
&
\quad
\cdot
h(t-\tau)e^{-\mu \tau}
-
\frac{\alpha^2(2\lambda_0-\lambda_\infty)}{2(\beta-\alpha)(\mu-\beta+\alpha)^2}
\big(e^{-2(\beta-\alpha)t - (\mu-2\beta+2\alpha)\tau} - e^{-\mu (2t - \tau)}\big)
-
\frac{\lambda_0-\lambda_\infty}{\mu-\beta+\alpha}\big(e^{-\mu t}
\nonumber
\\
&
\quad
 - e^{-\mu (2t - \tau)}\big)
+
e^{-\mu \tau}
\bigg(
\frac{\lambda_\infty}{\mu}\left(1 - e^{-\mu (t-\tau)}\right)
+
\frac{\lambda_0 - \lambda_\infty }{\mu - \beta + \alpha}
\left(e^{-(\beta - \alpha) (t-\tau)}  - e^{- \mu (t-\tau)}\right)
\bigg)^2
+
\frac{e^{-(\beta - \alpha)\tau} - e^{-\mu \tau}}{\mu - \beta + \alpha}
\nonumber
\\
&
\quad
\cdot
\Bigg(
\frac{\alpha  (2\beta - \alpha)\lambda_\infty}{2(\beta - \alpha)(\mu + \beta - \alpha)}
\left(1 - e^{-(\mu + \beta - \alpha)(t-\tau)}\right)
+
\frac{\alpha\beta(\lambda_0 - \lambda_\infty)}{\mu(\beta - \alpha)}
\big(
e^{-(\beta - \alpha)(t-\tau)}
-
e^{-(\mu + \beta - \alpha)(t-\tau)}
\big)
\nonumber
\\
&
\quad
-
\frac{\alpha^2  (2\lambda_0 - \lambda_\infty)}{2(\beta - \alpha)(\mu - \beta + \alpha)}
 \big(e^{-2(\beta - \alpha)(t-\tau)}
 - e^{-(\mu + \beta - \alpha)(t-\tau)}\big)
\Bigg)
+
(\lambda_0 - \lambda_\infty)
   \frac{e^{-(\beta - \alpha)\tau} - e^{-\mu \tau}}{\mu - \beta + \alpha}
\Bigg(
\frac{\lambda_0 - \lambda_\infty }{\mu - \beta + \alpha}
\nonumber
\\
&
\quad
\cdot
\left(e^{-2(\beta - \alpha)(t-\tau)}  - e^{-(\mu + \beta - \alpha)(t-\tau)}\right)
+
\frac{\lambda_\infty}{\mu}\left(e^{-(\beta - \alpha)(t-\tau)} - e^{-(\mu + \beta - \alpha)(t-\tau)}\right)
\Bigg)
-
\bigg(
\frac{\lambda_\infty}{\mu}\left(1 - e^{-\mu t}\right)
+
\frac{\lambda_0 - \lambda_\infty }{\mu - \beta + \alpha}
\nonumber
\\
&
\quad
\cdot
\big(e^{-(\beta - \alpha) t} - e^{- \mu t}
\big)
\bigg)
\bigg(
\frac{\lambda_\infty}{\mu}\left(1 - e^{-\mu (t-\tau)}\right)
+
\frac{\lambda_0 - \lambda_\infty }{\mu - \beta + \alpha}
\left(e^{-(\beta - \alpha) (t-\tau)}  - e^{- \mu (t-\tau)}
\right)
\bigg)
\end{align}
when $\mu \ne \beta - \alpha$ and
\begin{align}
&
\Cov{Q_t, Q_{t-\tau}}
=
\frac{\lambda_\infty}{\mu}\left(1 - e^{-\mu \tau}\right)
\left(
\frac{\lambda_\infty}{\mu}\left(1 - e^{-\mu (t-\tau)}\right)
+
\left(\lambda_0 - \lambda_\infty \right) (t-\tau) e^{-\mu (t-\tau)}
\right)
+
\frac{\lambda_\infty}{\mu}\left(e^{-\mu \tau}-e^{-\mu t}\right)
\nonumber
\\
&
\quad
+
\frac{\alpha(2\beta-\alpha)\lambda_\infty}{4\mu^3}\left(e^{-\mu \tau} - e^{-\mu (2t-\tau)}\right)
-
\Big(
\frac{\alpha(2\beta-\alpha)\lambda_\infty}{2\mu^2}
+
\frac{2\alpha\beta(\lambda_0 - \lambda_\infty)}{\mu^2}
\Big)
(t-\tau)e^{-\mu(2t- \tau)}
+
\Big(
\lambda_0-\lambda_\infty
\nonumber
\\
&
\quad
+
\frac{2\alpha\beta(\lambda_0-\lambda_\infty)}{\mu^2}
\Big)
\frac{e^{-(\beta-\alpha)t -(\mu-\beta+\alpha)\tau}-e^{-\mu(2 t-\tau)}}{\mu }
-
\frac{\alpha^2(2\lambda_0-\lambda_\infty)}{2\mu }
(t-\tau)^2 e^{-\mu (2t-\tau)}
+
(\lambda_0-\lambda_\infty)
\nonumber
\\
&
\quad
\cdot
\bigg(
\frac{(t-\tau) e^{-\mu t}}{\mu}
+
\frac{e^{-\mu (2t-\tau)} - e^{-\mu t}}{\mu^2}
\bigg)
+
e^{-\mu \tau}
\left(
\frac{\lambda_\infty}{\mu}\left(1 - e^{-\mu (t-\tau)}\right)
+
\left(\lambda_0 - \lambda_\infty \right) (t-\tau) e^{-\mu (t-\tau)}
\right)^2
\nonumber
\\
&
\quad
+
\bigg(
 \frac{\alpha (2\mu + \alpha)\lambda_\infty}{4\mu^2}\left(1 - e^{-2\mu (t-\tau)}\right)
+
\frac{\alpha\beta(\lambda_0 - \lambda_\infty)}{\mu^2} \big(e^{-\mu (t-\tau)}
- e^{-2\mu (t-\tau)}\big)
-
\frac{\alpha^2 (2\lambda_0 - \lambda_\infty)}{2\mu}
(t-\tau)e^{-2\mu (t - \tau)}
\bigg)
\nonumber
\\
&
\quad
\cdot
\tau e^{-\mu\tau}
+
  \tau
(\lambda_0 - \lambda_\infty)e^{- \mu t}
\Bigg(
\frac{\lambda_\infty}{\mu}\left(1 - e^{-\mu (t-\tau)}\right)
+
\left(\lambda_0 - \lambda_\infty \right) (t-\tau) e^{-\mu (t-\tau)}
\Bigg)
-
\bigg(
\frac{\lambda_\infty}{\mu}\left(1 - e^{-\mu t}\right)
\nonumber
\\
&
\quad
+
\left(\lambda_0 - \lambda_\infty \right) t e^{-\mu t}
\bigg)
\bigg(
\frac{\lambda_\infty}{\mu}\left(1 - e^{-\mu (t-\tau)}\right)
+
\left(\lambda_0 - \lambda_\infty \right) (t-\tau) e^{-\mu (t-\tau)}
\bigg)
\end{align}
when $\mu = \beta - \alpha$, where $h(s) = s e^{-2\mu s}$ if $2\mu = \beta - \alpha$ and $h(s) = \frac{e^{-(\beta - \alpha)s} - e^{-2\mu s}}{2\mu - \beta + \alpha}$ if $2\mu \ne \beta - \alpha$ for all $s \geq 0$.
\end{proposition}
\begin{proof}
The stated forms follow by simplification of the expressions in Corollary~\ref{hawkqueuehyper}, yielding
\begin{align*}
\E{Q_t}
=
\frac{\lambda_\infty}{\mu}\left(1 - e^{-\mu t}\right)
+
\frac{\lambda_0 - \lambda_\infty }{\mu - \beta + \alpha}\left(e^{-(\beta - \alpha) t}  - e^{- \mu t}\right)
\end{align*}
for the mean of the $Hawkes/M/\infty$ queue,
\begin{align*}
\Cov{\lambda_t,Q_t}
&
=
\frac{\alpha  (2\beta - \alpha)\lambda_\infty}{2(\beta - \alpha)(\mu + \beta - \alpha)}\left(1 - e^{-(\mu + \beta - \alpha)t}\right)
+
\frac{\alpha\beta(\lambda_0 - \lambda_\infty)}{\mu(\beta - \alpha)} \big(e^{-(\beta - \alpha)t}
- e^{-(\mu + \beta - \alpha)t}\big)
\\
&
\quad
-
\frac{\alpha^2  (2\lambda_0 - \lambda_\infty)}{2(\beta - \alpha)(\mu - \beta + \alpha)}
 \left(e^{-2(\beta - \alpha)t} - e^{-(\mu + \beta - \alpha)t}\right)
\end{align*}
for the covariance between the queue and the intensity, and
\begin{align*}
&
\Var{Q_{t}}
=
\frac{\lambda_\infty }{\mu}\left(1-e^{-\mu t}\right)
+
\frac{\alpha(2\beta-\alpha)\lambda_\infty}{2\mu(\beta-\alpha)(\mu+\beta-\alpha)}\left(1 - e^{-2\mu t}\right)
-
\bigg(
\frac{\alpha(2\beta-\alpha)\lambda_\infty}{(\beta-\alpha)(\mu+\beta-\alpha)}
+
\frac{2\alpha\beta(\lambda_0 - \lambda_\infty)}{\mu(\beta-\alpha)}
\nonumber
\\
&
\quad
-
\frac{\alpha^2(2\lambda_0-\lambda_\infty)}{(\beta-\alpha)(\mu-\beta+\alpha)}
\bigg)
\frac{e^{-(\mu+\beta-\alpha)t}-e^{-2\mu t}}{\mu-\beta+\alpha}
+
\bigg(
\lambda_0-\lambda_\infty
+
\frac{\mu(\lambda_0-\lambda_\infty)}{\mu-\beta+\alpha}
 +
  \frac{2\alpha\beta(\lambda_0-\lambda_\infty)}{\mu(\beta-\alpha)}
\bigg)
\nonumber
\\
&
\quad
\cdot
\frac{e^{-(\beta-\alpha)t}-e^{-2\mu t}}{2\mu - \beta + \alpha}
-
\frac{\alpha^2(2\lambda_0-\lambda_\infty)}{2(\beta-\alpha)(\mu-\beta+\alpha)^2}
\left(e^{-2(\beta-\alpha)t} - e^{-2\mu t}\right)
-
\frac{\lambda_0-\lambda_\infty}{\mu-\beta+\alpha}\left(e^{-\mu t} - e^{-2\mu t}\right)
\end{align*}
for the variance of the queue, all in the case where $\mu \ne \beta - \alpha$. The remaining derivation follows directly from substitution of these functions and the corresponding expressions for remaining cases and epochs into Equations \ref{autocorexpeq1} and \ref{autocorexpeq2}.
\end{proof} 

\vspace{-.05in}
\begin{figure}[h]
	\begin{center}
\includegraphics[width=.45\textwidth]{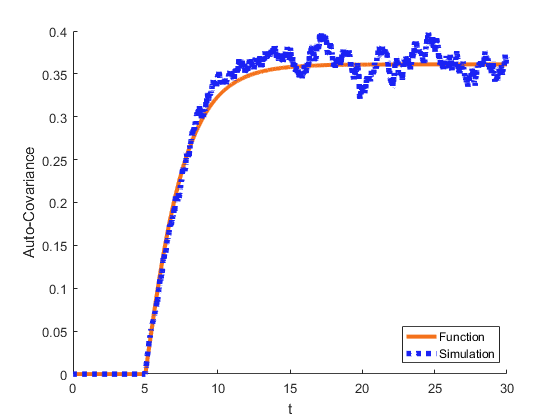}\includegraphics[width=.45\textwidth]{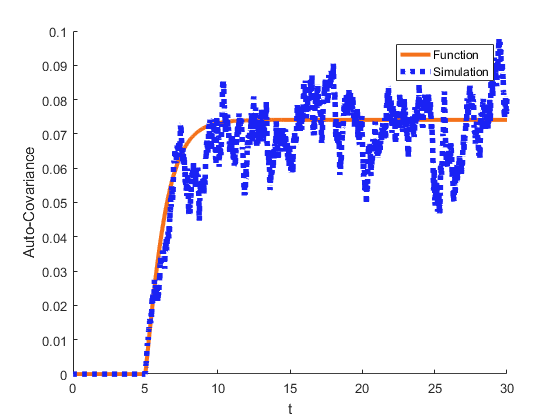}
\caption{Auto-covariance of the ${Hawkes / M / \infty}$ Queue for ${\tau = 5}$, where ${\alpha = \frac 3 4}$, ${\beta = \frac 5 4}$, ${\lambda^* = \mu = 1}$ (left) and  ${\alpha = 1}$, ${\beta = 2}$, ${\lambda^* = \mu = 1}$ (right).} \label{queueautoco}
\end{center}
\end{figure}

In Figure~\ref{queueautoco} the expressions in Proposition~\ref{hawkesminfautoco} are compared to simulations, based on 100,000 replications.

\subsection{Generating Functions for the $Hawkes/PH/\infty$ Queue}\label{ssgenfunc}

To complement these findings, we also derive a form for the moment generating function for a general queueing system driven by a Hawkes process.

\begin{theorem}\label{MGF}
Consider a queueing system with arrivals occurring in accordance to a Hawkes process $\left(\lambda_t, N_t\right)$ with dynamics given in Equation~\ref{hawkdyn} \edit{ with $\alpha < \beta$} and phase-type distributed service. Let $\delta \in \mathbb{R}^{n+1}_+$ and let $M(\delta , t) = M(\delta_0 , \dots , \delta_n, t) = \E{e^{\delta_0 \lambda_t + \sum_{i=1}^n \delta_i Q_{t,i}}}$. Then, the moment generating function for the queueing system $M(\delta, t)$ is given by the solution to the following partial differential equation,
\begin{align}
\frac{\partial M(\delta, t)}{\partial t}  &= \delta_0 \beta \lambda^* M(\delta, t) + \left(\sum_{i=1}^n \theta_i(e^{\delta_0 \alpha + \delta_i} - 1) - \delta_0\beta \right)\frac{\partial M(\delta, t)}{\partial \delta_0}
\\&+ \sum_{i=1}^n\left(\mu_{i0}(e^{-\delta_i}-1) + \sum_{k \ne i} \mu_{ik}(e^{\delta_k-\delta_i}-1)\right) \frac{\partial M(\delta, t)}{\partial \delta_i}  \, .\nonumber
\end{align}
\end{theorem}

\begin{proof}
This proof makes use of techniques similar to the prior theorems, and so we omit the preceding infinitesimal generator steps. Note that $\frac{\partial M(\delta, t)}{\partial t} = \frac{\partial}{\partial t}\E{e^{\delta_0 \lambda_t + \sum_{i=1}^n \delta_i Q_{t,i}}} $. From this, we start with the following.
\begin{align*}
\frac{\partial M(\delta, t)}{\partial t}
&= \mathrm{E}\Bigg[\delta_0 \beta (\lambda^* - \lambda_t) e^{\delta_0 \lambda_t + \sum_{i=1}^n \delta_i Q_{t,i}}
+ \sum_{j=1}^n \lambda_t \theta_j \left(e^{\delta_0 (\lambda_t + \alpha) + \sum_{k \ne j} \delta_k Q_{t,k} + \delta_j (Q_{t,j} + 1)} - e^{\delta \lambda_t + \sum_{i=1}^n \delta_i Q_{t,i}} \right)
\\&+ \sum_{k=1}^n \sum_{j \ne k} \mu_{jk}Q_{t,j}\left(e^{\delta_0 \lambda_t + \sum_{l \ne j \wedge l \ne k} \delta_l Q_{t,l} + \delta_j(Q_{t,j} - 1) + \delta_k(Q_{t,k} + 1)} - e^{\delta_0 \lambda_t + \sum_{i=1}^n \delta_i Q_{t,i}} \right)
\\&+ \sum_{j=1}^n \mu_{j0}Q_{t,j}\left(e^{\delta_0 \lambda_t + \sum_{k \ne j} \delta_k Q_{t,k} + \delta_j(Q_{t,k} -1 )}-e^{\delta_0 \lambda_t + \sum_{i=1}^n \delta_i Q_{t,i}}\right) \Bigg]
\intertext{Now, we distribute terms and notice that the difference of exponentials here can be expressed as the following products.}
\frac{\partial M(\delta, t)}{\partial t}&= \mathrm{E}\Bigg[\delta_0 \beta \lambda^* e^{\delta_0 \lambda_t + \sum_{i=1}^n \delta_i Q_{t,i}}  - \delta_0 \beta\lambda_t e^{\delta_0 \lambda_t + \sum_{i=1}^n \delta_i Q_{t,i}}
+ \sum_{j=1}^n \lambda_t \theta_j e^{\delta_0 \lambda_t + \sum_{i=1}^n \delta_i Q_{t,i}} \left(e^{\delta_0 \alpha + \delta_j} - 1\right)
\\&+ \sum_{k=1}^n \sum_{j \ne k} \mu_{jk}Q_{t,j}e^{\delta_0 \lambda_t + \sum_{i=1}^n \delta_i Q_{t,i}} \left(e^{\delta_k - \delta_j} - 1\right)
+ \sum_{j=1}^n \mu_{j0}Q_{t,j}e^{\delta_0 \lambda_t + \sum_{i=1}^n \delta_i Q_{t,i}} \left(e^{-\delta_j} - 1\right) \Bigg]
\intertext{Here, we can now use linearity of expectation and group like terms.}
\frac{\partial M(\delta, t)}{\partial t}&= \delta_0 \beta \lambda^* \E{e^{\delta_0 \lambda_t + \sum_{i=1}^n \delta_i Q_{t,i}}}
+ \left( \sum_{j=1}^n \theta_j(e^{\delta_0 \alpha + \delta_j} - 1) - \delta_0 \beta\right) \E{\lambda_t e^{\delta_0 \lambda_t + \sum_{i=1}^n \delta_i Q_{t,i}}}
\\&+ \sum_{j=1}^n \left(\mu_{j0}(e^{-\delta_j}-1) + \sum_{k \ne j} \mu_{jk}(e^{\delta_k - \delta_j} - 1) \right) \E{Q_{t,j} e^{\delta_0 \lambda_t + \sum_{i=1}^n \delta_i Q_{t,i}}}
\end{align*}
Finally, here we recognize the form of partial derivatives of $M(\delta, t)$ in each expectation, and so we \edit{ simplify to the desired result.}
\end{proof}

We can use this to also find a partial differential equation for the natural logarithm of the moment generating function. This is called the cumulant moment generating function, as the derivative of this function yields the cumulant moments.

\begin{corollary}\label{cgfPDE}
Consider a queueing system with arrivals occurring in accordance to a Hawkes process $\left(\lambda_t, N_t\right)$ with dynamics given in Equation~\ref{hawkdyn} and phase-type distributed service. Let $\delta \in \mathbb{R}^{n+1}_+$ and let $G(\delta , t) = G(\delta_0 , \dots , \delta_n, t) = \log\left(\E{e^{\delta_0 \lambda_t + \sum_{i=1}^n \delta_i Q_{t,i}}}\right)$. Then, the cumulant moment generating function for the queueing system $G(\delta, t)$ is given by the solution to the following partial differential equation,
\begin{align}
\frac{\partial G(\delta, t)}{\partial t}  &= \delta_0 \beta \lambda^* + \left( \sum_{i=1}^n \theta_i(e^{\delta_0 \alpha + \delta_i} - 1) - \delta_0\beta \right)\frac{\partial G(\delta, t)}{\partial \delta_0}
\\&+ \sum_{i=1}^n\left(\mu_{i0}(e^{-\delta_i}-1) + \sum_{k \ne i} \mu_{ik}(e^{\delta_k-\delta_i}-1)\right) \frac{\partial G(\delta, t)}{\partial \delta_i}  \, .\nonumber
\end{align}
\end{corollary}
\begin{proof}
To begin, we see from the derivative of the logarithm and the chain rule that
\begin{align*}
\frac{\partial G(\delta, t)}{\partial t}
&=
\frac{\partial }{\partial t}
\log\left(\E{e^{\delta_0 \lambda_t + \sum_{i=1}^n \delta_i Q_{t,i}}}\right)
=
\frac{\frac{\partial }{\partial t}
\E{e^{\delta_0 \lambda_t + \sum_{i=1}^n \delta_i Q_{t,i}}}}
{
\E{e^{\delta_0 \lambda_t + \sum_{i=1}^n \delta_i Q_{t,i}}}
}
\intertext{and here we can recognize that these expectations are the moment generating function. Using Theorem~\ref{MGF}, we have}
\frac{\partial G(\delta, t)}{\partial t}
&=
 \delta_0 \beta \lambda^*
 +
 \left(\sum_{i=1}^n \theta_i(e^{\delta_0 \alpha + \delta_i} - 1) - \delta_0\beta\right)\frac{\frac{\partial }{\partial \delta_0}
\E{e^{\delta_0 \lambda_t + \sum_{i=1}^n \delta_i Q_{t,i}}}}
{
\E{e^{\delta_0 \lambda_t + \sum_{i=1}^n \delta_i Q_{t,i}}}
}
\\&\quad
+
\sum_{i=1}^n\left(\mu_{i0}(e^{-\delta_i}-1) + \sum_{k \ne i} \mu_{ik}(e^{\delta_k-\delta_i}-1)\right) \frac{\frac{\partial }{\partial \delta_i}
\E{e^{\delta_0 \lambda_t + \sum_{i=1}^n \delta_i Q_{t,i}}}}
{
\E{e^{\delta_0 \lambda_t + \sum_{i=1}^n \delta_i Q_{t,i}}} }.
\end{align*}
Now we recognize that $\frac{\frac{\partial }{\partial \delta_i}
\E{e^{\delta_0 \lambda_t + \sum_{i=1}^n \delta_i Q_{t,i}}}}
{
\E{e^{\delta_0 \lambda_t + \sum_{i=1}^n \delta_i Q_{t,i}}} }= \frac{\partial G(\delta, t)}{\partial \delta_i} $, and so we have the stated result.
\end{proof}

Comparing these two partial differential equations, we see that the expression for the cumulant moment generating function only depends on the partial derivatives, not on the function itself.  In some cases the cumulant moment generating function is better since \edit{it} directly will compute the variance, skewness, and higher order cumulants directly without having to know the relationships between cumulants and moments.  Moreover, the \edit{cumulant} moments have shift and scale invariance properties, which are often desired. \edit{The PDE in Corollary~\ref{cgfPDE} produces a form that \minrev{provides} insight to the \minrev{solution} through use of the method of characteristics, which we now show in the following theorem.}

\edit{
\begin{theorem}
Consider a queueing system with arrivals occurring in accordance to a Hawkes process $\left(\lambda_t, N_t\right)$ with dynamics given in Equation~\ref{hawkdyn} and phase-type distributed service with transient state sub-generator matrix $S \in \mathbb{R}^{n \times n}$. Let $\delta \in \mathbb{R}^{n+1}_+$ and let $G(\delta , t) = G(\delta_0 , \dots , \delta_n, t) = \log\left(\E{e^{\delta_0 \lambda_t + \sum_{i=1}^n \delta_i Q_{t,i}}}\right)$. Then, the cumulant moment generating function for the queueing system $G(\delta, t)$ is given by
\begin{align}
G(\delta, t)
&
=
\beta \lambda^* \int_0^t  h(z) \mathrm{d}z
+
h(0)\lambda_0
\end{align}
where $h(z)$ is the solution to the ordinary differential equation
$$
\updot{h}(z) = 1 - e^{\alpha h(z)}\theta^\T \left(\onevec{} + e^{-S(z-t)}\left(e^{\diag{\delta}} - I\right)\onevec{} \right) + \beta h(z)
$$
with initial value $h(t) = \delta_0$.
\end{theorem}
\begin{proof}
We proceed by the method of characteristics for the PDE given in Corollary~\ref{cgfPDE}. To do so, let $z$ be a parametrization variable and let $\Delta_0, \Delta_1, \dots, \Delta_n$ be characteristics variables. From recognizing the linearity of the PDE, we see that we can implement the method of characteristics by setting $\updot{\Delta}_i(z) := \frac{\mathrm{d}\Delta_i(z)}{\mathrm{d}z}$ equal to the function serving as coefficient of $\frac{\partial G(\delta, t)}{\partial \delta_i}$ in the PDE for each $i \in \{0 , \dots , n\}$, each with initial condition that $\Delta_i(t) = \delta_i$. This yields the following system of characteristic ODE's:
\begin{align*}
\updot{\Delta}_0(z) &= 1 - e^{\Delta_0 \alpha}\sum_{j\ne i}\theta_je^{\Delta_j} + \Delta_0 \beta, \\
\updot{\Delta}_i(z) &= \mu_i - \mu_{i0}e^{-\Delta_i} - \sum_{j \ne i} \mu_{ij}e^{\Delta_j - \Delta_i} &&  \forall i \in \{1, \dots , n\}.
\end{align*}
We now let $x \in \mathbb{R}^n$ be such that $x_i = e^{\Delta_i}$. Note that this substitution can also be expressed $x = e^{\diag{\Delta}}\onevec{}$, as this will be of use in solving the system. Then, we have that $\updot{x}_i(z) = x_i(z) \updot{\Delta}_i(z)$. In this form, the last $n$ characteristic ODE's can be expressed as
$$
\updot{x}(z) =  - S x(z) + S \onevec{}
$$
which means that
$$
x(z) = \onevec{} + e^{-S (z - t)}\left(e^{\diag{\delta}} - I\right)\onevec{}
$$
\edit{where} we have used the initial condition $x(t) = e^{\diag{\Delta(t)}} = e^{\diag{\delta}}$. We now note that to follow the method of characteristics fully and receive a closed form \minrev{solution} to the PDE we would want to solve the remaining characteristic ODE
$$
\updot{\Delta}_0(z)
=
1 - e^{\Delta_0 \alpha}\sum_{j\ne i}\theta_j e^{\Delta_j} + \Delta_0 \beta
=
1 - e^{\Delta_0 \alpha}\theta^\T x + \Delta_i \beta
$$
which has initial condition that $\Delta_0(t) = \delta_0$. Because this form of ODE is not known to have a closed form solution in terms of standard math functions, we let $h(z)$ be defined as the solution to this initial value problem. Then, we now complete the method of characteristics by solving
$$
\updot{g}(z) = \beta \lambda^* \Delta_0(z) = \beta\lambda^* h(z)
$$
with the initial condition that $g(0) = G(\Delta(0), 0) = \Delta_0(0) \lambda_0 = h(0)\lambda_0$. Since this ODE is already separated, we have
$$
g(z) - h(0)\lambda_0
=
g(z) - g(0)
=
\int_0^z \updot{g}(\xi) \mathrm{d}\xi
=
\beta \lambda^* \int_0^z  h(\xi) \mathrm{d}\xi .
$$
Thus, we now have
$$
G(\delta, t)
=
g(t)
=
\beta \lambda^* \int_0^t  h(\xi) \mathrm{d}\xi
+
h(0)\lambda_0
$$
and this is the stated result.
\end{proof} 
}

\edit{While the ODE in this statement may not be able to be solved for a closed form expression outside of special cases, this reduction of the PDE to an ODE simplifies numerical implementations. We now note that this of course extends to the moment generating function as well by simply taking the exponential of the cumulant generating function.}

\subsection{Simulation Study}\label{subsec_sim}

To conclude Section~\ref{sec_hawkphinf} we provide a collection of simulation examples that verify the accuracy of our expressions for the moments in a variety of settings. In each example we derive the simulated functions via 100,000 replications of the procedure described in \citet{ogata1981lewis}. We start with the mean and variance of a single phase system, as shown in the pair of plots below in Figure~\ref{queuecomp}.

\vspace{-.05in}
\begin{figure}[h]
	\begin{center}
\includegraphics[width=.45\textwidth]{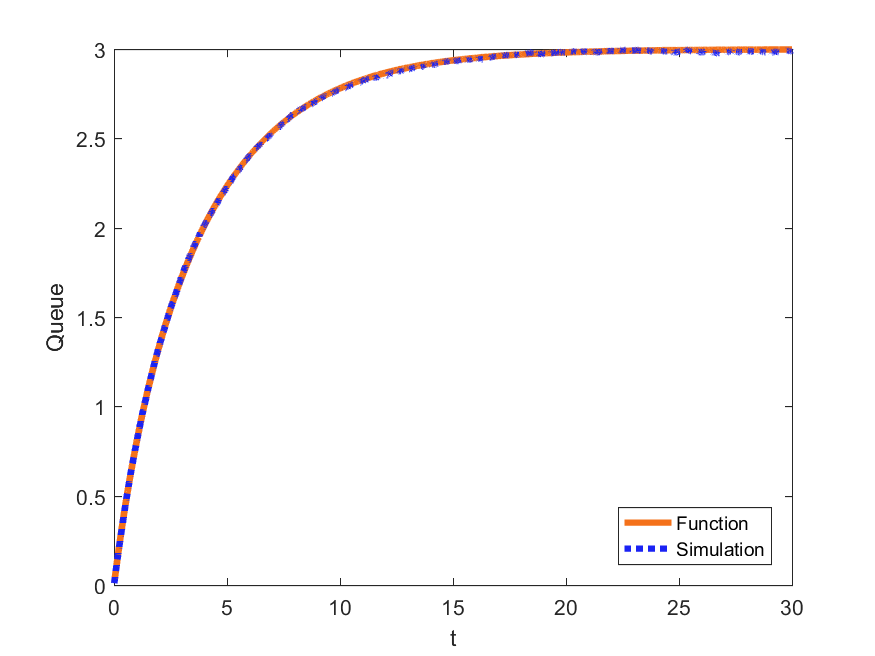}\includegraphics[width=.45\textwidth]{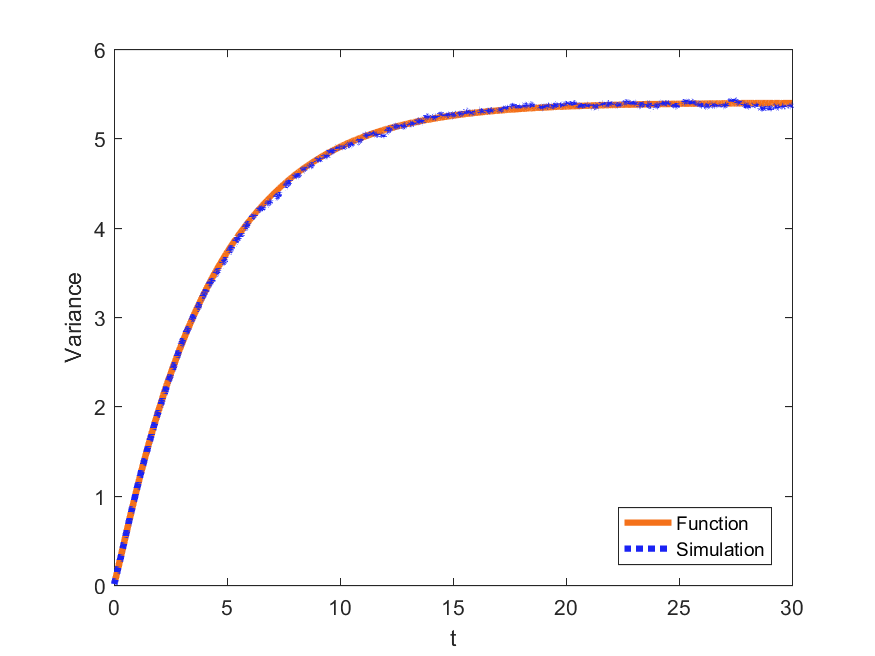}
\caption{Mean (left) and Variance (right) of ${Q_t}$ in ${Hawkes / M / \infty}$, ${\alpha = \frac 1 2}$, ${\beta = \frac 3 4}$, ${\lambda^* = \mu = 1}$.} \label{queuecomp}
\end{center}
\end{figure}

As a second example, we also consider a three-phase Erlang distributed service. We use two different parameter settings, one in which the mean service duration is 1 and another in which the mean service length is 6. In the first case, $\alpha = \frac{1}{2}$, $\beta = \frac{3}{4}$, and $\lambda^* = 1$. In the latter, $\alpha = \frac{3}{4}$, $\beta = \frac{5}{4}$, and $\lambda^* = 1$. The mean is shown in Figure~\ref{queueerlang}, the variance in Figure~\ref{varerlang}, the covariance of the queue and the intensity in Figure~\ref{covqlerlang}, and the covariance of the phases of the queue in Figure~\ref{covqqerlang}.

\vspace{-.25in}
\begin{figure}[h]
	\begin{center}
\includegraphics[width=.45\textwidth]{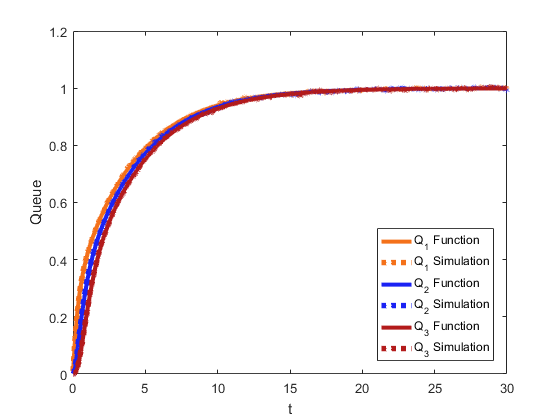}\includegraphics[width=.45\textwidth]{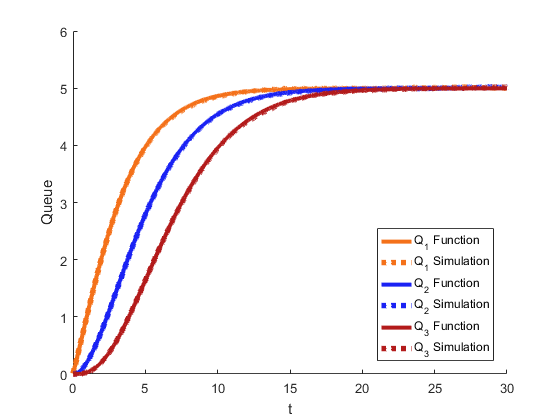}
\caption{Mean of the ${Hawkes / E_3 / \infty}$ Queue, where ${\alpha = \frac 1 2}$, ${\beta = \frac 3 4}$, ${\lambda^* = 1}$, ${\frac{1}{\mu} = 1}$ (left) and  ${\alpha = \frac 3 4}$, ${\beta = \frac 5 4}$, ${\lambda^* = 1}$, ${\frac{1}{\mu}  = 6}$ (right).} \label{queueerlang}
\end{center}
\end{figure}

\begin{figure}[h]
	\begin{center}
\includegraphics[width=.45\textwidth]{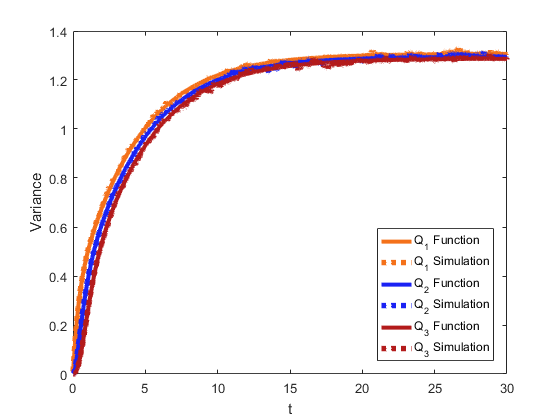}\includegraphics[width=.45\textwidth]{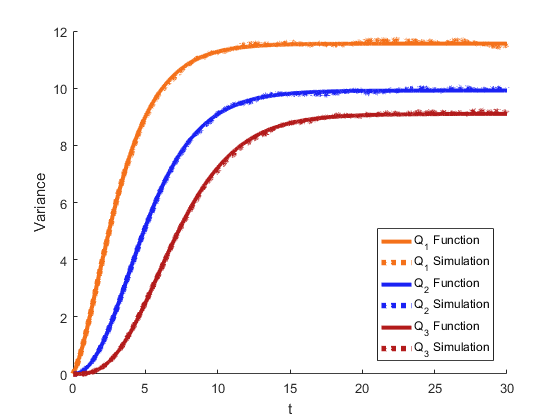}
\caption{Variance of the ${Hawkes / E_3 / \infty}$ Queue, where ${\alpha = \frac 1 2}$, ${\beta = \frac 3 4}$, ${\lambda^* = 1}$, ${\frac{1}{\mu} = 1}$ (left) and  ${\alpha = \frac 3 4}$, ${\beta = \frac 5 4}$, ${\lambda^* = 1}$, ${\frac{1}{\mu}  = 6}$ (right).} \label{varerlang}
\end{center}
\end{figure}

\begin{figure}[h]
	\begin{center}
\includegraphics[width=.45\textwidth]{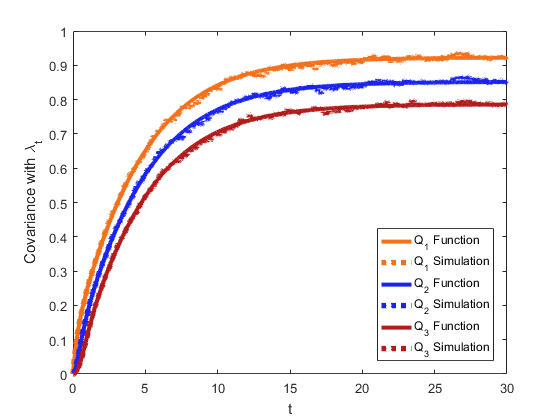}\includegraphics[width=.45\textwidth]{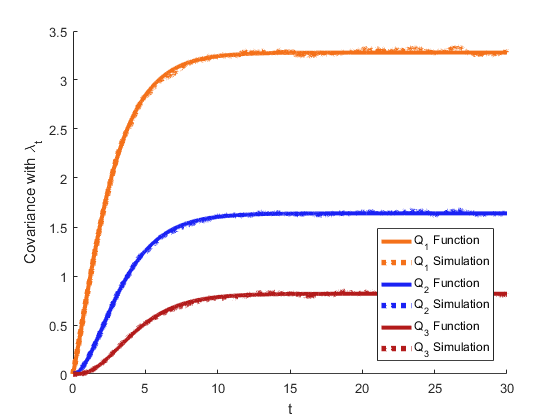}
\caption{Covariance of ${Hawkes / E_3 / \infty}$ Queue, where ${\alpha = \frac 1 2}$, ${\beta = \frac 3 4}$, ${\lambda^* = 1}$, ${\frac{1}{\mu} = 1}$ (left) and  ${\alpha = \frac 3 4}$, ${\beta = \frac 5 4}$, ${\lambda^* = 1}$, ${\frac{1}{\mu}  = 6}$ (right).}
\label{covqlerlang}
\end{center}
\end{figure}

\begin{figure}[h]
	\begin{center}
\includegraphics[width=.45\textwidth]{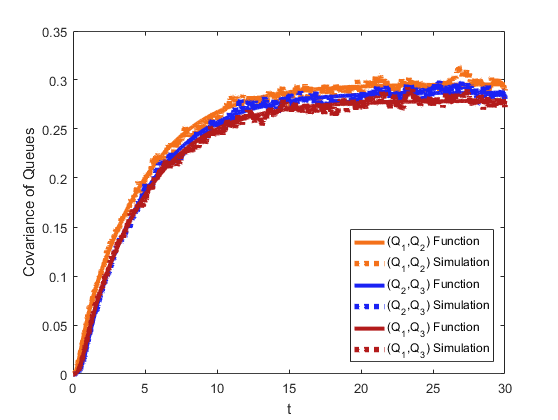}\includegraphics[width=.45\textwidth]{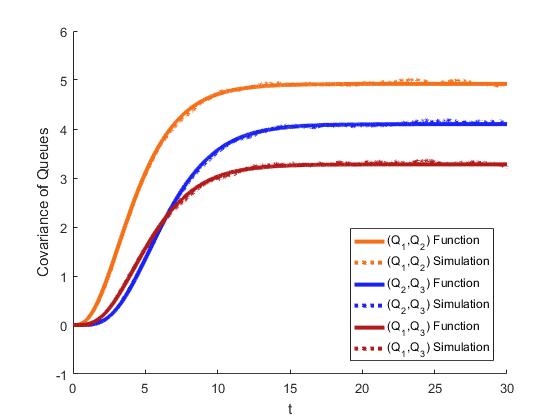}
\caption{Covariance between Phases in the ${Hawkes / E_3 / \infty}$ Queue, where ${\alpha = \frac 1 2}$, ${\beta = \frac 3 4}$, ${\lambda^* = 1}$, ${\frac{1}{\mu} = 1}$ (left) and  ${\alpha = \frac 3 4}$, ${\beta = \frac 5 4}$, ${\lambda^* = 1}$, ${\frac{1}{\mu}  = 6}$ (right).}  \label{covqqerlang}
\end{center}
\end{figure}

In addition to the Erlang setting, we also verify the performance of the hyper-exponential service equations. We again consider a three phase distributed service and display a pair of scenarios. In both parameter groups $\theta = [.15, .4, .45]^\T$ and $\mu = [1, 4, 6]^\T$. In the first setting we consider $\alpha = \frac 1 2$, $\beta = 1$, and $\lambda^* = 2$, whereas in the second setting $\alpha = 1$, $\beta = 2$, and $\lambda^* = 2$. These are displayed in the same order as the Erlang examples are: mean in Figure~\ref{queuehyp}, variance in Figure~\ref{varhyp}, covariance with the intensity in Figure~\ref{covqlhyp}, and covariance of the queues in Figure~\ref{covqqhyp}.

\begin{figure}[h]
	\begin{center}
\includegraphics[width=.45\textwidth]{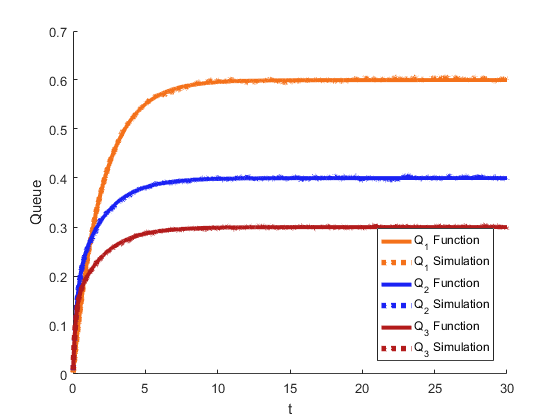}\includegraphics[width=.45\textwidth]{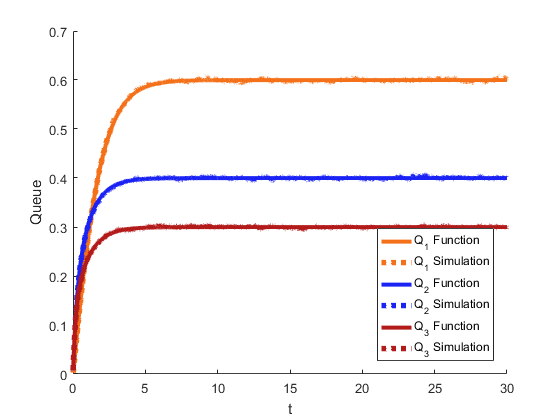}
\caption{Mean of the ${Hawkes / H_3 / \infty}$ Queue, where ${\alpha = \frac 1 2}$, ${\beta = 1}$, ${\lambda^* = 2}$, ${\theta = [.15, .4, .45]^\T}$, ${\mu = [1, 4, 6]^\T}$ (left) and  ${\alpha = 1}$, ${\beta = 2}$, ${\lambda^* = 2}$, ${\theta = [.15, .4, .45]^\T}$, ${\mu = [1, 4, 6]^\T}$ (right).} \label{queuehyp}
\end{center}
\end{figure}

\begin{figure}[h]
	\begin{center}
\includegraphics[width=.45\textwidth]{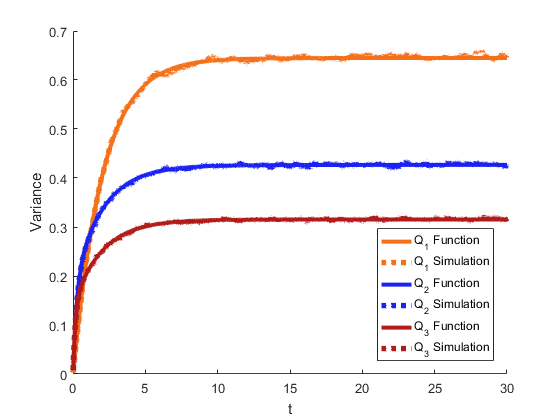}\includegraphics[width=.45\textwidth]{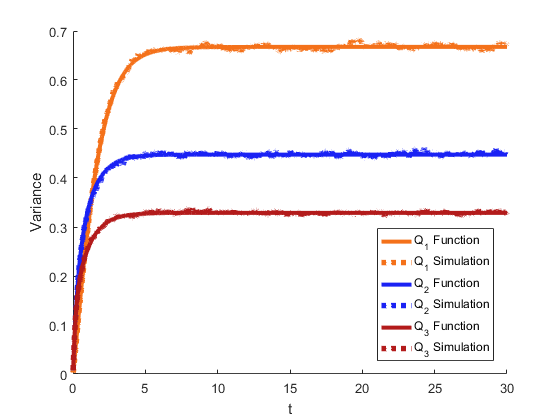}
\caption{Variance of the ${Hawkes / H_3 / \infty}$ Queue, where ${\alpha = \frac 1 2}$, ${\beta = 1}$, ${\lambda^* = 2}$, ${\theta = [.15, .4, .45]^\T}$, ${\mu = [1, 4, 6]^\T}$ (left) and  ${\alpha = 1}$, ${\beta = 2}$, ${\lambda^* = 2}$, ${\theta = [.15, .4, .45]^\T}$, ${\mu = [1, 4, 6]^\T}$ (right).} \label{varhyp}
\end{center}
\end{figure}

\begin{figure}[h]
	\begin{center}
\includegraphics[width=.45\textwidth]{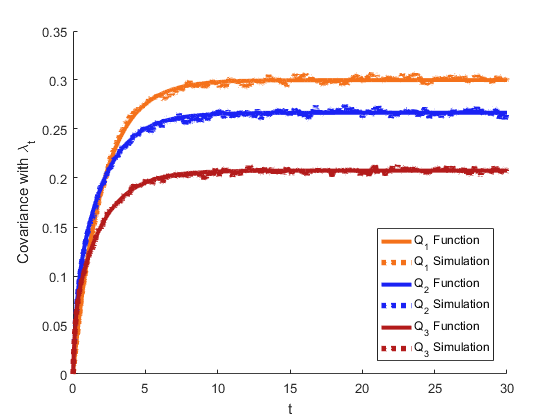}\includegraphics[width=.45\textwidth]{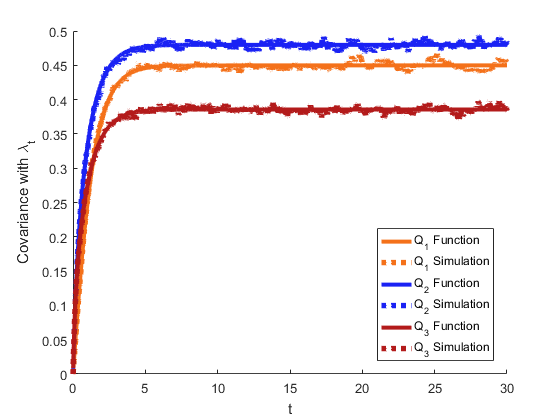}
\caption{Covariance of ${\lambda_t}$ and the ${Hawkes / H_3 / \infty}$ Queue, where ${\alpha = \frac 1 2}$, ${\beta = 1}$, ${\lambda^* = 2}$, ${\theta = [.15, .4, .45]^\T}$, ${\mu = [1, 4, 6]^\T}$ (left) and  ${\alpha = 1}$, ${\beta = 2}$, ${\lambda^* = 2}$, ${\theta = [.15, .4, .45]^\T}$, ${\mu = [1, 4, 6]^\T}$ (right).}  \label{covqlhyp}
\end{center}
\end{figure}

\begin{figure}[h]
	\begin{center}
\includegraphics[width=.45\textwidth]{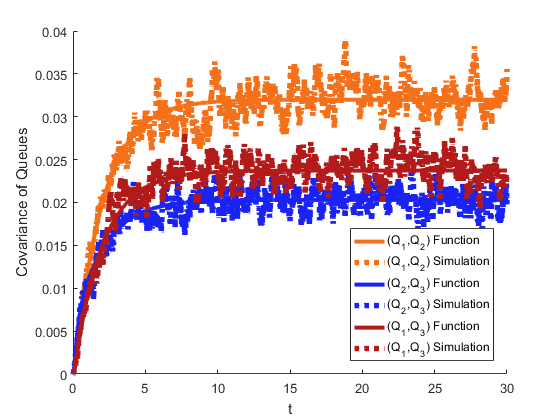}\includegraphics[width=.45\textwidth]{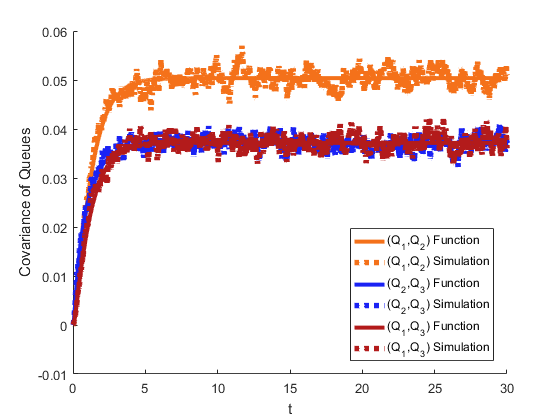}
\caption{Covariance between Phases in the ${Hawkes / H_3 / \infty}$ Queue, where ${\alpha = \frac 1 2}$, ${\beta = 1}$, ${\lambda^* = 2}$, ${\theta = [.15, .4, .45]^\T}$, ${\mu = [1, 4, 6]^\T}$ (left) and  ${\alpha = 1}$, ${\beta = 2}$, ${\lambda^* = 2}$, ${\theta = [.15, .4, .45]^\T}$, ${\mu = [1, 4, 6]^\T}$ (right).} \label{covqqhyp}
\end{center}
\end{figure}

In conducting these simulation experiments we have made an interesting observation. Consider the following example: let $\lambda^* = 1$, $\alpha = 1$, and $\beta = 2$. Then, let $D = 1$ be the fixed service length in a $Hawkes/D/\infty$ system and let $\mu = 1$ be the parameter of the exponential distribution in a $Hawkes/M/\infty$ system. We plot the simulated variances of these two systems in Figure~\ref{hawkesMDfig} based on 10,000 replications, in which we find that the variance is larger in the deterministic service setting.

 \begin{figure}[h]
\begin{center}	
\includegraphics[width=.6\textwidth]{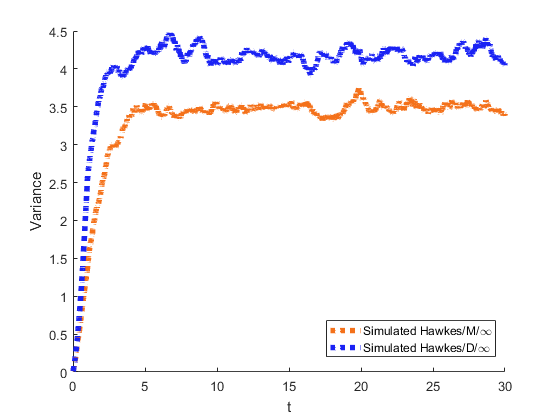}
\caption{Comparison of Variances in ${Hawkes/M/\infty}$ and ${Hawkes/D/\infty}$ Queues when ${\frac{1}{\mu} = D = 1}$, with ${\lambda^* = 1}$, ${\alpha = 1}$, and ${\beta = 2}$.}
\label{hawkesMDfig}
\end{center}
\end{figure}

 While this relationship may seem unexpected, there is an intuitive explanation for it. Because the Hawkes process exhibits clustering behavior in the arrival times, a service system with fixed service length will also experience clusters of departures times. By comparison, a system with random service durations has the opportunity to counteract the clustering behavior and disperse the departure times. In Proposition~\ref{hawkMDcomp} we show that the steady-state variance in the deterministic service setting is greater than that of the exponential service setting. 
 
 \begin{proposition}\label{hawkMDcomp}
 For equal Hawkes process parameters $\lambda^*$, $\alpha$, and $\beta$ and equivalent service parameters $D = \frac{1}\mu > 0$, the steady-state variance of the $Hawkes/D/\infty$ queue is greater than the steady-state variance of the $Hawkes/M/\infty$ queue.
 \end{proposition}
 \begin{proof}
 Let $\beta > \alpha > 0$ and let $\lambda^* > 0$. Further, let $D = \frac{1}\mu > 0$. By Theorem~\ref{hawkesDinfthem}, the steady-state variance of the $Hawkes/D/\infty$ queue is
 \begin{align*}
 \mathcal{V}_{\mathrm{D}} \equiv \lambda_\infty D\left(1 + \frac{2\alpha\beta - \alpha^2}{(\beta-\alpha)^2}\right) - \lambda_\infty(1 - e^{-(\beta-\alpha)D})\frac{2\alpha\beta - \alpha^2}{(\beta-\alpha)^3}.
 \end{align*}
 Likewise, Corollary~\ref{hawkqueuelimit} gives the steady-state variance in the exponential service case as
  \begin{align*}
 \mathcal{V}_{\mathrm{M}} 
 \equiv 
 \frac{\lambda_\infty}{\mu}\left(1 + \frac{2\alpha\beta - \alpha^2}{2(\beta-\alpha)(\mu +\beta-\alpha)}\right) ,
 \end{align*}
 as noted in Remark~\ref{hawkqueuelimitremark}. Then, the difference between these terms is
 \begin{align*}
 \mathcal{V}_{\mathrm{D}} - \mathcal{V}_{\mathrm{M}}
 =
 \frac{\lambda_\infty}{\mu}\left(\frac{2\alpha\beta - \alpha^2}{(\beta-\alpha)^2}\right)\left(1 - \frac{\beta - \alpha}{2(\mu+\beta-\alpha)} - \frac{\mu - \mu e^{-(\beta - \alpha)\frac 1 \mu}}{\beta - \alpha}\right),
 \end{align*}
 where we have substituted $\frac{1}\mu$ for $D$. Because of the assumed relationships among the parameters, $\mathcal{V}_{\mathrm{D}} - \mathcal{V}_{\mathrm{M}}$ is positive if and only if the expression inside the lattermost parenthesis is. Multiplying this expression by $\frac{2}{\mu^2}(\mu + \beta-\alpha)(\beta-\alpha) > 0$ and simplifying yields
 $$
 \Upsilon \left(\frac{\beta-\alpha}{\mu}\right) 
 \equiv
  \left(\frac{\beta-\alpha}{\mu}\right)^2 - 2\left(1-e^{-\frac{\beta-\alpha}{\mu}}\right) + 2 \left(\frac{\beta-\alpha}{\mu}\right)e^{-\frac{\beta-\alpha}\mu} .
 $$
 We can re-parameterize this expression as $\Upsilon(x)$ for $x \equiv  \frac{\beta-\alpha}{\mu}$. By checking the first derivative of $\Upsilon(x)$, we see that it is strictly increasing for $x \geq 0$. Since $\Upsilon(0) = 0$ and $\frac{\beta-\alpha}{\mu} > 0$ for any valid $\alpha$, $\beta$, and $\mu$, we  have that $\mathcal{V}_{\mathrm{D}} - \mathcal{V}_{\mathrm{M}} > 0$.
 \end{proof} 
 
 In Figure~\ref{hawkesLN} we observe that this behavior can also occur in non-Markovian service settings, shown here for lognormal distributions based on 10,000 simulation replications. In this experiment each lognormal distribution has a mean of 1 and the variances increase from 0 to 5 with a step size of 0.5. Note that all the mean queue lengths appear to be converging to 1 in steady-state. Further, we see that the means of systems with higher variance in the lognormal service distribution are converging more slowly than those of lower lognormal variance. However, the opposite relationship appears to hold in terms of the variances of the queues: the higher the variance of the lognormal, the lower the variance of the queue.\\
 
 \vspace{-.2in}
 \begin{figure}[h]
	\begin{center}
\includegraphics[width=.495\textwidth]{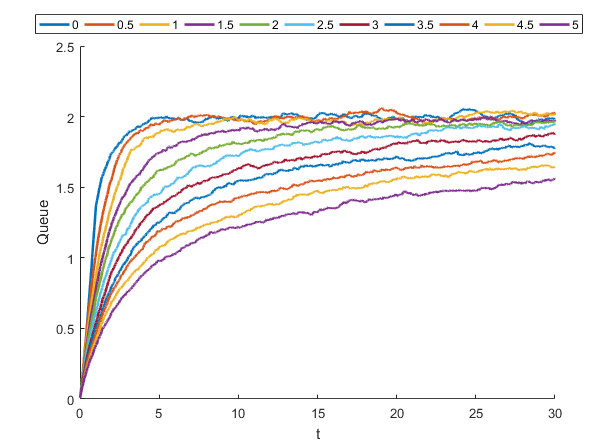}\includegraphics[width=.495\textwidth]{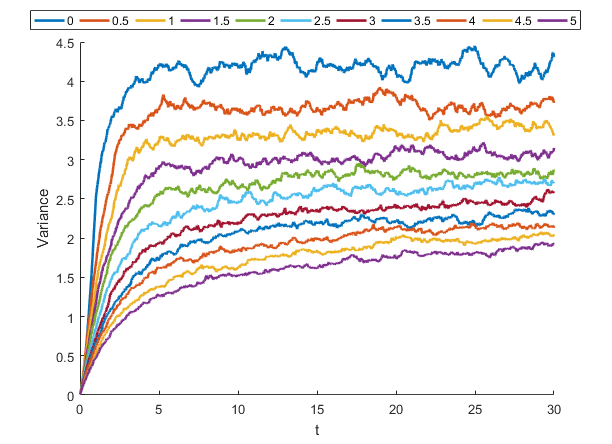}
\caption{Mean (left) and Variance (right) of the ${Hawkes/Lognormal/\infty}$ with ${\lambda^* = 1}$, ${\alpha = 1}$, and ${\beta = 2}$ where Mean Service Durations is 1 and Service Variance Increases from 0 to 5.} \label{hawkesLN}
\end{center}
\end{figure}

\section{Applications} \label{sec_applications}

To motivate this study and demonstrate its findings, we now \edit{briefly discuss two applications of this work, one concerned with viral internet traffic and one covering night clubs. Each is inspired by the self-excitement behavior of the Hawkes process, and in these \minrev{settings} we consider the impact and influence one arrival can have on a system and how managers of such systems might try to harness that influence for some kind of benefit.}

\subsection{Trending Web Traffic}\label{webtraffic}
In May 2017 website rankings for the \edit{United} States, Youtube, Facebook, and Reddit each ranked among the top 5 most visited websites, with Twitter in the top 10 and LinkedIn and Instagram both in the top 15, per Alexa \cite{websites}. For Facebook, Reddit, and Twitter in particular, users' interactions with the sites frequently involve viewing links to external media like videos, articles, and shopping sales. A user's exposure to a webpage and her likelihood to share it herself is directly influenced by whether she sees the link from other users. As users choose to visit and potentially re-share links posted by other users, the link may start trending or become ``viral.'' This means that it is receiving high levels of traffic and arrivals to the site, and this may lead to even more arrivals while the users continue to share it on various social platforms. For a business or organization, going viral can lead to significant jumps in exposure, interest, and revenue.

As a basic example, we analyzed publicly available Twitter data \cite{mckelvey2013truthy}. This data set covers all tweets featuring both a URL and a hashtag from November 2012 and includes the tweet timestamp, the hashtags used, and the URL's linked, as well as an anonymous user ID. Perhaps the most notable event captured among the reactions in this data set is the 2012 U.S. Presidential election, which was held on November 6. Among the bountiful election-related tweets are 106 posts of the music video for Young Jeezy's 2008 song \textit{My President} from the start of November 5 to midday on November 7.  A plot of the timestamps of these tweets along with the total number of tweets occurring by that time is below. Note the flurry of posts once the election results were announced; 60 of the data's 106 postings of the video occur within an hour's time. \edit{A quick numerical investigation suggests that this type of extreme viral reaction may be more likely in certain parameter settings. In 100,000 simulation replications of a system with $\lambda^* = 0.5$, $\alpha = 19.5$, and $\beta = 20$, 82.4\% of the trials had a majority of arrivals occur within one time quartile. By comparison, in the same number of replications for a system with $\lambda^* = 1$, $\alpha = 0.5$, and $\beta = 1$, this only occurred for 18.0\% of the experiments.
} However, even outside of the main spike in this data, users seem to be posting the video in clustered time segments, approximately at the 6, 20, 45, 48, and 52 hour marks. These clusters suggest that these arrivals could be appropriately modeled by a Hawkes process, particularly when compared to a Poisson process.

\vspace{-.15in}
\begin{figure}[H]
\begin{center}	
\includegraphics[scale=.625]{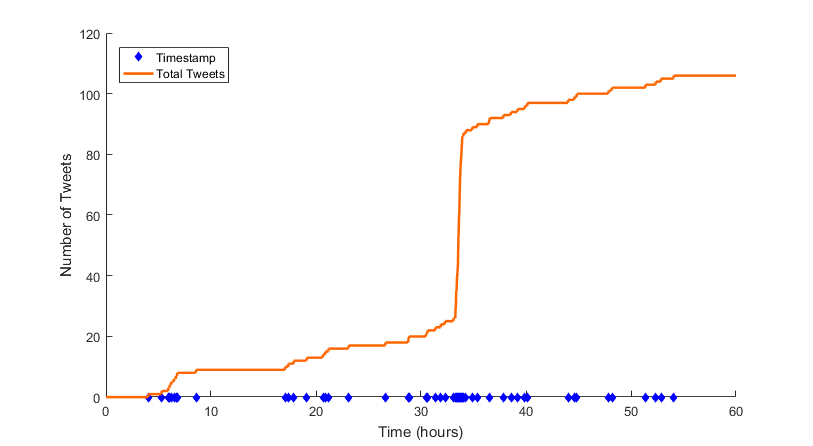}
\vspace{-.05in}
\caption{Tweets of Young Jeezy - \textit{My President} music video from November 5 - 7, 2012.}
\label{mypres}
\end{center}
\end{figure}
\vspace{-.25in}

Using what we have observed from this data as inspiration, we now model users arriving to a webpage as a Hawkes process. Because of the viral behavior we have seen in this type of arrivals, we will investigate the impact of a click. Consider a Hawkes Process $N_t$ with baseline intensity $\lambda^*$, initial intensity $\lambda_0$, jump size $\alpha$, and decay parameter $\beta$. Now, let $\edit{\hat N_t}$ represent an independent Hawkes process that is identical to $N_t$ \edit{in terms of parameters} with the exception that it experienced an arrival at time 0, whereas $N_t$ starts empty. This means that the baseline intensity, jump size, and decay parameter are the same for $\edit{\hat N_t}$ as they were for $N_t$, but the initial intensity is $\lambda_0 + \alpha$ and $\edit{\hat N_0 = 1}$. Then, by \edit{Proposition~\ref{hawkgeneral}},
\begin{align*}
\E{\edit{\hat N_t}} - \E{N_t}
&=
\lambda_\infty t + \frac{\lambda_0 + \alpha - \lambda_\infty}{\beta - \alpha}\left(1 - e^{-(\beta - \alpha)t}\right) + 1
-
\lambda_\infty t - \frac{\lambda_0 - \lambda_\infty}{\beta - \alpha}\left(1 - e^{-(\beta - \alpha)t}\right)
\\
&
=
\frac{\beta}{\beta - \alpha} - \frac{ \alpha }{\beta - \alpha}e^{-(\beta - \alpha)t}
\longrightarrow
\frac{\beta}{\beta - \alpha} \text{ as $t\to\infty$}
\end{align*}
which shows that the gap between the two expectations is positive and grows throughout time. However, this is simply tracking the number of visitors\edit{; it} does not account for the time the users spend on the site. To capture this, we can extend this arrival model to a queueing model in which the service represents the time the user spends on the webpage. Provided the website is well hosted, this \edit{can be modeled as} an infinite server queue as any user can visit the webpage that chooses to do so. If the time each user spends on the page is independently and exponentially distributed with rate $\mu$, we see that the expected number of users on the page at time $t$ is $\E{Q_t}$. Then, from time 0 to time $T$ the expected total time spent on the page across all users $\sigma(T)$ is
\begin{align*}
\sigma(T)
=
\int_0^T \E{Q_t} \, \mathrm{d}t
&
=
\int_0^T
\left(\frac{\lambda_\infty}{\mu} \left(1 - e^{-\mu t}\right)
+
\frac{ \lambda_0 - \lambda_\infty }{\mu - \beta + \alpha}\left(e^{-(\beta - \alpha) t}  - e^{- \mu t}\right)
\right) \mathrm{d}t
\\
&
=
\frac{\lambda_\infty}{\mu} \left(T
-
\frac{1 - e^{-\mu T}}{\mu} \right)
+
\frac{ \lambda_0 - \lambda_\infty }{\mu - \beta + \alpha}\left(\frac{1 -e^{-(\beta - \alpha) T}}{\beta - \alpha} - \frac{1  - e^{- \mu T}}{\mu}\right)
\end{align*}
where we have applied the results of Corollary~\ref{hawkqueuehyper} for hyper-exponential service with $n=1$ and $\mu \ne \beta - \alpha$, thus yielding exponential service. Now, suppose that a website earns $m$ dollars per unit of time in advertising revenue for each user on the site. Then, the expected earnings by time $T$ is $A(T) = m \sigma(T)$. We can now repeat the value of a click experiment when also considering service. Let $Q_t$ be a queueing system with exponential service at rate $\mu$, infinite servers, and Hawkes process arrivals with parameters $\lambda^*$, $\alpha$, and $\beta$ and assume the queue starts empty. Then, let \edit{$\hat Q_t$ be the analogous adaptation of $Q_t$ that $\hat N_t$ is to $N_t$.} Let $A(T)$ and \edit{$\hat A(T)$} be the corresponding expected \edit{dwell time} revenues, each with earning rate $m$. Note that the expected time the initial customer has spent in the system by time $T$ is $\min\{S, T\}$ where $S$ is the duration of her service. Hence the revenue associated with her visit to the page by time $T$ is $m \frac{1 - e^{-\mu T}}{\mu}$. Then,
\begin{align*}
\hat A(T) - A(T)
&=
m \frac{1 - e^{-\mu T}}{\mu}
+
m
\frac{ \alpha }{\mu - \beta + \alpha}\left(\frac{1 -e^{-(\beta - \alpha) T}}{\beta - \alpha} - \frac{1  - e^{- \mu T}}{\mu}\right)
\\
&
=
\frac{m}{\mu}\left(1 + \frac{1}{\beta - \alpha}\right)
-
m
\frac{ \alpha e^{-(\beta - \alpha) T} }{(\beta - \alpha)(\mu - \beta + \alpha)}
-
m
\frac{(\mu - \beta)e^{-\mu T}}{\mu(\mu - \beta + \alpha)} \edit{,}
\end{align*}
which can be shown to also always grow with $T$ via its first derivative. \edit{We can also further observe that as $\alpha \to \beta$ each of these gaps \minrev{grows} towards infinity, and thus so grows the impact of a click in viral settings.}

Note that this model can also be used for internet-inspired applications other than users arriving to internet pages. For example, as mobile carriers continue to add cloud storage based services and allow customers to upload pictures from their smart phones as soon as they are taken, the $Hawkes / M / \infty$ queue can be used to describe the number of pictures being uploaded at once. \edit{For further reading on the Hawkes process and its use in internet traffic applications see \cite{rizoiu2017expecting}, in which the authors develop a novel Hawkes-process-based model for the popularity of online content in great detail.}

\subsection{Club Queue}\label{clubqueue}

From our Hawkes driven infinite server queue with phase-type service distributions, we can construct what we refer to as the \emph{Club Queue}. This stems from an application perhaps uncommon to queueing systems, a nightclub. This setting features a key characteristic: the best club has the most people waiting for it. \edit{Because of this, the Hawkes process naturally represents the excitation exhibited by club-goers joining a queue as many club-goers might call their friends to join them. With this application in mind, it is important to understand the characteristics of nightclubs.  Many nightclubs have waiting spaces for potential customers outside the club.  Moreover, inside the club is where much of the activity happens.  Thus, using phase-type distributions we can model the inside and outside of the club as two phases of services or a two dimensional phase-type queue. The first phase of service can be considered ``admittance'' to the service with the second step being the service itself.} \minrev{Because the clubs' bouncers have the ability to admit customers into the venue from any position in the external queue and because each customer determines how long she stays in the club, we model this scenario as an infinite server queue.} This process is visualized below, where $\mu_O$ and $\mu_I$ are the rates of each step of service.\\
\begin{figure}[H]
\begin{center}
\begin{tikzpicture}[>=latex]
\draw (-1,0) -- ++(2cm,0) -- ++(0,-1.5cm) -- ++(-2cm,0);
\foreach \i in {1,...,4}
  \draw (1cm-\i*10pt,0) -- +(0,-1.5cm);

\draw (4.4,0) -- ++(2cm,0) -- ++(0,-1.5cm) -- ++(-2cm,0);
\foreach \i in {1,...,4}
  \draw (6.4cm-\i*10pt,0) -- +(0,-1.5cm);

\draw (1.75,-0.75cm) circle [radius=0.75cm]; 
\draw (7.15,-0.75cm) circle [radius=0.75cm]; 

\draw[->] (2.5,-0.75) -- +(55pt,0);
\draw[->] (7.9,-0.75) -- +(40pt,0);
\draw[<-] (-1,-0.75) -- +(-50pt,0);
\node at (3.35,-0.5cm) {$\mu_O$};
\node at (3.4,-1.05cm) {\footnotesize Admittance};
\node at (8.45,-0.5cm) {$\mu_I$};
\node at (-2, -0.5cm) {$\lambda_t$};
\node[align=center] at (0.65cm,-2cm) {Arrivals};
\node[align=center] at (6.15cm,-2cm) {Service};
\node[align=center] at (1.75cm,-0.75cm) {$Q_O$};
\node[align=center] at (7.15cm,-0.75cm) {$Q_I$};
\end{tikzpicture}
\end{center}
\begin{caption}
{Club Queue Process Diagram.}
\end{caption}\label{clubqueuediagram}
\end{figure}
We can represent the Club Queue using the two dimensional vector of queue lengths $Q(t)$ for $t \geq 0$, with coordinates $Q_I(t)$ and $Q_O(t)$ representing the service systems inside and outside the club, respectively. \edit{A fundamental managerial task is to figure out at what rate to admit club-goers into the club to maximize profitability while making the club attractive from the outside. This is non-trivial as a short line outside the club might signal to others that the club is not interesting and make them choose to not go inside the club.  However, if the line is too long, there are many customers not actively generating revenue for the club and \minrev{becoming} frustrated with the wait outside.}  With this in mind, we construct the following objective function that maximizes the rate at which the bouncer of the club should let club-goers inside the club \minrev{over the finite time horizon $[0,T]$, where $T > 0$}.
\begin{equation}
\mathcal{\zeta}(\mu_O(t)) = r_O \mu_O \E{Q_O(t)} + r_I  \E{Q_I(t)} - c (\mu_O\E{Q_O(t)} - k)^2 - w \mu_O^2    \label{objecive_function0}
\end{equation}
Here $r_O \geq 0$ and $r_I \geq 0$ are revenues generated from the cover outside and inside the club respectively. We also have that $c$ is a penalty for having the overall admittance rate be too slow or too fast and finally, $w$ is a penalty for admitting each individual customer too quickly.
A complete formulation of this optimal control problem is presented next.

\begin{problem}[Unconstrained Club Profit Model]\label{profit_problem}
\begin{eqnarray*}
 &\max_{\lbrace \mu_O\geq 0\rbrace}&\int_{0}^{T} \left [ r_O \mu_O(t) \E{Q_O(t)} + r_I  \E{Q_I(t)} - c (\mu_O(t)\E{Q_O(t)} - k)^2 - w \mu_O(t)^2  \right ] \mathrm{d}t \nonumber \\
&&\text{subject to} \nonumber\\\label{equation_motion}
&&\updot{\mathrm{E}}[\lambda(t)]=\beta \cdot ( \lambda^* - \mathrm{E}[\lambda(t)] ) + \alpha \cdot \mathrm{E}[\lambda(t)]  \nonumber\\
&&\updot{\mathrm{E}}[Q_O(t)]= \mathrm{E}[\lambda(t)]-\mu_O(t) \cdot \mathrm{E}[Q_O(t)] \nonumber\\
&&\updot{\mathrm{E}}[Q_I(t)]=\mu_O(t) \cdot \mathrm{E}[Q_O(t)]  - \mu_I \cdot \mathrm{E}[Q_I(t)] \nonumber
\end{eqnarray*}
\end{problem}

The solution to this problem gives the optimal rate to admit club-goers across time in order to maximize the difference between club revenue and  the queue length and admittance rate penalties. This is characterized by the following theorem.

\begin{theorem}
The optimal solution to Problem~\ref{profit_problem} is given by $\mu^*_O(t)$, where
\begin{align}
\mu^*_O(t) = \frac{(r_O + 2ck - \gamma_1 + \gamma_2)\E{Q_O(t)}}{2w + 2c\E{Q_O(t)}^2} \label{optmu}
\end{align}
for all \minrev{$t \in [0,T]$}.
\end{theorem}

\begin{proof}
We start by transforming the optimization model into a single Hamiltonian equation, which can be thought of as an unconstrained version of the Lagrangian. For this problem, we have the Hamiltonian $\mathcal{H}$ as
\begin{align*}
\mathcal{H}(t,\gamma)
&=
\mathcal{\zeta}(\mu_O(t))
-
\gamma_1\left(\updot{\mathrm{E}}[Q_O(t)] - \mathrm{E}[\lambda(t)] + \mu_O \mathrm{E}[Q_O(t)]\right)
-
\gamma_2\bigg(\updot{\mathrm{E}}[Q_I(t)] - \mu_O \mathrm{E}[Q_O(t)]
\\
&
\quad
 + \mu_I  \mathrm{E}[Q_I(t)]\bigg)
-
\gamma_3\left(\updot{\mathrm{E}}[\lambda(t)] - \beta \cdot ( \lambda^*(t) - \mathrm{E}[\lambda(t)] ) - \alpha \cdot \mathrm{E}[\lambda(t)] \right)
\end{align*}
where each $\gamma_i \in \mathbb{R}$ for $i \in \{1,2,3\}$. To achieve optimality in the control problem, the method ensures that $\mu_O(t)$ is such that $\frac{\mathrm{d}\mathcal{H}}{\mathrm{d}\mu_O(t)} = 0$ for all \minrev{$t \in [0,T]$}. We see that the derivative of the Hamiltonian with respect to $\mu_O(t)$ is
\begin{align*}
\frac{\mathrm{d}\mathcal{H}}{\mathrm{d}\mu_O(t)}
&=
r_O\E{Q_O(t)} - 2c\mu_O(t)\E{Q_O(t)}^2 + 2ck\E{Q_O(t)} - 2w\mu_O(t) - \gamma_1 \E{Q_O(t)} + \gamma_2 \E{Q_O(t)} .
\end{align*}
Thus, the optimal $\mu^*_O(t)$ is found by solving
\begin{align*}
0
&=
 \frac{\mathrm{d}\mathcal{H}}{\mathrm{d}\mu_O(t)}
 =
(r_O  + 2ck - \gamma_1  + \gamma_2) \E{Q_O(t)}
- (2c\E{Q_O(t)}^2 + 2w)\mu^*_O(t)
\end{align*}
for $\mu^*_O(t)$, which \minrev{yields the expression in Equation~\ref{optmu}}. Because the objective function is concave in $\mu_O(t)$ at every $t$, we have that this solution corresponds to a maximum.
\end{proof} 

Using the differential equations shown in Section \ref{sec_hawkphinf}, this optimization problem can be solved numerically by the Forward Backward sweep method as in \citet{niyirora2016optimal, qin2017dynamic, lenhart2007optimal}. We now give two example outputs of this method below.
\begin{figure}[H]
\captionsetup{justification=centering}
		\hspace{-.35in}~\includegraphics[scale = .45]{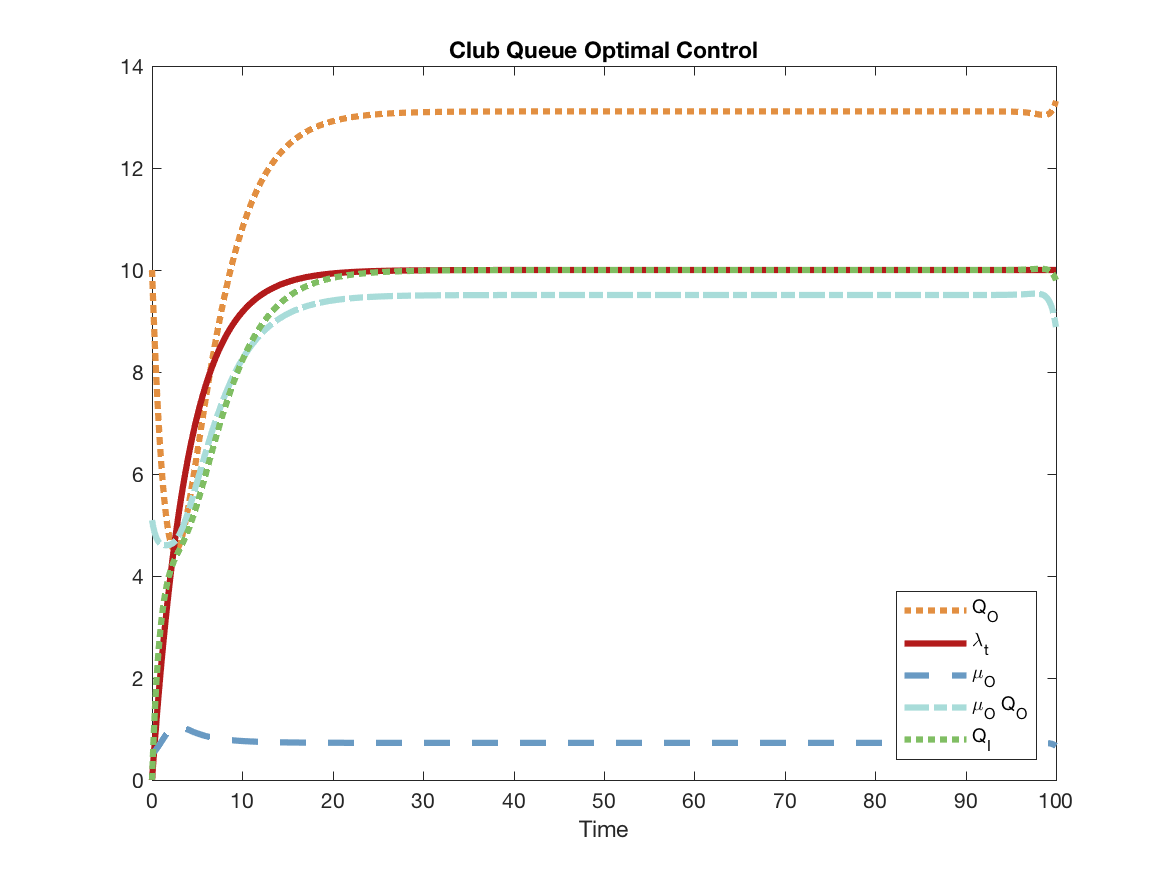}~\hspace{-.4in}~\includegraphics[scale = .45]{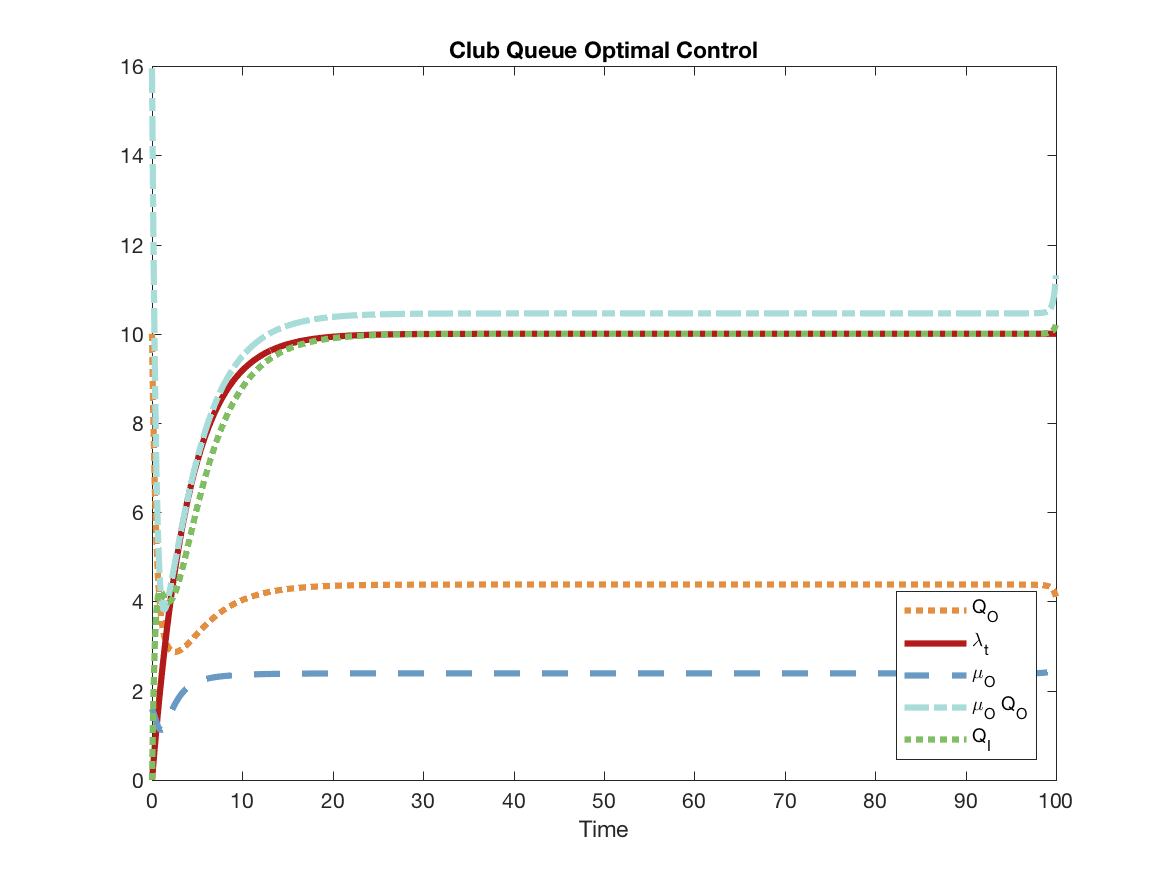}
\caption{Example Forward Backward Sweep Implementation.} \label{hawkescontrol}
\end{figure}

In the scenario on the left, the parameters are as follows: $r_O$, the external entrance revenue rate, is equal to 100 units of currency per units of time. The revenue per person inside, $r_I$, is equal to 100 units of currency per person. The cost of deviating from the desired admittance rate $k$, $c$, is also 100, whereas $k = 8$. Finally, the penalty for admitting individuals too quickly, $w = 150$. On the right, $w$ is instead 100 and $k = 12$. These changes have significant impacts on the resulting solution. On the left the outside queue is allowed to grow roughly three times as large whereas on the right $\mu_O$ is approximately twice the size of that on the left.

\section{Conclusion and Final Remarks} \label{sec_conclusion}

In this paper, we analyze a new infinite server stochastic queueing model that is driven by a Hawkes arrival process and phase-type \minrev{distributed service}.  We are able to derive the exact moments and moment generating function for the Hawkes driven queue as well as the Hawkes process itself.

Although we have analyzed this queueing model in great detail, there are many extensions that are worthy of future study.  One extension that we intend to explore is the impact of a non-stationary baseline intensity in the spirit of \citet{massey2013gaussian, pender2014gram, engblom2014approximations, pender2016analysis, pender2015nonstationary, pender2015truncated, pender2016risk}.  In one simple example, we could set the baseline be $\lambda^*(t) = \lambda^* + \rho \cdot \sin(t) $.  This analysis of a non-stationary baseline intensity is important not only because arrival rates of customers are not constant over time, but also because it is important to know how to distinguish and separate the impact of the time varying arrival rate from the impact of the stochastic dynamics of the self-excitation.  The extension of one periodic function such as $\sin(t)$ seems analytically tractable, however, additional functions may require Fourier analysis.

Other extensions include the modeling of different types of queueing models other than the infinite server model.  For example, it would be interesting to apply our analysis to the Erlang-A queueing model with abandonments.  With regard to obtaining analytical expressions for the Erlang-A model, this is a non-trivial problem because even the Erlang-A queueing model with a Poisson arrival process is analytically \edit{somewhat} intractable.  This presents new challenges for deriving analytical formulas and approximations for the moment behavior of this type of queueing model.  Work by \citet{massey2011poster, pender2014poisson, pender2014laguerre, pender2015nonstationary, pender2016sampling, daw2017new} shows that simple closure approximations or spectral expansions can be effective at approximating the dynamics of the Erlang-A model and variants.  Thus, a natural extension is to apply these techniques to the Erlang-A setting when it is driven by a Hawkes process.  Not only do these approximations have the potential to describe the moment dynamics, but they can be used to stabilize performance measures like in \citet{pender2017approximating}.  A detailed analysis of these extensions will provide a better understanding how the information that operations managers provide to their customers will affect the dynamics of these real world systems like in \citet{pender2017queues, pender2018analysis, pender2017strong}.  We plan to explore these extensions in subsequent work.

\section*{Acknowledgements}
This work is supported by the National Science Foundation under grant DGE-1650441.

\bibliographystyle{plainnat}
\bibliography{hawkes}

\section*{Appendix}

\subsection*{A.1.\quad Auto-covariance of the $Hawkes/PH/\infty$ Queue}
\begin{proposition}
Consider the $Hawkes/PH/\infty$ queue described in Section~\ref{sec_hawkphinf} with sub-generator matrix $S \in \mathbb{R}^{n \times n}$ such that $S + (\beta - \alpha)I$ is nonsingular. Then, for $t \geq \tau \geq 0$,
\begin{align*}
&\Cov{Q_t, Q_{t-\tau}}
=
\lambda_\infty\left(-S^\T \right)^{-1}\left(I - e^{S^\T \tau}\right)\theta
\Bigg(
\lambda_\infty\hspace{-.1cm} \left(- S^\mathrm{T}\right)^{-1}
\big(I - e^{S^\mathrm{T} (t - \tau)}\big)\theta
-
\left( \lambda_0 - \lambda_\infty \right)\left(S^\mathrm{T} + (\beta - \alpha)I\right)^{-1}
\\
&
\quad
\cdot
\big(e^{-(\beta - \alpha) (t - \tau)} I - e^{S^\mathrm{T} (t - \tau)}\big)\theta
\Bigg)^\T
-
\left(S^\T + (\beta - \alpha)I\right)^{-1} \left(e^{-(\beta - \alpha)\tau} I - e^{S^\T \tau} \right) \theta
\Bigg(
\frac{\alpha(2\beta - \alpha)\lambda_\infty}{2(\beta - \alpha)}
\\
&
\quad
\cdot
\left((\beta - \alpha)I - S^\mathrm{T}\right)^{-1}\left(I - e^{(S^\mathrm{T} - (\beta - \alpha)I)(t- \tau)}\right)\theta
-
\frac{\alpha\beta(\lambda_0 - \lambda_\infty)}{\beta - \alpha}
\left(S^\mathrm{T}\right)^{-1}
\Big(e^{-(\beta - \alpha)(t-\tau)}I
\\
&
\quad
-
e^{(S^\mathrm{T} - (\beta - \alpha)I)(t-\tau)}\Big) \theta
+
\frac{\alpha^2 (2\lambda_0 - \lambda_\infty)}{2(\beta - \alpha)} \left(S^\mathrm{T} + (\beta - \alpha)I\right)^{-1}
\nonumber
\Big(e^{-2(\beta - \alpha)(t-\tau)}I
\\
&
\quad
-
e^{(S^\mathrm{T} - (\beta - \alpha)I)(t-\tau)}\Big) \theta
+
\left(\lambda_\infty + (\lambda_0 - \lambda_\infty)e^{-(\beta - \alpha)(t-\tau)}\right)
\Big(
\lambda_\infty\hspace{-.1cm} \left(- S^\mathrm{T}\right)^{-1}\big(I - e^{S^\mathrm{T}(t - \tau)}\big)\theta
\\
&
\quad
-
\left( \lambda_0 - \lambda_\infty \right)\left(S^\mathrm{T} + (\beta - \alpha)I\right)^{-1}
\big(e^{-(\beta - \alpha) (t - \tau)} I - e^{S^\mathrm{T} (t - \tau)}\big)\theta
\Big)
\Bigg)^\T
+
\lambda_\infty \left(S^\T + (\beta - \alpha)I\right)^{-1}
\\
&
\quad
\cdot
 \left(e^{-(\beta - \alpha)\tau} I - e^{S^\T \tau} \right)
\theta
\Big(
\lambda_\infty \left(- S^\mathrm{T}\right)^{-1}
\big(I - e^{S^\mathrm{T} (t - \tau)}\big)\theta
-
\left( \lambda_0 - \lambda_\infty \right)\left(S^\mathrm{T} + (\beta - \alpha)I\right)^{-1}
\\
&
\quad
\cdot
\big(e^{-(\beta - \alpha) (t - \tau)} I - e^{S^\mathrm{T} (t - \tau)}\big)\theta
\Big)^\T
+
\frac{\alpha(2\beta - \alpha)\lambda_\infty}{2(\beta - \alpha)}
\left((\beta - \alpha)I - S^\T\right)^{-1}
\textcolor{black}{e^{S^\T \tau}}
\Bigg(
2(\beta - \alpha)e^{S^\T (t - \tau)}
\\
&
\quad
\cdot
M_{0,\theta, S}(t - \tau)e^{S (t - \tau)}
+
\theta \theta^\T
-
e^{S^\T (t - \tau)}
\theta \theta^\T
e^{S (t - \tau)}
+
e^{S^\T (t - \tau)}
\theta \theta^\T
\left(e^{-(\beta - \alpha) (t - \tau)}I - e^{S (t - \tau)} \right)
\\
&
\quad
\cdot
((\beta - \alpha)I + S)^{-1}((\beta - \alpha)I - S)
+
\left((\beta - \alpha)I - S^\T\right)\left((\beta - \alpha)I + S^\T\right)^{-1}
\Big(e^{-(\beta - \alpha) (t - \tau)}I
\\
&
\quad
-
e^{S^\T (t - \tau)}\Big) \theta \theta^\T e^{S (t - \tau)}
\Bigg)
\left((\beta - \alpha)I - S\right)^{-1}
+
\frac{\alpha\beta(\lambda_0 - \lambda_\infty)}{\beta - \alpha}
\left( S^\T\right)^{-1}
\textcolor{black}{e^{S^\T \tau}}
\Bigg(
(\beta - \alpha)e^{S^\T (t - \tau)}
\\
&
\quad
\cdot
M_{-(\beta - \alpha),\theta, S}(t - \tau)e^{S (t - \tau)}
+
e^{-(\beta - \alpha) (t - \tau)}\theta \theta^\T
-
e^{S^\T (t - \tau)}
\theta \theta^\T
e^{S (t - \tau)}
-
e^{S^\T (t - \tau)}
\theta \theta^\T
\\
&
\quad
\cdot
\left(e^{-(\beta - \alpha) (t - \tau)}I - e^{S (t - \tau)} \right)
((\beta - \alpha)I + S)^{-1}S
-
S^\T\left((\beta - \alpha)I + S^\T\right)^{-1}
\Big(e^{-(\beta - \alpha) (t - \tau)}I
\\
&
\quad
-
e^{S^\T (t - \tau)}\Big)
\theta \theta^\T e^{S (t - \tau)}
\Bigg)
S^{-1}
-
\frac{\alpha^2(2\lambda_0 - \lambda_\infty)}{2(\beta - \alpha)}\big((\beta - \alpha)I + S^\T\big)^{-1}
\textcolor{black}{e^{S^\T \tau}}
\Bigg(
e^{-2(\beta - \alpha) (t - \tau)}\theta \theta^\T
\\
&
\quad
-
e^{S^\T (t - \tau)}
\theta \theta^\T
e^{S (t - \tau)}
-
e^{S^\T (t - \tau)}
\theta \theta^\T  \big(e^{-(\beta - \alpha) (t - \tau)}I
- e^{S (t - \tau)} \big)
-
\left(e^{-(\beta - \alpha) (t - \tau)}I - e^{S^\T (t - \tau)}\right)
\\
&
\quad
\cdot
\theta \theta^\T e^{S (t - \tau)}
\Bigg)
\left((\beta - \alpha)I + S\right)^{-1}
-
\lambda_\infty
\textcolor{black}{e^{S^\T \tau}}
\mathrm{diag}\Big(
\left(S^\mathrm{T}\right)^{-1}
\left(I - e^{S^\mathrm{T} (t - \tau)}\right)\theta
\Big)
  -
 ( \lambda_0 - \lambda_\infty )
 \\
 &
 \quad
 \cdot
 \textcolor{black}{e^{S^\T \tau}}
\diag{
\left(S^\mathrm{T} + (\beta - \alpha)I\right)^{-1}\left(e^{-(\beta - \alpha) (t - \tau)} I - e^{S^\mathrm{T} (t - \tau)}\right)\theta
}
+
\Big(
\lambda_\infty \left(- S^\mathrm{T}\right)^{-1} \hspace{-.1cm}\big(e^{S^\T \tau} - I\big)\theta
\\
&
\quad
-
( \lambda_0 - \lambda_\infty )\hspace{-.1cm}
\left(S^\mathrm{T} + (\beta - \alpha)I\right)^{-1}\hspace{-.1cm}\big(e^{-(\beta - \alpha)I t + (S^\T +(\beta-\alpha)I) \tau} I - e^{-(\beta - \alpha)t}I + e^{S^\T t} - e^{S^\mathrm{T} t }\big)\theta
\Big)
\hspace{-.1cm}
\Big(
\lambda_\infty\hspace{-.1cm} \left(- S^\mathrm{T}\right)^{-1}
\\
&
\quad
\cdot
\big(I - e^{S^\mathrm{T} t}\big)\theta
-
\left( \lambda_0 - \lambda_\infty \right)\left(S^\mathrm{T} + (\beta - \alpha)I\right)^{-1}\hspace{-.1cm}\big(e^{-(\beta - \alpha) t} I - e^{S^\mathrm{T} t}\big)\theta
\Big)^\T .
\end{align*}
\end{proposition}
\begin{proof}
The stated result follows directly from substitution of the expressions in Theorem~\ref{hawkqueuegeneral} into Equation~\ref{autocorphaseeq}.
\end{proof}

\end{document}